\documentclass[11pt,twoside]{article}

\usepackage{amsfonts,amsmath,amssymb,amsthm}
\usepackage{fullpage}
\usepackage{graphics,graphicx}
\usepackage{cases}

\newlength{\widebarargwidth}
\newlength{\widebarargheight}
\newlength{\widebarargdepth}
\DeclareRobustCommand{\widebar}[1]{%
  \settowidth{\widebarargwidth}{\ensuremath{#1}}%
  \settoheight{\widebarargheight}{\ensuremath{#1}}%
  \settodepth{\widebarargdepth}{\ensuremath{#1}}%
  \addtolength{\widebarargwidth}{-0.3\widebarargheight}%
  \addtolength{\widebarargwidth}{-0.3\widebarargdepth}%
  \makebox[0pt][l]{\hspace{0.3\widebarargheight}%
    \hspace{0.3\widebarargdepth}%
    \addtolength{\widebarargheight}{0.3ex}%
    \rule[\widebarargheight]{0.95\widebarargwidth}{0.1ex}}%
  {#1}}

\usepackage[colorlinks]{hyperref}

\newcommand{\GamHat}{\ensuremath{\widehat{\Gamma}}}
\newcommand{\gamhat}{\ensuremath{\widehat{\gamma}}}
\newcommand{\betahat}{\ensuremath{\widehat{\beta}}}
\newcommand{\betastar}{\ensuremath{\beta^*}}
\newcommand{\zhat}{\ensuremath{\widehat{z}}}
\newcommand{\Loss}{\ensuremath{\mathcal{L}}}

\newcommand{\nutil}{\ensuremath{\widetilde{\nu}}}
\newcommand{\ball}{\ensuremath{\mathbb{B}}}

\newcommand{\betamin}{\ensuremath{\betastar_{\min}}}
\newcommand{\betatil}{\ensuremath{\widetilde{\beta}}}

\newcommand{\ztil}{\ensuremath{\widetilde{z}}}

\newcommand{\Lossbar}{\ensuremath{\widebar{\Loss}}}

\newcommand{\Sigmahat}{\ensuremath{\widehat{\Sigma}}}
\newcommand{\Thetahat}{\ensuremath{\widehat{\Theta}}}

\newcommand{\Thetastar}{\ensuremath{\Theta^*}}

\newcommand{\tr}{\ensuremath{\operatorname{trace}}}
\newcommand{\VEC}{\ensuremath{\operatorname{vec}}}
\newcommand{\Sigmastar}{\ensuremath{\Sigma^*}}

\newcommand{\real}{\ensuremath{\mathbb{R}}}
\newcommand{\defn}{\ensuremath{:  =}}
\newcommand{\inprod}[2]{\ensuremath{\langle #1 , \, #2 \rangle}}
\DeclareMathOperator{\sign}{sign}
\newcommand{\supp}{\ensuremath{\operatorname{supp}}}

\newcommand{\order}{{\mathcal{O}}}
\newcommand{\Gamstar}{\ensuremath{\Gamma^*}}
\newcommand{\Deltastar}{\ensuremath{\Delta^*}}
\newcommand{\scriptT}{\ensuremath{\mathcal{T}}}
\newcommand{\Qhat}{\ensuremath{\widehat{Q}}}

\newcommand{\condind}{\ensuremath{\perp\!\!\!\perp}}
\newcommand{\opnorm}[1]{\left|\!\left|\!\left|{#1}\right|\!\right|\!\right|}
\newcommand{\Cov}{\ensuremath{\operatorname{Cov}}}

\newcommand{\E}{\ensuremath{\mathbb{E}}}
\newcommand{\mprob}{\ensuremath{\mathbb{P}}}
\newcommand{\boracle}{\ensuremath{\betahat^{\mathcal{O}}}}

\newcommand{\numobs}{\ensuremath{n}}
\newcommand{\pdim}{\ensuremath{p}}

\newtheorem{lem*}{Lemma}
\newtheorem{thm*}{Theorem}
\newtheorem{alg*}{Algorithm}
\newtheorem{cor*}{Corollary}
\newtheorem{prop*}{Proposition}

\newtheorem{assumption}{Assumption}

\usepackage{color}

\newcommand{\kdim}{\ensuremath{k}}

\newcommand{\CLOW}{c_\ell} \newcommand{\CUP}{c_u}
\newcommand{\CFINAL}{c_3}

\newcommand{\LossSn}{\ensuremath{(\Loss_\numobs)\Big|_S}}
\newcommand{\LossSnbar}{\ensuremath{(\widebar{\Loss}_\numobs)\Big|_S}}
\newcommand{\specvec}{\nu} 
\newcommand{\Term}{\ensuremath{T}}

\makeatletter
\long\def\@makecaption#1#2{
        \vskip 0.8ex
        \setbox\@tempboxa\hbox{\small {\bf #1:} #2}
        \parindent 1.5em  
        \dimen0=\hsize
        \advance\dimen0 by -3em
        \ifdim \wd\@tempboxa >\dimen0
                \hbox to \hsize{
                        \parindent 0em
                        \hfil 
                        \parbox{\dimen0}{\def\baselinestretch{0.96}\small
                                {\bf #1.} #2
                                } 
                        \hfil}
        \else \hbox to \hsize{\hfil \box\@tempboxa \hfil}
        \fi
        }
\makeatother

\begin{document}

\begin{center}
	
	{\bf{\LARGE{Support recovery without incoherence: \\A case for
              nonconvex regularization}}}
	\vspace*{.25in}

	\begin{tabular}{ccc}
	{\large{Po-Ling Loh$^1$}} & \hspace*{1.5in} & {\large{Martin
            J. Wainwright$^2$}}
        \\ {\large{\texttt{loh@wharton.upenn.edu}}} & &
           {\large{\texttt{wainwrig@berkeley.edu}}}
    

	\vspace*{.2in} \\

	Department of Statistics$^1$ & \hspace{1.25in} & Department of
        Statistics$^2$ \\ The Wharton School && Department of EECS
        \\ University of Pennsylvania && UC Berkeley \\ Philadelphia,
        PA 19104 & & Berkeley, CA 94720
	\end{tabular}

\vspace*{.2in}

December 2014 

\vspace*{.2in}

\end{center}

\begin{abstract}
We demonstrate that the primal-dual witness proof method may be used
to establish variable selection consistency and $\ell_\infty$-bounds
for sparse regression problems, even when the loss function and/or
regularizer are nonconvex.  Using this method, we derive two theorems
concerning support recovery and $\ell_\infty$-guarantees for the
regression estimator in a general setting. Our results provide
rigorous theoretical justification for the use of nonconvex
regularization: For certain nonconvex regularizers with vanishing
derivative away from the origin, support recovery consistency may be
guaranteed without requiring the typical incoherence conditions
present in $\ell_1$-based methods.  We then derive several corollaries
that illustrate the wide applicability of our method to analyzing composite
objective functions involving losses such as least squares, nonconvex
modified least squares for errors-in-variables linear regression, the
negative log likelihood for generalized linear models, and the
graphical Lasso. We conclude with empirical studies to corroborate our
theoretical predictions.
\end{abstract}


\section{Introduction}

The last two decades have generated a significant body of work
involving convex relaxations of nonconvex problems arising in
high-dimensional sparse regression (e.g., see the
papers~\cite{DonSta89, Tib96, CheEtal98, CandesTao06, BicEtal08,
  Wai09} and references therein).  In broad terms, the underlying goal
is to identify a relatively \emph{sparse} solution from among a set of
candidates that also yields a good fit to the data. A hard sparsity
constraint is most directly encoded in terms of the $\ell_0$-``norm,"
which counts the number of nonzero entries in a vector. However, this
produces a nonconvex optimization problem that may be NP-hard to solve
or even approximate~\cite{NesNem87, Vav95}. As a result, much work has
focused instead on a slightly different problem, where the
$\ell_0$-based constraint is replaced by the \emph{convex}
$\ell_1$-norm.  There is a relatively well-developed theoretical
understanding of the conditions under which such $\ell_1$-relaxations
produce good estimates of the underlying parameter vector (e.g., see
the papers~\cite{BicEtal08, GeeBuh09, NegRavWaiYu12} and references
therein).

Although the $\ell_1$-norm encourages sparsity in the solution,
however, it differs from the $\ell_0$-norm in a crucial aspect:
Whereas the $\ell_0$-norm is constant for any nonzero argument, the
$\ell_1$-norm increases linearly with the absolute value of the
argument. This linear increase leads to a bias in the resulting
$\ell_1$-regularized solution, and noticeably affects the performance
of the estimator in finite-sample settings~\cite{FanLi01, BreHua11,
  MazEtal11}.  Motivated by this deficiency of
$\ell_1$-regularization, several authors have proposed alternative
forms of nonconvex regularization, including the smoothly clipped
absolute deviation (SCAD) penalty~\cite{FanLi01}, minimax concave
penalty (MCP)~\cite{Zha10}, and log-sum penalty
(LSP)~\cite{CanEtal08}. These regularizers may be viewed as a hybrid
of $\ell_0$- and $\ell_1$-regularizers---they resemble the
$\ell_1$-norm in a neighborhood of the origin, but become
(asymptotically) constant at larger values. Although the nonconvexity
of the regularizer causes the overall optimization problem to be
nonconvex, numerous empirical studies have shown that gradient-based
optimization methods, while only guaranteed to find local optima,
often produce estimators with consistently smaller estimation error
than the estimators produced by the convex
$\ell_1$-penalty~\cite{FanLi01, HunLi05, ZouLi08, BreHua11,
  MazEtal11}.

In recent years, several important advances have been made toward
developing a theoretical framework for nonconvex regularization.
Zhang and Zhang~\cite{ZhaZha12} provide results showing that
\emph{global} optima of nonconvex regularized least squares problems
are statistically consistent for the true regression vector, leaving
open the question of how to find such optima efficiently. Fan et
al.~\cite{FanEtal14} show that one step of a local linear
approximation (LLA) algorithm, initialized at a Lasso solution,
results in a local optimum of the nonconvex regularized least squares
problem that satisfies oracle properties; Wang et al.~\cite{WanEtal13}
establish similar guarantees for the output of a particular
path-following algorithm.  Our own past work~\cite{LohWai13} provides
a general set of sufficient conditions under which all stationary
points of the nonconvex regularized problem are guaranteed to lie
within statistical precision of the true parameter, which
substantially simplifies the optimization problem to one of finding
stationary points.  Our work also establishes bounds on the
$\ell_2$-norm and prediction error that agree with the well-known
bounds for the convex $\ell_1$-regularizer, up to constant factors.

Despite these advances, however, an important question has remained:
Are stationary points of such nonconvex problems also consistent for
variable selection? In other words, does the support set of a
stationary point agree with the support of the true regression vector,
with high probability, and at what rate does the error probability
tend to zero?  In addition to providing a natural venue for
establishing bounds on the $\ell_\infty$-error, support recovery
results furnish a much better understanding of when stationary points
of nonconvex objectives are actually unique.  For convex objectives,
various standard proof techniques for variable selection consistency
now exist, including approaches via Karush-Kuhn-Tucker optimality
conditions and primal-dual witness arguments
(e.g.,~\cite{Zhao06,Wai09, LeeEtal14}).  However, these arguments, as
previously stated, have relied crucially on convexity of both the loss
and regularizer.

The first main contribution of our paper is to show how the
primal-dual witness technique may be modified and extended to a
certain class of \emph{nonconvex} problems.  Our proof hinges on the
notion of generalized gradients from nonsmooth analysis~\cite{Cla75},
and optimization-theoretic results on norm-regularized, smooth, but
possibly nonconvex functions~\cite{FleWat80}. Our main result is to
establish sufficient conditions for variable selection consistency
when both the loss and regularizer are allowed to be nonconvex,
provided the loss function satisfies a form of restricted strong
convexity and the regularizer satisfies suitable mild
conditions. These assumptions are similar to the conditions required
in our earlier work on $\ell_1$- and
$\ell_2$-consistency~\cite{LohWai13}, with an additional assumption on the minimum signal strength that allows us to derive stronger support recovery guarantees. Remarkably, our results
demonstrate that for a certain class of regularizers---including the
SCAD and MCP regularizers---we may dispense with the usual incoherence
conditions required by $\ell_1$-based methods, and still guarantee
support recovery consistency for all stationary points of the
resulting nonconvex program.  This provides a strong theoretical reason for why certain nonconvex
regularizers might be favored over their convex counterparts. In
addition, we establish that for the same class of nonconvex regularizers, the unique
stationary point is in fact equal to the oracle solution.

Several other authors have mentioned the potential for nonconvex
regularizers to deliver estimation and support recovery guarantees
under weaker assumptions than the $\ell_1$-norm. The same line of work
introducing nonconvex penalties such as the SCAD and MCP and
developing subsequent theory~\cite{FanLi01, FanPen04, ZouLi08, Zha12,
  FanEtal14} demonstrates that in the absence of incoherence conditions, nonconvex regularized
problems possess local optima that are statistically consistent and
satisfy an oracle property. Since nonconvex programs may
have multiple solutions, however, these papers have focused on
establishing theoretical guarantees for the output of \emph{specific} optimization algorithms. More
recently, Wang et al.~\cite{WanEtal13} propose a path-following
homotopy algorithm for obtaining solutions to nonconvex regularized
$M$-estimators, and show that iterates of the homotopy algorithm
converge at a linear rate to the oracle solution of the $M$-estimation
problem.  In contrast to theory of this type---applicable to a
particular algorithm---the theory in our paper is purely statistical
and does \emph{not} concern iterates of a particular optimization
algorithm. Indeed, the novelty of our theoretical results is that they
establish support recovery consistency for \emph{all} stationary
points and, moreover, shed light on situations where such stationary
points are actually unique. Finally, Pan and Zhang~\cite{PanZha15}
provide results showing that under restricted eigenvalue assumptions
on the design matrix that are weaker than the standard restricted
eigenvalue conditions, a certain class of nonconvex regularizers yield
estimates that are consistent in $\ell_2$-norm. They provide
bounds on the sparsity of approximate global and approximate sparse
(AGAS) solutions, a notion also studied in earlier
work~\cite{ZhaZha12}.  However, their theoretical development stops
short of providing conditions for recovering the exact support of the
underlying regression vector.

The remainder of our paper is organized as follows: In
Section~\ref{SecBackground}, we provide basic background material on
regularized $M$-estimators and set up notation for the paper.  We also
outline the primal-dual witness proof method.
Section~\ref{SecMain} is devoted to the statements of our main results
concerning support recovery and $\ell_\infty$-bounds, followed by
corollaries that specialize our results to linear regression,
generalized linear models, and the graphical Lasso. In each case, we
contrast our conditions for nonconvex regularizers to those required
by convex regularizers and discuss the implications of our
significantly weaker assumptions. We provide proofs of our main
results in Appendices~\ref{SecSuppRecovery} and~\ref{SecThmEllInf}, with supporting results and more
technical lemmas contained in later appendices. Finally,
Section~\ref{SecSims} contains convergence guarantees for the
composite gradient descent algorithm and a variety of illustrative
simulations that confirm our theoretical results.

\paragraph{Notation:}
For functions $f(n)$ and $g(n)$, we write $f(n) \precsim g(n)$ to mean
that \mbox{$f(n) \le c g(n)$} for some universal constant $c \in (0,
\infty)$, and similarly, $f(n) \succsim g(n)$ when \mbox{$f(n) \ge c'
  g(n)$} for some universal constant $c' \in (0, \infty)$. We write
$f(n) \asymp g(n)$ when $f(n) \precsim g(n)$ and $f(n) \succsim g(n)$
hold simultaneously. For a vector $v \in \real^p$ and a subset $S
\subseteq \{1, \dots, p\}$, we write $v_S \in \real^S$ to denote the
vector $v$ restricted to $S$. For a matrix $M$, we write
$\opnorm{M}_2$ and $\opnorm{M}_F$ to denote the spectral and Frobenius
norms, respectively, and write $\opnorm{M}_{\max} \defn \max_{i,j}
|m_{ij}|$ to denote the elementwise $\ell_\infty$-norm of $M$. For a
function $h: \real^p \rightarrow \real$, we write $\nabla h$ to denote
a gradient or subgradient, if it exists. Finally, for $q, r > 0$, we
write $\ball_q(r)$ to denote the $\ell_q$-ball of radius $r$ centered
around 0.


\section{Problem formulation}
\label{SecBackground}

In this section, we briefly review the theory of regularized
$M$-estimators. We also outline the primal-dual witness proof
technique that underlies our proofs of variable consistency for
nonconvex problems.


\subsection{Regularized $M$-estimators}

The analysis of this paper applies to regularized $M$-estimators of
the form
\begin{equation}
\label{EqnMEst}
\betahat \in \arg\min_{\|\beta\|_1 \le R, \; \beta \in \Omega} \left
\{\Loss_\numobs(\beta) + \rho_\lambda(\beta)\right\},
\end{equation}
where $\Loss_\numobs$ denotes the empirical loss function and
$\rho_\lambda$ denotes the penalty function, both assumed to be
continuous.  In our framework, both of these functions are allowed to
be nonconvex, but the theory applies a fortiori when only one function
is convex.  The prototypical example of a loss function to keep in
mind is the least squares objective, $\Loss_\numobs(\beta) =
\frac{1}{2 \numobs} \|y - X \beta\|_2^2$.  We include the side
constraint, $\|\beta\|_1 \le R$, in order to ensure that a global
minimum $\betahat$ exists.\footnote{In the sequel, we will give examples of
  nonconvex loss functions for which the global minimum fails to exist
  without such a side constraint (cf.\ Section~\ref{SecLoss} below).}
For modeling purposes, we have also allowed for an additional
constraint, $\beta \in \Omega$, where $\Omega$ is an open convex set;
note that we may take $\Omega = \real^\pdim$ when this extra
constraint is not needed.

The analysis of this paper is restricted to the class of
\emph{coordinate-separable regularizers}, meaning that $\rho_\lambda$
is expressible as the sum
\begin{equation}
\label{EqnRhoDecomp}
\rho_\lambda(\beta) = \sum_{j=1}^\pdim \rho_\lambda(\beta_j).
\end{equation}
Here, we have engaged in some minor abuse of notation; the
functions \mbox{$\rho_\lambda: \real \rightarrow \real$} appearing on
the right-hand side of equation~\eqref{EqnRhoDecomp} are univariate
functions acting upon each coordinate.  Our results are readily
extended to the inhomogenous case, where different coordinates have
different regularizers $\rho^j_\lambda$, but we restrict ourselves to
the homogeneous case in order to simplify our discussion.

From a statistical perspective, the purpose of solving the
optimization problem~\eqref{EqnMEst} is to estimate the vector
$\betastar \in \real^\pdim$ that minimizes the expected loss function:
\begin{align}
\label{EqnBetaPop}
\betastar & \defn \arg \min_{\beta \in \Omega}
\E[\Loss_\numobs(\beta)],
\end{align}
where we assume that $\betastar$ is unique and independent of the sample
size. Our goal is to develop sufficient conditions under which a minimizer
$\betahat$ of the composite objective~\eqref{EqnMEst} is a
consistent estimator for $\betastar$. Consequently, we will always
choose $R \ge \|\betastar\|_1$, which ensures that $\betastar$ is a
feasible point.


\subsection{Class of regularizers}
\label{SecPenalties}

In this paper, we study the class of regularizers $\rho_\lambda: \real
\rightarrow \real$ that are amenable in the following sense.

\paragraph{Amenable regularizers:}  For a parameter $\mu \geq 0$, we say 
that $\rho_\lambda$ is $\mu$-amenable if:
\begin{enumerate}
\item[(i)] The function $t \mapsto \rho_\lambda(t)$ is symmetric
  around zero (i.e., $\rho_\lambda(t) = \rho_\lambda(-t)$ for all
  $t$), and $\rho_\lambda(0) = 0$.
\item[(ii)] The function $t \mapsto \rho_\lambda(t)$ is nondecreasing
  on $\real^+$.
\item[(iii)] The function $t \mapsto \frac{\rho_\lambda(t)}{t}$ is
  nonincreasing on $\real^+$.
\item[(iv)] The function $t \mapsto \rho_\lambda(t)$ is differentiable,
  for $t \neq 0$.
\item[(v)] The function $t \mapsto \rho_\lambda(t) + \frac{\mu}{2}
  t^2$ is convex, for some $\mu > 0$.
\item[(vi)] $\lim \limits_{t \rightarrow 0^+}\rho_\lambda'(t) =
  \lambda$.
\end{enumerate}
We say that $\rho_\lambda$ is $(\mu, \gamma)$-amenable if, in addition:
\begin{enumerate}
\item[(vii)] There is some scalar $\gamma \in (0, \infty)$ such that
  $\rho_\lambda'(t) = 0$, for all $t \ge \gamma \lambda$.
\end{enumerate}
Conditions (vi) and (vii) are also known as the \emph{selection} and
\emph{unbiasedness} properties, respectively, and the MCP
regularizer~\cite{Zha12} described below minimizes the maximum
concavity of $\rho$ subject to (vi)--(vii). Note that the usual
$\ell_1$-penalty $\rho_\lambda(t) = \lambda |t|$ is $0$-amenable, but
it is \emph{not} $(0, \gamma)$-amenable, for any $\gamma < \infty$.
The notion of $\mu$-amenability was also used in our past work on
$\ell_2$-bounds for nonconvex regularizers~\cite{LohWai13}, with the
exception of the selection property (vi).  Since the goal of the
current paper is to obtain \emph{stronger} conclusions, in terms of
variable selection and $\ell_\infty$-bounds, we will also require
$\rho_\lambda$ to satisfy the selection and unbiasedness properties.

Note that if we define $q_\lambda(t) \defn \lambda |t| -
\rho_\lambda(t)$, the conditions (iv) and (vi) together imply that
$q_\lambda$ is everywhere differentiable. Furthermore, if
$\rho_\lambda$ is $(\mu, \gamma)$-amenable, we have $q'_\lambda(t) =
\lambda \cdot \sign(t)$, \mbox{for all $|t| \geq \gamma \lambda$.} In
Appendix~\ref{AppAmenable}, we provide some other useful results concerning amenable regularizers. \\

Many popular regularizers are either $\mu$-amenable or $(\mu,
\gamma)$-amenable.  Let us consider a few examples to illustrate.

\paragraph{Smoothly clipped absolute deviation (SCAD) penalty:}  This 
penalty, due to Fan and Li~\cite{FanLi01}, takes the form
\begin{align}
\label{EqnSCADdefn}
\rho_\lambda(t) & \defn \begin{cases} \lambda |t|, & \mbox{for $|t|
    \le \lambda$,} \\
-\frac{t^2 - 2a\lambda |t| + \lambda^2}{2(a-1)}, & \mbox{for $\lambda <
  |t| \le a \lambda$,} \\
\frac{(a+1)\lambda^2}{2}, & \mbox{for $|t| > a\lambda$},
\end{cases}
\end{align}
where $a > 2$ is a fixed parameter. A straightforward calculation show
that the SCAD penalty is $(\mu, \gamma)$-amenable, with $\mu =
\frac{1}{a-1}$ and $\gamma = a$. \\

\paragraph{Minimax concave penalty (MCP):} This penalty, due to 
Zhang~\cite{Zha12}, takes the form
\begin{align}
\label{EqnMCPdefn}
\rho_\lambda(t) & \defn \sign(t) \, \lambda \cdot \int_0^{|t|} \left(1
- \frac{z}{\lambda b}\right)_+ dz,
\end{align}
where $b > 0$ is a fixed parameter. The MCP regularizer is $(\mu,
\gamma)$-amenable, with $\mu = \frac{1}{b}$ and $\gamma =
b$. \\

\noindent Finally, let us consider some examples of penalties that are
$\mu$-amenable, but \emph{not} $(\mu, \gamma)$-amenable, for any
$\gamma < \infty$.

\paragraph{Standard Lasso penalty:}  As mentioned previously, the
$\ell_1$-penalty $\rho_\lambda(t) = \lambda |t|$ is $0$-amenable, but not
$(0, \gamma)$-amenable, for any $\gamma < \infty$.

\paragraph{Log-sum penalty (LSP):} This penalty, studied in 
past work~\cite{CanEtal08}, takes the form
\begin{equation}
\label{EqnLSPdefn}
\rho_\lambda(t) = \log(1 + \lambda |t|).
\end{equation}
For $t > 0$, we have $\rho_\lambda'(t) = \frac{\lambda}{1+\lambda t}$, and
$\rho_\lambda''(t) = \frac{-\lambda^2}{(1+\lambda t)^2}$.  In
particular, the LSP regularizer is $\lambda^2$-amenable, but not $(\lambda^2, \gamma)$-amenable, for any $\gamma <
\infty$.


\subsection{Nonconvexity and restricted strong convexity}
\label{SecLoss}

Next, we consider some examples of the types of nonconvex loss
functions treated in this paper.  At a high level, we consider loss
functions that are differentiable and satisfy a form of
\emph{restricted strong convexity}, as used in large body of past work
on analysis of high-dimensional sparse $M$-estimators
(e.g.,~\cite{AgaEtal12, BicEtal08,GeeBuh09,LohWai13, NegRavWaiYu12}).
In order to provide intuition before stating the formal definition, note that
for any convex and differentiable function $f: \real^\pdim \rightarrow
\real$ that is globally convex and locally strongly convex around a
point $\beta \in \real^\pdim$, there exists a constant $\alpha > 0$ such
that
\begin{align}
\label{EqnChocoPie}
 \inprod{\nabla f(\beta + \Delta) - \nabla f(\beta)}{\Delta} & \geq
 \alpha \cdot \min \{ \|\Delta\|_2, \|\Delta\|_2^2 \},
\end{align}
for all $\Delta \in \real^\pdim$.  The notion of restricted strong
convexity (with respect to the $\ell_1$-norm) weakens this requirement
by adding a tolerance term that penalizes non-sparse vectors.  In
particular, for positive constants $\{(\alpha_j, \tau_j)\}_{j=1}^2$, we have the
following definition:
\paragraph{Restricted strong convexity:}
Given any pair of vectors $\beta, \Delta \in \real^p$, the loss
function $\Loss_\numobs$ satisfies an $(\alpha, \tau)$-RSC condition,
if:
\begin{subnumcases}{
\label{EqnRSC}
 \inprod{\nabla \Loss_\numobs(\beta + \Delta) - \nabla
   \Loss_\numobs(\beta)}{\Delta} \geq}
\label{EqnLocalRSC}
\alpha_1 \|\Delta\|_2^2 - \tau_1 \frac{\log p}{n} \|\Delta\|_1^2,
\quad \quad \mbox{for all $\|\Delta\|_2 \leq 1$,} \\
\label{EqnL2Bd}
\alpha_2 \|\Delta\|_2 - \tau_2 \sqrt{\frac{\log p}{n}} \|\Delta\|_1,
\quad \mbox{for all $\|\Delta\|_2 \geq 1$,}
\end{subnumcases}
where $(\alpha_1, \alpha_2)$ are strictly positive constants, and
$(\tau_1, \tau_2)$ are nonnegative constants. \\

As noted in inequality~\eqref{EqnChocoPie}, any locally strongly convex function that is also globally convex satisfies the RSC
condition with tolerance parameters $\tau_1 = \tau_2 = 0$.  For
$\tau_1, \tau_2 > 0$, the RSC
condition imposes strong curvature only in certain directions of
$p$-dimensional space---namely, those nonzero directions $\Delta \in
\real^\pdim$ for which the ratio
$\frac{\|\Delta\|_1}{\|\Delta\|_2}$ is relatively small; i.e., less than a constant multiple of
$\sqrt{\frac{\numobs}{\log \pdim}}$.  Note that for any $\kdim$-sparse
vector $\Delta$, we have $\frac{\|\Delta\|_1}{\|\Delta\|_2} \leq
\sqrt{\kdim}$, so that the RSC definition guarantees a form of strong
convexity for all $\kdim$-sparse vectors when $\numobs \succsim
\kdim \log \pdim$. \\

A line of past work
(e.g.,~\cite{RasEtAl10,RudZho13,NegRavWaiYu12,LohWai13}) shows that
the RSC condition holds, with high probability, for many types of convex
and nonconvex objectives arising in statistical estimation problems. We now consider a few illustrative examples.

\paragraph{Standard linear regression:}  
Consider the standard linear regression model, in which we observe
i.i.d.\ pairs $(x_i, y_i) \in \real^\pdim \times \real$, linked by the
linear model
\begin{align*}
y_i = x_i^T \betastar + \epsilon_i, \qquad \mbox{for $i = 1, \ldots,
  \numobs$,}
\end{align*}
and the goal is to estimate $\betastar \in \real^p$.  A standard loss
function in this case is the least squares function
$\Loss_\numobs(\beta) = \frac{1}{2 \numobs} \|y - X \beta\|_2^2$,
where $y \in \real^\numobs$ is the vector of responses and $X \in
\real^{\numobs \times \pdim}$ is the design matrix with $x_i \in
\real^\pdim$ as its $i^{th}$ row.  In this special case, for any
$\beta, \Delta \in \real^\pdim$, we have
\begin{align*}
\inprod{\nabla \Loss_\numobs(\beta + \Delta) - \nabla
  \Loss_\numobs(\beta)}{\Delta} & = \Delta^T \Big(\frac{X^T
  X}{\numobs} \Big) \Delta \, = \; \frac{\|X \Delta\|_2^2}{\numobs}.
\end{align*}
Consequently, for the least squares loss function, the RSC condition
is essentially equivalent to lower-bounding sparse restricted
eigenvalues~\cite{GeeBuh09, BicEtal08}.


\paragraph{Linear regression with errors in covariates:}

Let us now turn to a simple extension of the standard linear
regression model.  Suppose that instead of observing the covariates
$x_i \in \real^\pdim$ directly, we observe the corrupted vectors $z_i
= x_i + w_i$, where $w_i \in \real^\pdim$ is some type of noise
vector.  This setup is a particular instantiation of a more general
errors-in-variables model for linear regression.  The standard Lasso
estimate (applied to the observed pairs $\{(z_i,
y_i)\}_{i=1}^\numobs$) is inconsistent in this setting.

As studied previously~\cite{LohWai11a}, it is natural
to consider a corrected version of the Lasso, which we state in terms
of the quadratic objective,
\begin{equation}
\label{EqnLossLinear}
\Loss_\numobs(\beta) = \frac{1}{2} \beta^T \GamHat \beta - \gamhat^T
\beta.
\end{equation}
Our past work~\cite{LohWai11a} shows that as long as $(\GamHat,
\gamhat)$ are unbiased estimates of $(\Sigma_x, \Sigma_x \betastar)$,
any global minimizer $\betahat$ of the appropriately regularized problem~\eqref{EqnMEst} is a consistent estimate for $\betastar$.  In the
additive corruption model described in the previous paragraph, a natural choice for
the pair $(\GamHat, \gamhat)$ is given by
\begin{equation*}
(\GamHat, \gamhat) = \left(\frac{Z^TZ}{n} - \Sigma_w, \; \frac{Z^T y}{n}\right),
\end{equation*}
where the covariance matrix $\Sigma_w = \Cov(w_i)$ is assumed to be
known.  However, in the high-dimensional setting ($\numobs \ll
\pdim$), the random matrix $\GamHat$ is \emph{not} positive semidefinite,
so the quadratic objective function~\eqref{EqnLossLinear} is
nonconvex. (This is also a concrete instance where the
objective function~\eqref{EqnMEst} requires the constraint
$\|\beta\|_1 \le R$ in order to be bounded below.)  Nonetheless, our
past work~\cite{LohWai11a,LohWai13} shows that under certain tail
conditions on the covariates and noise vectors, the loss
function~\eqref{EqnLossLinear} does satisfy a form of restricted
strong convexity.


\paragraph{Generalized linear models:}

Moving beyond standard linear regression, suppose the pairs
$\{(x_i, y_i)\}_{i=1}^n$ are drawn from a generalized linear model
(GLM). Recall that the conditional distribution for a GLM takes the
form
\begin{equation}
\label{EqnCondGLM}
\mprob(y_i \mid x_i, \beta, \sigma) = \exp\left(\frac{y_i x_i^T \beta
  - \psi(x_i^T \beta)}{c(\sigma)}\right),
\end{equation}
where $\sigma > 0$ is a scale parameter and $\psi$ is the cumulant
function. The loss function corresponding to the negative log
likelihood is given by
\begin{equation}
\label{EqnLossGLM}
\Loss_\numobs(\beta) = \frac{1}{n} \sum_{i=1}^n \left(\psi(x_i^T
\beta) - y_i x_i^T \beta \right),
\end{equation}
and it is easy to see that equation~\eqref{EqnLossGLM} reduces to
equation~\eqref{EqnLossLinear} with the choice $\psi(t) =
\frac{t^2}{2}$. Using properties of exponential families, we may also
check that equation~\eqref{EqnBetaPop} holds.  Negahban et
al.~\cite{NegRavWaiYu12} show that a form of restricted strong
convexity holds for a broad class of generalized linear models.

\paragraph{Graphical Lasso:}
Now suppose that we observe a sequence $\{x_i\}_{i=1}^n \subseteq
\real^p$ of $\pdim$-dimensional random vectors with mean 0.
Our goal is to estimate the inverse covariance matrix $\Thetastar
\defn \Sigma_x^{-1}$, assumed to be relatively sparse.  In the
Gaussian case, sparse inverse covariance matrices arise from imposing
a Markovian structure on the random vector~\cite{Speed86}.  Letting
$\Sigmahat \defn \frac{X^TX}{n}$ denote the sample covariance matrix,
consider the loss function
\begin{equation}
\label{EqnGlassoLn}
\Loss_\numobs(\Theta) = \tr(\Sigmahat \Theta) - \log \det(\Theta),
\end{equation}
which we refer to as the graphical Lasso loss. Taking derivatives, it
is easy to check that \mbox{$\Thetastar = \arg\min_{\Theta}
  \E[\Loss_\numobs(\Theta)]$,} which verifies the population-level
condition~\eqref{EqnBetaPop}.  As we show in Section~\ref{SecGLasso},
the graphical Lasso loss is locally strongly convex, so it
satisfies the restricted strong convexity condition with $\tau_1 =
\tau_2 = 0$.


\subsection{Primal-dual witness proof technique}
\label{SecPDW}

We now outline the main steps of the primal-dual witness (PDW) proof
technique, which we will use to establish support recovery and
$\ell_\infty$-bounds for the program~\eqref{EqnMEst}. Such a technique
was previously applied only in situations where $\Loss_\numobs$ is
convex and $\rho_\lambda$ is the $\ell_1$-penalty, but we show that
this machinery may be extended via a careful analysis of local optima
of norm-regularized functions based on generalized gradients.

As stated in Theorem~\ref{ThmSuppRecovery} below, the success of the PDW construction guarantees that stationary points of the nonconvex objective are consistent for variable selection consistency---in fact, they are unique. Recall that $\betatil \in \real^p$ is a \emph{stationary point} of the optimization program~\eqref{EqnMEst} if we have \mbox{$\inprod{\nabla
    \Loss_\numobs(\betatil) + \nabla \rho_\lambda(\betatil)}{\beta -
    \betatil} \ge 0$}, for all $\beta$ in the feasible
region~\cite{Ber99}. Due to the possible nondifferentiability of
$\rho_\lambda$ at 0, we abuse notation slightly and denote
\begin{equation*}
\inprod{\nabla \rho_\lambda(\betatil)}{\beta - \betatil} = \lim_{t
  \rightarrow 0^+} \inprod{\nabla \rho_\lambda(\betatil + t(\beta -
  \betatil))}{\beta - \betatil}
\end{equation*}
(see, e.g., Clarke~\cite{Cla75} for a more comprehensive treatment of
such generalized gradients). The set of stationary points includes all
local/global minima of the program~\eqref{EqnMEst}, as well as any
interior local maxima.

The key steps of the PDW argument are as follows:
\begin{enumerate}
\item[(i)] Optimize the \emph{restricted program}
\begin{equation}
\label{EqnRestrictedM}
\betahat_S \in \arg\min_{\beta \in \real^S: \; \|\beta\|_1 \le R, \;
  \beta \in \Omega} \left\{\Loss_\numobs(\beta) +
\rho_\lambda(\beta)\right\},
\end{equation}
where we enforce the additional constraint that $\supp(\betahat_S)
\subseteq \supp(\betastar) \defn S$. Establish that $\|\betahat_S\|_1
< R$; i.e., $\betahat_S$ is an interior point of the feasible set.
\item[(ii)] Define $\zhat_S \in \partial \|\betahat_S\|_1$, and choose
  $\zhat_{S^c}$ to satisfy the zero-subgradient condition
\begin{equation}
\label{EqnZeroSub}
\nabla \Loss_\numobs(\betahat) - \nabla q_\lambda(\betahat) + \lambda \zhat
= 0,
\end{equation}
where $\zhat = (\zhat_S, \zhat_{S^c})$, $\betahat \defn (\betahat_S,
0_{S^c})$, and $q_\lambda(t) \defn \lambda |t| - \rho_\lambda(t)$.
Establish \emph{strict dual feasibility} of $\zhat_{S^c}$; i.e.,
$\|\zhat_{S^c}\|_\infty < 1$.
\item[(iii)] Show that $\betahat$ is a \emph{local minimum} of the
  full program~\eqref{EqnMEst}, and moreover, \emph{all} stationary
  points of the program~\eqref{EqnMEst} are supported on $S$.
\end{enumerate}
Note that the output $(\betahat, \zhat)$ of the PDW construction
depends implicitly on the choice of $\lambda$ and $R$.

Under the restricted strong convexity condition, the restricted
problem~\eqref{EqnRestrictedM} minimized in step (i) is actually a
convex program. Hence, if $\|\betahat_S\|_1 < R$, the zero-subgradient
condition~\eqref{EqnZeroSub} must hold at $\betahat_S$ for the
restricted problem~\eqref{EqnRestrictedM}. Note that when
$\Loss_\numobs$ is convex and $\rho_\lambda$ is the $\ell_1$-penalty
as in the conventional setting, the additional $\ell_1$-constraint in
the programs~\eqref{EqnMEst} and~\eqref{EqnRestrictedM} is omitted. If
also $\Omega = \real^p$, the vector $\betahat_S$ is automatically a
zero-subgradient point if it is a global minimum of the restricted
program~\eqref{EqnRestrictedM}, which greatly simplifies the
analysis. Our refined analysis shows that under suitable restrictions,
global optimality still holds for $\betahat_S$ and $\betahat$, and the
convexity of the restricted program therefore implies uniqueness.

In the sections to follow, we show how the primal-dual witness
technique may be used to establish support recovery results for
general nonconvex regularized $M$-estimators, and then derive
sufficient conditions under which stationary points of the
program~\eqref{EqnMEst} are in fact unique.


\section{Main statistical results and consequences}
\label{SecMain}

We begin by stating our main theorems concerning support recovery and
$\ell_\infty$-bounds, and then specialize our analysis to particular settings of
interest.

\subsection{Main results}
\label{SecStat}

Our main statistical results concern stationary points of the
regularized $M$-estimator~\eqref{EqnMEst}, where the loss function
satisfies the RSC condition~\eqref{EqnRSC} with parameters
$\{(\alpha_j, \tau_j)\}_{j=1}^2$, and the regularizer is
$\mu$-amenable with $\mu \in [0, \alpha_1)$.  Our first theorem
  concerns the success of the PDW construction described in
  Section~\ref{SecPDW}. The theorem guarantees support recovery
  provided two conditions are met, the first involving an appropriate
  choice of $\lambda$ and $R$, and the second involving strict dual
  feasibility of the dual vector $\zhat$. Note that it is through
  validating the second condition that the incoherence assumption
  arises in the usual $\ell_1$-analysis, but we demonstrate in our
  corollaries to follow that strict dual feasibility may be guaranteed
  under \emph{weaker} conditions when a $(\mu, \gamma)$-amenable
  regularizer is used, instead. (See Appendix~\ref{AppDual} for a technical
  discussion.)  The proof of Theorem~\ref{ThmSuppRecovery} is
  contained in Appendix~\ref{SecSuppRecovery}.

\begin{thm*}[PDW construction for nonconvex functions]
\label{ThmSuppRecovery}
Suppose $\Loss_\numobs$ is an \mbox{$(\alpha, \tau)$-RSC} function and $\rho_\lambda$ is $\mu$-amenable, for some $\mu \in [0,
  \alpha_1)$. Suppose that:
\begin{itemize}
\item[(a)] The parameters $(\lambda, R)$ satisfy the bounds
\begin{subequations}
\begin{align}
\label{EqnLambdaCond}
2 \cdot \max \left\{\|\nabla \Loss_\numobs(\betastar)\|_\infty, \; \alpha_2
\sqrt{\frac{\log k}{n}}\right\} & \le \lambda \le
\sqrt{\frac{(2\alpha_1 - \mu) \alpha_2}{56k}}, \qquad \text{and} \\
\label{EqnRCond}
\max\left\{2 \; \|\betastar\|_1, \; \frac{60 k\lambda}{2\alpha_1 -
  \mu}\right\} & \le R \le \min\left\{\frac{\alpha_2}{8 \lambda}, \;
\frac{\alpha_2}{\tau_2} \sqrt{\frac{n}{\log p}}\right\}.
\end{align}
\end{subequations}
\item[(b)] For some $\delta \in \left[\frac{4R\tau_1 \log
    p}{n\lambda}, \; 1\right]$, the dual vector $\zhat$ from the PDW
  construction satisfies the strict dual feasibility condition
\begin{equation}
\label{EqnStrictDual}
\|\zhat_{S^c}\|_\infty \leq 1 - \delta.
\end{equation}
\end{itemize}
Then for any $k$-sparse vector $\betastar$, the
program~\eqref{EqnMEst} with a sample size $n \ge
\frac{2\tau_1}{2 \alpha_1 - \mu} k \log p$ has a unique stationary
point, given by the primal output $\betahat$ of the PDW construction.
\end{thm*}

\paragraph{Remark:} Of course, Theorem~\ref{ThmSuppRecovery} is 
vacuous unless proper choices of $\lambda, R$, and $\delta$ exist. In
the corollaries to follow, we show that $\|\nabla
\Loss_\numobs(\betastar)\|_\infty \le c \sqrt{\frac{\log p}{n}}$, with
high probability, in many settings of interest. In particular, we may
choose $\lambda \asymp \sqrt{\frac{\log p}{n}}$ to satisfy
inequality~\eqref{EqnLambdaCond} when the sample size satisfies $n
\succsim k \log p$. Note that $R \asymp \frac{1}{\lambda}$ then causes
inequality~\eqref{EqnRCond} to be satisfied under the same sample size
scaling. Indeed, if the RSC parameters were known, we could simply
take $R = \min\left\{\frac{\alpha_2}{8 \lambda}, \;
\frac{\alpha_2}{\tau_2} \sqrt{\frac{n}{\log p}}\right\}$ and then
focus on tuning $\lambda$. Finally, note that the inequality
$\frac{4R\tau_1 \log p}{n\lambda} \le 1$ is satisfied as long as $R
\le \frac{n\lambda}{4\tau_1 \log p}$, which is guaranteed by the
preceding choice of $(\lambda, R)$ and the scaling $n \succsim k \log
p$.  In this way, the existence of an appropriate choice of $\delta$
is guaranteed.\footnote{An important observation is that the parameter
  $\delta$ does not actually appear in the statistical estimation
  procedure and is simply a byproduct of the PDW analysis. Hence, it
  is not necessary to know or estimate a valid value of $\delta$.} \\

Note also that our results require the assumption $\mu < 2 \alpha_1$,
where a smaller gap of $(2 \alpha_1 - \mu)$ translates into a larger
sample size requirement. This consideration may motivate an advantage
of using the LSP regularizer over a regularizer such as SCAD or MCP;
as discussed in Section~\ref{SecPenalties}, the SCAD and MCP
regularizers have $\mu$ equal to a constant value, whereas $\mu =
\lambda^2 \rightarrow 0$ for the LSP. On the other hand, the LSP is
\emph{not} $(\mu, \gamma)$-amenable, which as discussed later, allows
us to remove the incoherence condition for SCAD and MCP when
establishing strict dual feasibility~\eqref{EqnStrictDual}. Indeed,
the MCP is designed so that $\mu$ is minimal subject to unbiasedness
and selection of the regularizer~\cite{Zha10}. This suggests that for
more incoherent designs, the LSP may be preferred for variable
selection, whereas for less incoherent designs, SCAD or MCP may be
better. In the simulations of Section~\ref{SecSims}, however, the LSP
regularizer only performs negligibly better than the $\ell_1$-penalty
in situations where the incoherence condition holds and the same
regularization parameter $\lambda$ is chosen.

Finally, we note that although the conditions of
Theorem~\ref{ThmSuppRecovery} are already relatively mild, they are
nonetheless \emph{sufficient} conditions. Indeed, as confirmed
experimentally, there are many situations where the condition $\mu <
2 \alpha_1$ does \emph{not} hold, yet the stationary points of the
program~\eqref{EqnMEst} still appear to be supported on $S$ and/or
unique. Two feasible explanations for this phenomenon are the
following: First, it is possible that in cases where $2 \alpha_1 \le
\mu$, the composite objective function is not convex over the entire
feasible set, yet $\betahat$ is still sufficiently close to
$\betastar$ and the positive definite condition
\begin{equation}
\label{EqnPSDLoss}
\left(\nabla^2 \Loss_\numobs(\betahat) + \nabla^2
\rho_\lambda(\betahat)\right)_{SS} \succ 0
\end{equation}
holds at the point $\betahat$. Examining the proof of
Theorem~\ref{ThmSuppRecovery}, we may see that
equation~\eqref{EqnPSDLoss}, together with strict dual feasibility, is
sufficient for establishing that $\betahat$ is a local minimum of the
program~\eqref{EqnMEst}; Lemma~\ref{LemStationBd} in
Appendix~\ref{SecSuppRecovery} then implies that all stationary points
are supported on $S$. Nonetheless, since the restricted
program~\eqref{EqnRestrictedM} is no longer guaranteed to be convex,
multiple stationary points may exist. Second, we note that although
strict convexity of the objective and a zero-subgradient
condition are certainly sufficient conditions to guarantee a unique
global minimum, a function may have diverse regions of
convexity/concavity and still possess a unique global minimum. In such cases, we
may still be lucky in simulation studies and obtain the global optimum. \\

Our second general theorem provides control on the $\ell_\infty$-error
between any stationary point and $\betastar$, and shows that the
local/global optimum of a nonconvex regularized program agrees with
the oracle result when the regularizer is $(\mu, \gamma)$-amenable. We
define the \emph{oracle estimator} according to
\begin{equation*}
\boracle_S \defn \arg\min_{\beta_S \in \real^S} \left\{\Loss_\numobs
\big(\beta_S, 0_{S^c}))\right\},
\end{equation*}
and write $\boracle \defn (\boracle_S, 0_{S^c})$. In other words, the
oracle estimator is the unpenalized estimator obtained from minimizing
$\Loss_\numobs$ over the true support set $S$.  As shown in the proof, under
the assumed RSC conditions, the restricted function $\Loss_\numobs \mid_S$
is strictly convex and $\boracle_S$ is uniquely defined.  With
this notation, we have the following result:
\begin{thm*}
\label{ThmEllInf}
Suppose the assumptions of Theorem~\ref{ThmSuppRecovery} are satisfied
and strict dual feasibility~\eqref{EqnStrictDual} holds. Then the
unique stationary point $\betahat$ of the program~\eqref{EqnMEst}
has the following properties:
\begin{enumerate}
\item[(a)] Let $\Qhat \defn \int_0^1 \nabla^2 \Loss_n \left(\betastar + t(\betahat - \betastar)\right)  dt$. Then
\begin{equation}
\label{EqnHoneydew}
\|\betahat - \betastar\|_\infty \le \left\|\left(\Qhat_{SS}\right)^{-1} \nabla
\Loss_\numobs(\betastar)_S\right\|_\infty + \lambda
\opnorm{\left(\Qhat_{SS}\right)^{-1}}_\infty.
\end{equation}
\item[(b)] Moreover, if $\rho_\lambda$ is $(\mu, \gamma)$-amenable and
  the minimum value $\betamin \defn \min_{j \in S} |\beta^*_j|$ is lower-bounded as
\begin{subequations}
\begin{align}
\label{EqnLavender}
\betamin \geq \lambda \left(\gamma + \opnorm{\left(\Qhat_{SS}\right)^{-1}}_\infty \right) +
\left\|\left(\Qhat_{SS}\right)^{-1} \nabla
\Loss_\numobs(\betastar)_S\right\|_\infty,
\end{align}
then $\betahat$ agrees with the oracle estimator $\boracle$, and we
have the tighter bound
\begin{align}
\label{EqnWatermelon}
\|\betahat - \betastar\|_\infty \le \left\|\left(\Qhat_{SS}\right)^{-1} \nabla
\Loss_\numobs(\betastar)_S\right\|_\infty.
\end{align}
\end{subequations}
\end{enumerate}
\end{thm*}
\noindent The proof of Theorem~\ref{ThmEllInf} is provided in
Appendix~\ref{SecThmEllInf}.

\paragraph{Remark:}
Theorem~\ref{ThmEllInf} underscores the strength of $(\mu,
\gamma)$-amenable regularizers. Indeed, with the addition of a
beta-min condition~\eqref{EqnLavender}, which provides a lower bound on the minimum signal strength, the unbiasedness property
allows us to remove the second term in inequality~\eqref{EqnHoneydew}
and obtain a faster oracle rate~\eqref{EqnWatermelon}. As described
in greater detail in the corollaries below, we may show that the
right-hand expression in inequality~\eqref{EqnWatermelon} is bounded
by $\order\left(\sqrt{\frac{\log p}{n}}\right)$, with high
probability, provided the spectrum of $\nabla^2 \Loss_n(\betastar)$ is bounded appropriately. \\

\noindent We now unpack the consequences of
Thoerems~\ref{ThmSuppRecovery} and~\ref{ThmEllInf} for several
concrete examples.


\subsection{Ordinary least squares linear regression}
\label{SecOLS}

Our first application focuses on the setting of ordinary least squares,
together with the nonconvex regularizers introduced in
Section~\ref{SecPenalties}: SCAD, MCP, and LSP. We compare the
consequences of Theorems~\ref{ThmSuppRecovery} and~\ref{ThmEllInf} for
each of these regularizers with the corresponding results for the convex
$\ell_1$-penalty. Our theory demonstrates a clear advantage of using
nonconvex regularizers such as SCAD and MCP that are $(\mu,
\gamma)$-amenable; whereas support recovery based on $\ell_1$-based
methods is known to require fairly stringent incoherence conditions,
our corollaries show that methods based on nonconvex regularizers will
guarantee support recovery even in the absence of incoherence conditions.

The $\rho_\lambda$-regularized form of least squares regression may be
written in the form
\begin{align}
\label{EqnOLSLoss}
\betahat \in \arg\min_{\|\beta\|_1 \le R} \left\{\frac{1}{2} \beta^T
\frac{X^TX}{n} \beta - \frac{y^T X}{n} \beta +
\rho_\lambda(\beta)\right\}.
\end{align}
Note that the Hessian of the loss function is given by $\nabla^2
\Loss_\numobs(\beta) = \frac{X^T X}{\numobs}$. While the sample
covariance matrix is always positive semidefinite, it has rank at most
$n$. Hence, in high-dimensional settings where $n < p$, the Hessian of
the loss function has at least $p-n$ zero eigenvalues, implying that
\emph{any} nonconvex regularizer $\rho_\lambda$ makes the overall
program~\eqref{EqnOLSLoss} nonconvex. \\

In analyzing the family of estimators~\eqref{EqnOLSLoss}, we assume
throughout that $\numobs \geq c_0 \kdim \log \pdim$, for a sufficiently
large constant $c_0$. By known information-theoretic
results~\cite{WaiIEEE09}, this type of lower bound is required for any
method to recover the support of a $\kdim$-sparse signal, hence is
not a limiting restriction.  With this setup, we have the following
result, proved in Appendix~\ref{AppCorOLSLoss}:
\begin{cor*}
\label{CorOLSLoss}
Suppose $X$ and $\epsilon$ are sub-Gaussian, and regularization
parameters $(\lambda, R)$ are chosen such that $\|\betastar\|_1 \le \frac{R}{2}$
and $\CLOW \sqrt{\frac{\log p}{n}} \le \lambda \le \frac{\CUP}{R}$, for
some constants $\CLOW$ and $\CUP$. Also suppose the sample covariance matrix $\GamHat = \frac{X^TX}{n}$ satisfies the condition
\begin{equation}
\label{EqnCInf}
\opnorm{\GamHat_{SS}^{-1}}_\infty \le c_\infty.
\end{equation}

\begin{enumerate}
\item[(a)] Suppose $\rho_\lambda$ is $\mu$-amenable, with $\mu < \lambda_{\min}(\Sigma_x)$, and $\GamHat$ also satisfies the
  incoherence condition
\begin{equation}
\label{EqnSampIncoh}
\opnorm{\GamHat_{S^cS} \GamHat_{SS}^{-1}}_\infty \le \eta < 1.
\end{equation}
Then with probability at least $1 - c_1 \exp(-c_2 \min\{k, \log p\})$, the nonconvex objective~\eqref{EqnOLSLoss} has a unique stationary point $\betahat$ (corresponding to the global
  optimum). Furthermore, $\supp(\betahat) \subseteq \supp(\betastar)$, and
\begin{equation}
\label{EqnEllInftyBound}
\|\betahat - \betastar\|_\infty \leq \CFINAL \sqrt{\frac{\log p}{n}} + c_\infty \lambda.
\end{equation}
\item[(b)] Suppose the regularizer $\rho_\lambda$ is $(\mu,
  \gamma)$-amenable, with $\mu < \lambda_{\min}(\Sigma_x)$. Also suppose
  \begin{equation*}
  \betamin \ge \lambda(\gamma + c_\infty) + c_3 \sqrt{\frac{\log p}{n}}.
  \end{equation*}
  Then with probability at least \mbox{$1 - c_1 \exp(-c_2 \min \{k,
    \log p\})$,} the nonconvex objective~\eqref{EqnOLSLoss} has a
  unique stationary point $\betahat$ given by the oracle estimator $\boracle$, and
  \begin{equation}
  \label{EqnEllInftyBound2}
  \|\betahat - \betastar\|_\infty \leq \CFINAL \sqrt{\frac{\log p}{n}}.
  \end{equation}
\end{enumerate}
\end{cor*}
\noindent Note that if we also have the beta-min condition $\betamin \geq 2 \left(c_3 \sqrt{\frac{\log
    \pdim}{\numobs}} + c_\infty \lambda\right)$ in part (a), then $\betahat$ is still a sign-consistent estimate of
$\betastar$; however, the guaranteed bound~\eqref{EqnEllInftyBound} is looser than the oracle bound~\eqref{EqnEllInftyBound2} derived in part (b). \\

The proof of Corollary~\ref{CorOLSLoss} is provided in
Appendix~\ref{AppCorOLSLoss}. Here, we make some comments about its
consequences.  Regularizers satisfying the conditions of part (b)
include the SCAD and MCP penalties.  Recall that for the SCAD penalty,
we have $\mu = \frac{1}{a-1}$; and for the MCP, we have $\mu =
\frac{1}{b}$. Hence, the lower-eigenvalue condition translates into
$\frac{1}{a-1} < \lambda_{\min}(\Sigma_x)$ and $\frac{1}{b} <
\lambda_{\min}(\Sigma_x)$, respectively.  The LSP penalty is an
example of a regularizer that satisfies the conditions of part (a), but not part (b): with this choice, we have $\mu = \lambda^2$, so
the condition $\mu < \lambda_{\min}(\Sigma_x)$ is
satisfied asymptotically whenever $\lambda_{\min}(\Sigma_x)$ is
bounded below by a constant.  A version of part (a) also holds for the
$\ell_1$-penalty, as shown in past work~\cite{Wai09}.

A valuable consequence of Corollary~\ref{CorOLSLoss} is that it establishes conditions under which stationary points are unique and variable selection consistency
holds, when using certain nonconvex regularizers. The
distinguishing point between parts (a) and (b) of the corollary is that using $(\mu, \gamma)$-amenable regularizers allow us to do away with an
incoherence assumption~\eqref{EqnSampIncoh} and guarantee that the
unique stationary point is in fact equal to the oracle estimator.

Furthermore, a great deal of past work on nonconvex
regularizers~\cite{FanLi01, Zha10, Zha12, ZhaZha12} has focused on the
ordinary least squares regression objective~\eqref{EqnOLSLoss}; hence,
it is instructive to interpret the results of
Corollary~\ref{CorOLSLoss} in light of this existing
work. Zhang~\cite{Zha12} shows that the two-step MC+ estimator
(beginning with a global optimum of the program~\eqref{EqnOLSLoss}
with the MCP regularizer) is guaranteed to be consistent for variable
selection, under only a sparse eigenvalue assumption on the design
matrix. Our result shows that the global optimum obtained in
the MCP step is actually already guaranteed to be consistent for variable
selection, provided we have only slightly stronger assumptions about
lower- and upper-eigenvalue bounds on the design matrix. In another
related paper, Wainwright~\cite{WaiIEEE09} establishes necessary
conditions for support recovery in a linear regression setting when
the covariates are drawn from a Gaussian distribution. As remarked in
that paper, the necessary conditions only require eigenvalue bounds on
the design matrix, in contrast to the more stringent incoherence
conditions appearing in necessary and sufficient conditions for the
success of the Lasso~\cite{Wai09,Zhao06}. Using standard matrix concentration results for sub-Gaussian variables, it may be shown that the inequalities~\eqref{EqnCInf} and~\eqref{EqnSampIncoh} hold, with high probability, when the population-level bounds are satisfied:
\begin{equation}
\label{EqnPopIncoh}
\opnorm{(\Sigma_{SS})^{-1}}_\infty \le \frac{c_\infty}{2}, \qquad  \text{and} \qquad \opnorm{(\Sigma_x)_{S^cS} (\Sigma_x)_{SS}^{-1}}_\infty \le
\frac{\eta}{2}.
\end{equation}
However, the second inequality~\eqref{EqnPopIncoh} is a fairly strong assumption
on the covariance matrix $\Sigma_x$, and as we explore in the simulations
of Section~\ref{SecSims} below, simple covariance matrices such as
the class~\eqref{EqnNonincoh} defined in Section~\ref{SecMatClass}
fail to satisfy the latter condition. Hence,
Corollary~\ref{CorOLSLoss} shows clear advantage of using the SCAD or
MCP regularizers over the $\ell_1$-penalty or LSP when $\Sigma_x$ is
not incoherent.


\subsection{Linear regression with corrupted covariates}

We now shift our focus to an application where the loss function
itself is nonconvex. In particular, we analyze the situation when the
general loss function is defined according to
equation~\eqref{EqnLossLinear}. To simplify our discussion, we only
state an explicit corollary for the case when $\rho_\lambda$ is the
convex $\ell_1$-penalty; the most general case, involving a nonconvex
quadratic form and a nonconvex regularizer, is simply a hybrid of the
analysis below and the arguments of the previous section. Our goal is
to illustrate the applicability of the primal-dual witness technique
for nonconvex loss functions. \\

Let us recall the problem of linear regression with corrupted
covariates, as previously introduced in Section~\ref{SecLoss}.  The
pairs $\{(x_i, y_i)\}_{i=1}^\numobs$ are drawn according to the
standard linear model $y_i = x_i^T \betastar + \epsilon_i$.  While the
response vector $y = \{y_i\}_{i=1}^\numobs$ is assumed to be observed, suppose
we observe only the corrupted versions $z_i = x_i + w_i$ of the
covariates.  Based on the observed variables $\{(z_i,
y_i)\}_{i=1}^\numobs$, we may then compute the quantities
\begin{equation}
\label{EqnGamCorrupted}
(\GamHat, \gamhat) \defn \left(\frac{Z^T Z}{n} - \Sigma_w, \frac{Z^T
  y}{n}\right),
\end{equation}
and estimate $\betastar$ based on the following
nonconvex program:
\begin{equation}
\label{EqnCorrLinearL1}
\betahat \in \arg\min_{\|\beta\|_1 \le R} \left\{\frac{1}{2} \beta^T
\GamHat \beta - \gamhat^T \beta + \lambda \|\beta\|_1\right\}.
\end{equation}
Also suppose $n \ge k^2$ and $\numobs \geq c_0 k \log p$, for a
sufficiently large constant $c_0$.

\begin{cor*}
\label{CorCorrLinearL1}
Suppose $(X, w, \epsilon)$ are sub-Gaussian,
$\lambda_{\min}(\Sigma_x) > 0$, and $(\lambda, R)$ are chosen
such that $\|\betastar\|_1 \le \frac{R}{2}$ and $\CLOW \sqrt{\frac{\log p}{n}}
\le \lambda \le \frac{\CUP}{R}$.  If in addition,
\begin{equation*}
\opnorm{\GamHat_{SS}^{-1}}_\infty \le c_\infty,
\end{equation*}
and
\begin{equation}
\label{EqnIncohLinear}
\opnorm{\GamHat_{S^cS} \GamHat_{SS}^{-1}}_\infty \le \eta < 1,
\end{equation}
then with probability at least $1 - c_1 \exp(-c_2 \min\{k, \log p\})$,
the nonconvex objective~\eqref{EqnCorrLinearL1} has a unique
stationary point $\betahat$ (corresponding to the global optimum) such
that \mbox{$\supp(\betahat) \subseteq \supp(\betastar)$,} and
\begin{align}
\label{EqnEllInfyLogistic}
\|\betahat - \betastar\|_\infty \leq & c_3 \, \sqrt{\frac{\log p}{n}} + c_\infty \lambda.
\end{align}
\end{cor*}
\noindent Note that if in addition, we have a lower bound of the form
$\betamin \geq 2 \left(c_3 \sqrt{\frac{\log p}{n}} + c_\infty \lambda\right)$, then we are guaranteed
that $\betahat$ is sign-consistent for $\betastar$. \\

Corollary~\ref{CorCorrLinearL1} may be understood as an extension of
part (a) of Corollary~\ref{CorOLSLoss}: it shows how the primal-dual
witness technique may be used even in a setting where the loss function is
nonconvex. Under the same incoherence
assumption~\eqref{EqnIncohLinear} and the familiar sample size scaling
$n \ge c_0 k \log p$ of the usual Lasso, stationary points of the
modified (nonconvex) Lasso program~\eqref{EqnCorrLinearL1} are also
support-recovery consistent.  Corollary~\ref{CorCorrLinearL1} also
implies the rather surprising result that, although the
objective~\eqref{EqnCorrLinearL1} is indeed nonconvex whenever $n < p$
and $\Sigma_w \succ 0$, it nonetheless has a unique stationary point
that is in fact equal to the global optimum. This result further
clarifies the simulation results appearing in Loh and
Wainwright~\cite{LohWai11a}. Indeed, those simulations are performed
with the setting $\Gamma = I_p$, so the incoherence
condition~\eqref{EqnIncohLinear} holds, with high probability, with $\eta$ close to
0. A careful inspection of the plots in Figure 2 of Loh and
Wainwright~\cite{LohWai11a} confirms the theoretical conclusion of
Corollary~\ref{CorCorrLinearL1}; more detailed simulations for
non-identity assignments of $\Gamma$ appear in Section~\ref{SecSims}
below. The proof of Corollary~\ref{CorCorrLinearL1} is provided in
Appendix~\ref{AppCorCorrLinearL1}.


\subsection{Generalized linear models}

To further illustrate the power of nonconvex regularizers, we now move
to the case where the loss function is the negative log likelihood of
a generalized linear model. We show that the incoherence condition
may again be removed if the regularizer $\rho_\lambda$ is $(\mu,
\gamma)$-amenable (as is the case for the SCAD and MCP regularizers).

For $\{(x_i, y_i)\}_{i=1}^n$ drawn from a GLM
distribution~\eqref{EqnCondGLM}, we take $\Omega = \real^p$ and
construct the composite objective
\begin{equation}
\label{EqnGLMLoss}
\betahat \in \arg\min_{\|\beta\|_1 \le R} \left\{\frac{1}{n}
\sum_{i=1}^n (\psi(x_i^T \beta) - y_i x_i^T \beta) +
\rho_\lambda(\beta)\right\}.
\end{equation}
We impose the following conditions on the covariates and the link
function:
\begin{assumption}
\label{AsGLM}
\mbox{}
\begin{itemize}
\item[(i)] The covariates are uniformly bounded as
$\|x_i\|_\infty \le M$, for all $i = 1, \ldots, \numobs$.
\item[(ii)] There are positive constants $\kappa_2$ and $\kappa_3$, such
  that $\|\psi''\|_\infty \le \kappa_2$ and $\|\psi'''\|_\infty \le
  \kappa_3$.
\end{itemize}
\end{assumption}
The conditions of Assumption~\ref{AsGLM}, although somewhat stringent,
are nonetheless satisfied in various settings of interest. In
particular, for logistic regression, we have $\psi(t) = \log(1 +
\exp(t))$, so
\begin{equation*}
\psi'(t) = \frac{\exp(t)}{1 + \exp(t)}, \qquad \psi''(t) =
\frac{\exp(t)}{(1 + \exp(t))^2}, \qquad \text{and} \qquad \psi'''(t) =
\frac{\exp(t)(1 - \exp(t))}{(1+\exp(t))^3},
\end{equation*}
and we may verify that the boundedness conditions in
Assumption~\ref{AsGLM}(ii) are satisfied with $\kappa_2 = 0.25$ and
$\kappa_3 = 0.1$. Also note that the uniform bound on $\psi'''$ is used
implicitly in the proof for support recovery consistency in the logistic regression analysis of
Ravikumar et al.~\cite{RavEtal10}, whereas the uniform bound on
$\psi''$ also appears in the conditions for $\ell_1$- and
$\ell_2$-consistency in other past work~\cite{NegRavWaiYu12,LohWai13}.
The uniform boundedness condition in Assumption~\ref{AsGLM}(i) is somewhat less
desirable: although it always holds for categorical data, it does not
hold for Gaussian covariates.  We suspect that is possible to
relax this constraint, but since our main goal is to illustrate the
more general theory, we keep it here. In what follows, let $Q^* \defn \E\left[\frac{1}{n}
  \sum_{i=1}^n \psi''(x_i^T \betastar) x_i x_i^T\right]$ denote the Fisher information matrix.

\begin{cor*}
\label{CorGLMLoss}
Under Assumption~\ref{AsGLM} and given a sample size $\numobs \geq c_0
\kdim^3 \log \pdim$, suppose $\rho_\lambda$ is $(\mu,
\gamma)$-amenable with $\mu < c_\psi \lambda_{\min}(\Sigma_x)$, where
$c_\psi$ is a constant depending only on $\psi$, and $(\lambda, R)$
are chosen such that $\|\betastar\|_1 \le \frac{R}{2}$ and $\CLOW
\sqrt{\frac{\log p}{n}} \le \lambda \le \frac{\CUP}{R}$.  Also suppose
that
\begin{equation*}
\opnorm{(Q^*_{SS})^{-1}}_\infty \le c_\infty, \quad \mbox{and} \quad
\betamin \ge \lambda(\gamma + 2c_\infty) + c_3 \sqrt{\frac{\log
    p}{n}}.
\end{equation*}
Then with probability at least $1 - c_1 \exp(-c_2 \min\{k, \log p\})$, the
nonconvex objective~\eqref{EqnGLMLoss} has a unique stationary point
$\betahat$ given by the oracle estimator $\boracle$, and
\begin{equation}
\label{EqnEllInftyGLM}
\|\betahat - \betastar\|_\infty \leq c_3 \sqrt{\frac{\log p}{n}}.
\end{equation}
\end{cor*}

It is worthwhile to compare Corollary~\ref{CorGLMLoss} with the analysis of
$\ell_1$-regularized logistic regression given in Ravikumar et
al.~\cite{RavEtal10} (see Theorem 1 in their paper).  Both results
require that the sample size is lower-bounded as $ \numobs \geq c_0
\kdim^3 \log \pdim$, but Ravikumar et al.~\cite{RavEtal10} also require
$Q^*$ to satisfy the incoherence condition
\begin{equation}
\label{EqnGLMIncoh}
\opnorm{Q^*_{S^cS} (Q^*_{SS})^{-1}}_\infty \le \eta < 1.
\end{equation}
As noted in their paper and by other authors, the
incoherence condition~\eqref{EqnGLMIncoh} is difficult to interpret
and verify for general GLMs. In contrast, Corollary~\ref{CorGLMLoss} shows that by
using a properly chosen nonconvex regularizer, this incoherence requirement may be
removed entirely.  In addition, Corollary~\ref{CorGLMLoss} is attractive in its generality, since it applies to more than just the logistic case with an
$\ell_1$-penalty and extends to various nonconvex problems
where the uniqueness of stationary points is not evident a priori. The proof of Corollary~\ref{CorGLMLoss} is contained in Appendix~\ref{AppCorGLMLoss}.


\subsection{Graphical Lasso}
\label{SecGLasso}

Finally, we discuss the consequences of our theorems for the graphical
Lasso, as previously described in Section~\ref{SecLoss}.  Recall that
for the graphical Lasso, the observations consist of a collection
$\{x_i\}_{i=1}^n$ of $\pdim$-dimensional vectors, and the goal is to
recover the support of the inverse covariance matrix $\Thetastar =
(\Cov(X))^{-1}$.  The analysis here is different and more subtle,
because we seek a high-dimensional result in which the sample size
scales only with the \emph{row sparsity} of the inverse covariance
matrix, as opposed to the total number of of nonzero parameters
(which grows linearly with the matrix dimension, for any connected
graph).  

In order to prove such a result, we consider the constrained estimator
\begin{equation}
\label{EqnGlassoLoss}
\Thetahat \in \arg\min_{\Theta \in S_{++}^p, \; \opnorm{\Theta}_2 \le
  \kappa} \left\{\tr(\Sigmahat \Theta) - \log \det(\Theta) + \sum_{j
  \neq k} \rho_\lambda(\Theta_{jk})\right\},
\end{equation}
where $S_{++}^\pdim$ denotes the convex cone of symmetric, strictly
positive definite matrices, and we impose the spectral norm bound
$\opnorm{\Theta}_2 \leq \kappa$ on the estimate.\footnote{We
  denote the constraint radius by $\kappa$ rather than $R$, in order
  to emphasize the difference from previous situations.}  

A more standard choice would be to use the $\ell_1$-norm bound
$\sum_{i,j} |\Theta_{ij}| \le R$ as the side-constraint, as done in
our past work on Frobenius norm bounds for the
graphical Lasso with nonconvex regularizers~\cite{LohWai13}. However, as we will show
here, the formulation~\eqref{EqnGlassoLoss} actually leads to variable
selection consistency results under the milder scaling $n \succsim d^2
\log p$, rather than the scaling $n \succsim s \log
p$ obtained in our analysis on Frobenius norm bounds~\cite{LohWai13}. Here, we use
$d$ to denote the maximum number of nonzeros in any row/column of
$\Thetastar$, and $s \defn |\supp(\Thetastar)|$ to denote the total
number of nonzero entries.

A few remarks are in order. First, when $\rho_\lambda$ is the convex
$\ell_1$-penalty and $\kappa = \infty$, the
program~\eqref{EqnGlassoLoss} is identical to the standard graphical
Lasso (e.g., ~\cite{AspBanGha08,RavEtal11, FriedHasTib2007, Rot08}).
However, since we are interested in scenarios where $\rho_\lambda$ is
allowed to be nonconvex, we include an additional spectral norm
constraint governed by $\kappa$.  As a technical comment, note that
our original assumptions required $\Omega$ to be an open subset of
$\real^p$. In the analysis to follow, we handle the symmetry
constraint $\Theta = \Theta^T$ by treating the
program~\eqref{EqnGlassoLoss} as an optimization problem over the
space $\real^{\frac{p^2+p}{2}}$, and then take $\Omega$ to be the open
subset of $\real^{\frac{p^2+p}{2}}$ corresponding to positive definite
matrices. Doing so makes the program~\eqref{EqnGlassoLoss} consistent
with the framework laid out earlier in our paper. In this case, the
oracle estimator is defined by
\begin{equation}
\label{EqnThetaOracle}
\Thetahat^{\mathcal{O}} \defn \arg\min_{\Theta \succeq 0}
\left\{\tr(\Sigmahat \Theta) - \log \det(\Theta): \; \supp(\Theta)
\subseteq \supp(\Thetastar) \right\}.
\end{equation}
With this setup, we have the following guarantee:
\begin{cor*}
\label{CorGlassoLoss}
Given a sample size $\numobs \geq c_0 d^2 \log \pdim$, suppose
the $x_i$'s are drawn from a sub-Gaussian distribution, and the
regularizer $\rho_\lambda$ is $(\mu, \gamma)$-amenable. Also suppose
that
\begin{equation*}
\opnorm{\left(\Thetastar \otimes \Thetastar\right)_{SS}}_\infty \le
c_\infty, \quad \mbox{and} \quad \betamin \ge \lambda(\gamma +
2c_\infty) + c_3 \sqrt{\frac{\log p}{n}}.
\end{equation*}
Then with probability at
least $1 - c_1 \exp(-c_2 \log p)$, the program~\eqref{EqnGlassoLoss}
with $\kappa = \sqrt{\frac{2}{\mu}}$ has a unique stationary point
$\Thetahat$ given by the oracle estimator $\Theta^{\mathcal{O}}$, and
\begin{equation}
\label{EqnEllInftyGL}
\|\Thetahat - \Thetastar\|_{\max} \leq c_3 \sqrt{\frac{\log p}{n}}.
\end{equation}
\end{cor*}

Moreover, as shown in the proof in Appendix~\ref{AppCorGlassoLoss},
the support containment condition and elementwise
bound~\eqref{EqnEllInftyGL} also imply bounds on the Frobenius and
spectral norms of the error, namely
\begin{equation}
\label{EqnSpectralGL}
\opnorm{\Thetahat - \Thetastar}_F \leq c_3 \, \sqrt{\frac{s \log
    p}{n}}, \qquad \text{and} \qquad \opnorm{\Thetahat - \Thetastar}_2
\leq c_3 \, \min\{\sqrt{s}, d\} \cdot \sqrt{\frac{\log p}{n}}.
\end{equation}

We reiterate that Corollary~\ref{CorGlassoLoss} does
\emph{not} involve any restrictive incoherence assumptions on the
matrix $\Thetastar$.  As remarked in past work~\cite{Mei08,RavEtal11},
such incoherence conditions for the graphical Lasso are very
restrictive---much more so than the corresponding incoherence
conditions for graph recovery using neighborhood
regression~\cite{MeiBuh06,Zhao06,Wai09}. Thus, in this setting, our corollary illustrates
another distinct advantage of nonconvex regularization.


\section{Simulations}
\label{SecSims}

In this section, we report the results of various simulations that we ran
in order to verify our theoretical results.

\subsection{Optimization algorithm}

We begin by describing the algorithm we use to optimize the
program~\eqref{EqnMEst}. We rewrite the program as
\begin{equation}
\label{EqnSplit}
\betahat \in \arg\min_{\|\beta\|_1 \le R, \; \beta \in \Omega}
\Big\{\underbrace{\Loss_\numobs(\beta) -
  q_\lambda(\beta)}_{\Lossbar_n(\beta)} + \lambda \|\beta\|_1\Big\},
\end{equation}
and apply the composite gradient descent algorithm due to
Nesterov~\cite{Nes07}. The updates of the composite gradient procedure
are given by
\begin{equation}
\label{EqnNesterov}
\beta^{t+1} \in \arg \min_{\|\beta\|_1 \le R, \; \beta \in \Omega}
\left\{\frac{1}{2} \left\|\beta - \left(\beta^t - \frac{\nabla
  \Lossbar_n(\beta^t)}{\eta}\right)\right\|_2^2 + \frac{\lambda}{\eta}
\|\beta\|_1\right\},
\end{equation}
where $\frac{1}{\eta}$ is the stepsize.

In the particular simulations of this section, we take $\Omega =
\real^p$. Then the iterates~\eqref{EqnNesterov} have the convenient
closed-form expression
\begin{equation}
\label{EqnUpdate}
\beta^{t+1} = S_{\lambda/\eta}\left(\beta^t - \frac{\nabla
  \Lossbar_n(\beta^t)}{\eta}\right), \qquad \mbox{where
  $S_{\lambda/\eta}(\beta_j) = \sign(\beta_j)\left(|\beta_j| -
  \frac{\lambda}{\eta}\right)_+$.}
\end{equation}
The following proposition guarantees the computational efficiency of
the general composite gradient descent
algorithm~\eqref{EqnNesterov}. For ease of analysis, we assume that
$\|\betastar\|_2 \le 1$, and the Taylor error $\scriptT(\beta_1,
\beta_2) \defn \Loss_\numobs(\beta_1) - \Loss_\numobs(\beta_2) -
\inprod{\nabla \Loss_\numobs(\beta_2)}{\beta_1 - \beta_2}$ satisfies
the following restricted strong convexity condition, for all $\beta_2
\in \ball_2(3) \cap \ball_1(R)$:
\begin{subnumcases}{
\label{EqnRSCt}
\scriptT(\beta_1, \beta_2) \geq}
\label{EqnT1RSC}
\alpha_1 \|\Delta\|_2^2 - \tau_1 \frac{\log p}{n} \|\Delta\|_1^2,
\quad \quad \forall \|\Delta\|_2 \leq 3, \\
\label{EqnT2RSC}
\alpha_2 \|\Delta\|_2 - \tau_2 \sqrt{\frac{\log p}{n}} \|\Delta\|_1,
\quad \forall \|\Delta\|_2 \geq 3,
\end{subnumcases}
as well as the restricted smoothness condition
\begin{equation}
\label{EqnRSMt}
\scriptT(\beta_1, \beta_2) \le \alpha_3 \|\beta_1 - \beta_2\|_2^2 +
\tau_3 \frac{\log p}{n} \|\beta_1 - \beta_2\|_1^2, \qquad \forall
\quad \beta_1, \beta_2 \in \Omega.
\end{equation}
We also assume for simplicity that $q_\lambda$ is convex, as is the
case for all the regularizers studied in this paper.
	
In the following statement, we let $\betahat$ be the unique global
optimum of the program~\eqref{EqnSplit}. Also denote $\phi(\beta) \defn \Loss_\numobs(\beta) +
\rho_\lambda(\beta)$.
\begin{prop*}
\label{PropOpt}
Suppose $\Loss_\numobs$ satisfies the RSC~\eqref{EqnRSCt} and
RSM~\eqref{EqnRSMt} conditions, and assume the regularizer $\rho_\lambda$ is
$\mu$-amenable with $\frac{\mu}{2} < \alpha \defn \min\{\alpha_1,
\alpha_2\}$, and $q_\lambda$ is convex.  Let the scalars $(R, \lambda)$ be chosen to satisfy the
bounds $R \sqrt{\frac{\log p}{n}} \le c$ and $\lambda \ge 4 \;
\max\left\{\|\nabla \Loss_\numobs(\betastar)\|_\infty, \; \tau
\sqrt{\frac{\log p}{n}}\right\}$.  Then for any stepsize parameter
$\eta \ge \max\{2\alpha_3 - \mu, \mu\}$ and tolerance $\delta
\geq c_0 \sqrt{\frac{k \log p}{n}} \cdot \|\betahat - \betastar\|_2$,
the iterates of the composite gradient descent
algorithm~\eqref{EqnNesterov} satisfy the $\ell_2$-bound
\begin{equation*}
\|\beta^t - \betahat\|_2 \leq c_1 \cdot \frac{\delta}{(2\alpha - \mu)^{1/2}},
\qquad \mbox{for all iterations $t \ge T^*(\delta)$,}
\end{equation*}
where $T^*(\delta) \defn \frac{2\log\left(\frac{\phi(\beta^0) -
    \phi(\betahat)}{\delta^2}\right)}{\log(1/\kappa)} + \left(1 +
\frac{\log 2}{\log(1/\kappa)}\right) \log \log \left(\frac{\lambda
  R}{\delta^2}\right)$.
\end{prop*}

We sketch the proof of Proposition~\ref{PropOpt} in
Appendix~\ref{AppPropOpt}. It establishes that the composite gradient
descent algorithm~\eqref{EqnNesterov} converges geometrically up to
tolerance \mbox{$\delta \asymp \sqrt{\frac{k \log p}{n}}
  \|\betahat - \betastar\|_2$;} moreover, only $T^*(\delta) \asymp \log(1/\delta)$ iterations are necessary.

Since we are interested in $\ell_\infty$-error bounds, we state a simple corollary to Proposition~\ref{PropOpt} that ensures convergence of the iterates in $\ell_\infty$-norm, up to accuracy $\order\left(\sqrt{\frac{\log p}{n}}\right)$, assuming statistical consistency of the global optimum and the scaling $n \succsim k^2 \log p$:

\begin{cor*}
\label{CorOpt}
Suppose, in addition to the assumptions of the previous proposition, that \mbox{$\|\betahat - \betastar\|_2 \le c \sqrt{\frac{k \log p}{n}}$}, and the sample size is lower-bounded as $n \geq c_0 k^2 \log p$. Then
the iterates of the composite gradient descent
algorithm~\eqref{EqnNesterov} satisfy the $\ell_\infty$-bound
\begin{equation}
\label{EqnEllInftyIterations}
\|\beta^t - \betahat\|_\infty \leq c_1'\cdot  \frac{1}{(2\alpha -
  \mu)^{1/2}} \sqrt{\frac{\log p}{n}}, \qquad \mbox{for all $t \ge
  T^*(\delta)$.}
\end{equation}
\end{cor*}
\noindent The proof of Corollary~\ref{CorOpt} is a simple consequence
of Proposition~\ref{PropOpt} and is supplied in
Appendix~\ref{AppCorOpt}. Note that the bound $\|\betahat - \betastar\|_2 \le c \sqrt{\frac{k \log p}{n}}$ holds, with high probability, as a consequence of Lemma~\ref{LemEll2Bounds} in Appendix~\ref{AppLemEll2}. Corollary~\ref{CorOpt} has a natural consequence for support
recovery: Suppose the estimate $\betahat$ satisfies an
$\ell_\infty$-bound of the form $\|\betahat - \betastar\|_\infty \leq
c_3 \sqrt{\frac{\log p}{n}}$, as guaranteed by
Theorem~\ref{ThmEllInf}.  Combining the
$\ell_\infty$-bound~\eqref{EqnEllInftyIterations} on iterate $\beta^t$
with the triangle inequality, we are then guaranteed that
\begin{align*}
\|\beta^t - \betastar\|_\infty & \leq  c_3' \sqrt{\frac{\log p}{n}},
\end{align*}
so the composite gradient descent algorithm
converges to a vector with the correct support, provided
\mbox{$\betamin \ge 2c_3' \sqrt{\frac{\log p}{n}}$.}


\subsection{Classes of matrices}
\label{SecMatClass}

Next, we describe two classes of matrices to be used in our
simulations.  The first class consists of matrices that do \emph{not}
satisfy the incoherence conditions, although the maximum and minimum
eigenvalues are bounded by constants. We define
\begin{equation}
\label{EqnNonincoh}
	M_1(\theta) = \left( \begin{array}{cccccccc} 1 & 0 & \cdots &
          0 & \theta & 0 & \cdots & 0 \\ 0 & 1 & \cdots & 0 & \theta &
          0 & \cdots & 0 \\ 0 & 0 & \ddots & 0 & \vdots & \vdots & &
          \vdots \\ 0 & 0 & \cdots & 1 & \theta & 0 & \cdots & 0
          \\ \theta & \theta & \cdots & \theta & 1 & 0 & \cdots & 0
          \\ 0 & 0 & \cdots & 0 & 0 & 1 & \cdots & 0 \\ \vdots &
          \vdots & & \vdots & \vdots & \vdots & \ddots & \vdots \\ 0 &
          0 & \cdots & 0 & 0 & 0 & \cdots & 1
	\end{array}
	\right).
\end{equation}
Hence, $M_1(\theta)$ is a matrix with 1's on the diagonal, $\theta$'s in
the first $k$ positions of the $(k+1)^\text{st}$ row and column, and
0's everywhere else. The following lemma, proved in
Appendix~\ref{AppLemGam1}, provides the incoherence parameter and
eigenvalue bounds for $M_1(\theta)$ as a function of $\theta$:
\begin{lem*}
\label{LemGam1}
With the shorthand notation $\Gamma = M_1(\theta)$ and $S = \{1,
\dots, k\}$, the incoherence parameter is given by
$\opnorm{\Gamma_{S^cS} \Gamma_{SS}^{-1}}_\infty = k\theta$, and the
minimum and maximum eigenvalues are given by
\mbox{$\lambda_{\min}(\Gamma) = 1 - \theta \sqrt{k}$} and
\mbox{$\lambda_{\max}(\Gamma) = 1 + \theta \sqrt{k}$.}
\end{lem*}

\noindent In particular, if $\theta \in \left(\frac{1}{k},
\frac{1}{\sqrt{k}}\right)$, Lemma~\ref{LemGam1} ensures that
$M_1(\theta)$ has bounded eigenvalues but does \emph{not} satisfy the
incoherence condition. \\

The second class of matrices is known as the \emph{spiked identity}
model~\cite{JohLu09} or \emph{constant correlation}
model~\cite{Zhao06}. We define the class according to
\begin{equation}
\label{EqnSpiked}
M_2(\theta) = \theta \mathbf{1} \mathbf{1}^T + (1-\theta) I_p,
\end{equation}
where $\theta \in [0,1]$ and $\mathbf{1} \in \real^p$ denotes the all
$1$'s vector.  An easy calcuation shows that
$\lambda_{\min}(M_2(\theta)) = 1 - \theta$ and
$\lambda_{\max}(M_2(\theta)) = 1 + \theta(k-1)$, whereas the
incoherence parameter is given by $\frac{\theta k}{1 + \theta(k-1)} <
1$ (see Corollary 1 of the paper~\cite{Zhao06}).


\subsection{Experimental results}

We ran experiments with the loss function coming from (a) ordinary
least squares linear regression, (b) least squares linear regression
with corrupted covariates, and (c) logistic regression. For all our
simulations, we used the regularization parameters $R = 1.1 \,
\|\betastar\|_1$ and $\lambda = \sqrt{\frac{\log p}{n}}$, and we set
the SCAD and MCP parameters to be $a = 2.5$ and $b = 1.5$,
respectively. Note that although the covariates in our simulations for
logistic regression do not satisfy the boundedness
Assumption~\ref{AsGLM}(i) imposed in our corollary, the generated plots
still agree qualitatively with our predicted theoretical results. \\

In our first set of simulations, we show that using the SCAD or MCP
regularizer in situations where the design matrix does not satisfy
incoherence conditions still results in an estimator that is variable
selection consistent. We generated i.i.d.\ covariates $x_i \sim N(0,
\Sigma_x)$, where $\Sigma_x = M_1(\theta)$ was obtained from the
family of non-incoherent matrices~\eqref{EqnNonincoh}, with $\theta =
\frac{2.5}{k}$. We chose $k \approx \sqrt{p}$ and $\betastar =
\left(\frac{1}{\sqrt{k}}, \cdots, \frac{1}{\sqrt{k}}, 0, \cdots,
0\right)$, the unit vector with the first $k$ components equal to
$\frac{1}{\sqrt{k}}$, and generated response variables according to
the linear model $y_i = x_i^T \betastar + \epsilon_i$, where
$\epsilon_i \sim N(0, (0.1)^2)$. In addition, we generated corrupted
covariates $z_i = x_i + w_i$, where $w_i \sim N(0, (0.2)^2)$, and $w_i
\condind x_i$. We then ran the composite gradient descent algorithm
with updates given by equation~\eqref{EqnUpdate}, where the loss
function is given by equation~\eqref{EqnLossLinear}, and $(\GamHat,
\gamhat)$ are defined as in
equation~\eqref{EqnGamCorrupted}. Figure~\ref{FigOLSCorr} shows the
results of our simulations for the problem sizes $p = 128, 256$, and
$256$. In panel (a), we see that the probability of correct support
recovery transitions sharply from 0 to 1 as the sample size increases
and $\rho_\lambda$ is the SCAD or MCP regularizer. In contrast, the probability
of recovering the correct support remains at 0 when $\rho_\lambda$ is
the $\ell_1$-penalty or LSP---by the structure of $\Sigma_x$,
regularization with the $\ell_1$-penalty or LSP results in an
estimator $\betahat$ that puts nonzero weight on the $(k+1)^\text{st}$
coordinate, as well. Note that we have rescaled the horizontal axis
according to $\frac{n}{k \log p}$ in order to match the scaling prescribed by
our theory; the three sets of curves for each regularizer roughly
align, as predicted by Theorem~\ref{ThmSuppRecovery}. Panel (b)
confirms that the $\ell_\infty$-error $\|\betahat - \betastar\|_\infty
$ decreases to 0 when using the SCAD and MCP, as predicted by
Theorem~\ref{ThmEllInf}. Finally, we plot the $\ell_2$-error
$\|\betahat - \betastar\|_2$ of the SCAD and MCP regularizers
alongside the $\ell_2$-error for the $\ell_1$-penalty and LSP in panel
(c). Although the $\ell_2$-error is noticeably smaller for SCAD and
MCP than for the $\ell_1$-penalty and LSP, as noted by previous
authors~\cite{FanLi01, BreHua11, MazEtal11}, all four regularizers are
nonetheless consistent in $\ell_2$-error, since a lower-eigenvalue
bound on the covariance matrix of the design is sufficient for
$\ell_2$-consistency~\cite{LohWai13}. For our choice of regularization
parameters, where the same value of $\lambda$ is shared between the
$\ell_1$-penalty and LSP, the two sets of curves for the
$\ell_1$-penalty and LSP nearly agree; as shown in Candes et
al.~\cite{CanEtal08}, the relative improvement of the LSP in
comparison to the $\ell_1$-penalty may vary widely depending on the
regularization parameter. \\

\begin{center}
\begin{figure}
	\begin{tabular}{cc}
		\includegraphics[width=7cm]{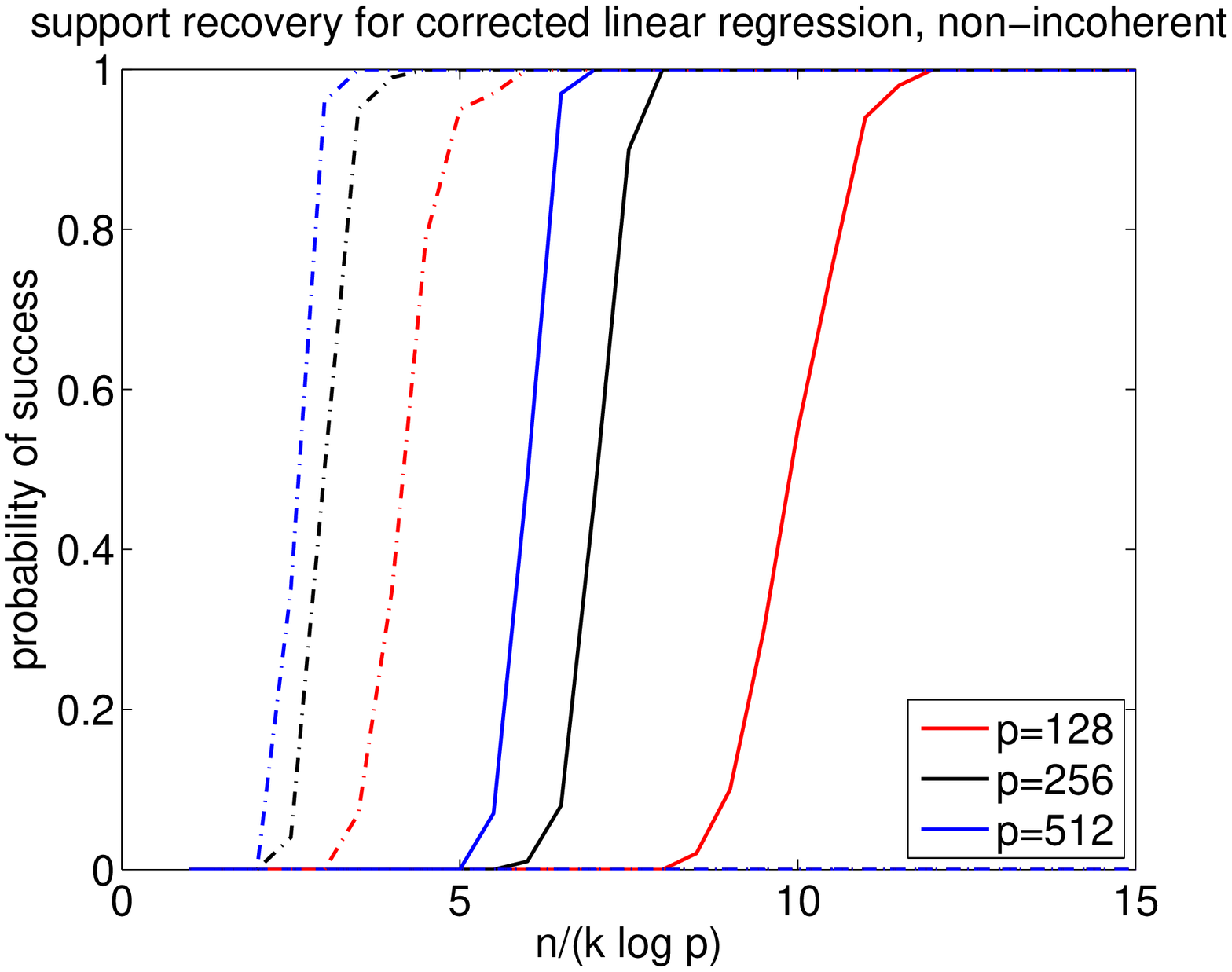}
                &
                \includegraphics[width=7cm]{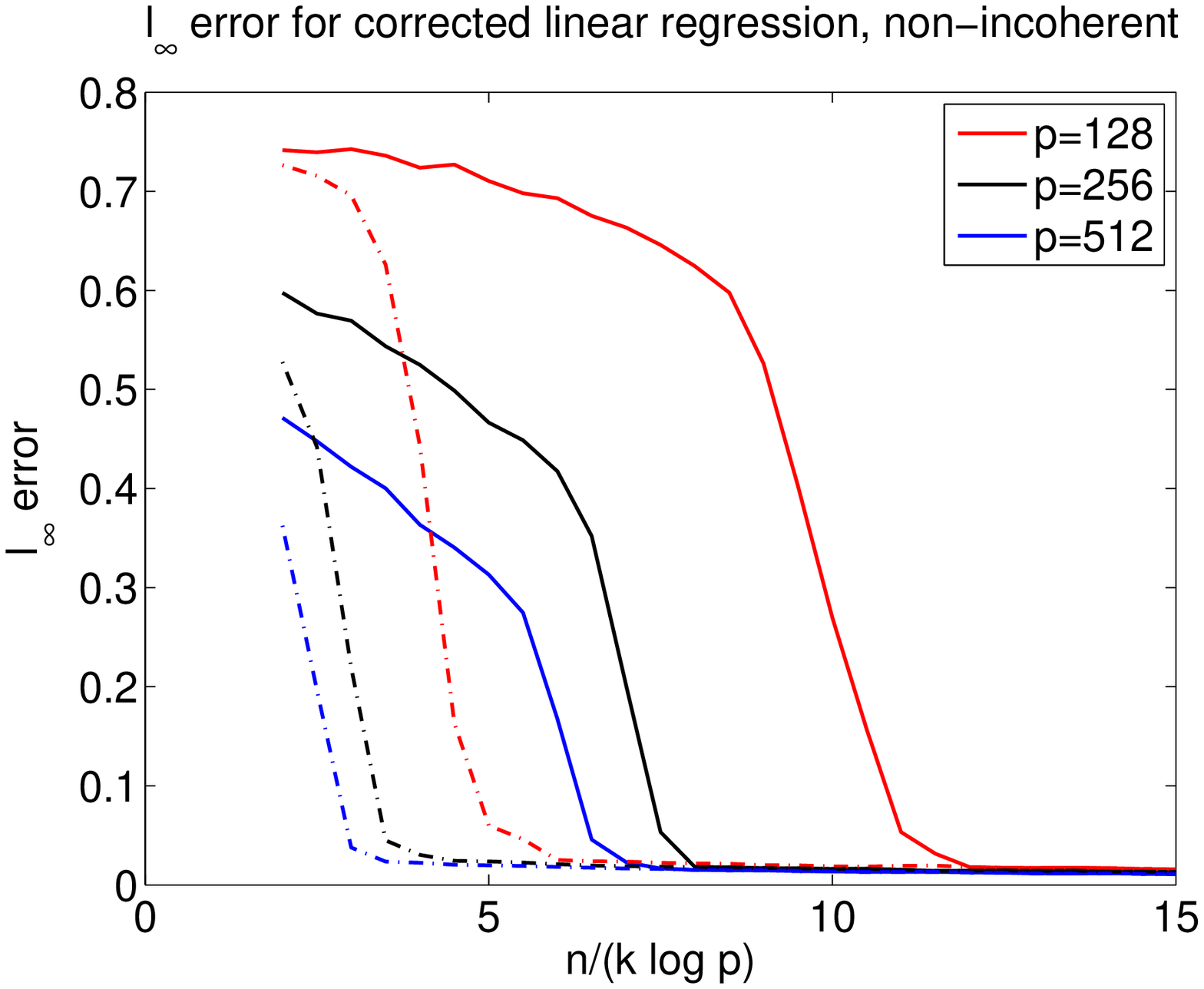}
                \\
		(a) & (b)
	\end{tabular}
	\begin{center}
	\begin{tabular}{c}
		\includegraphics[width=7cm]{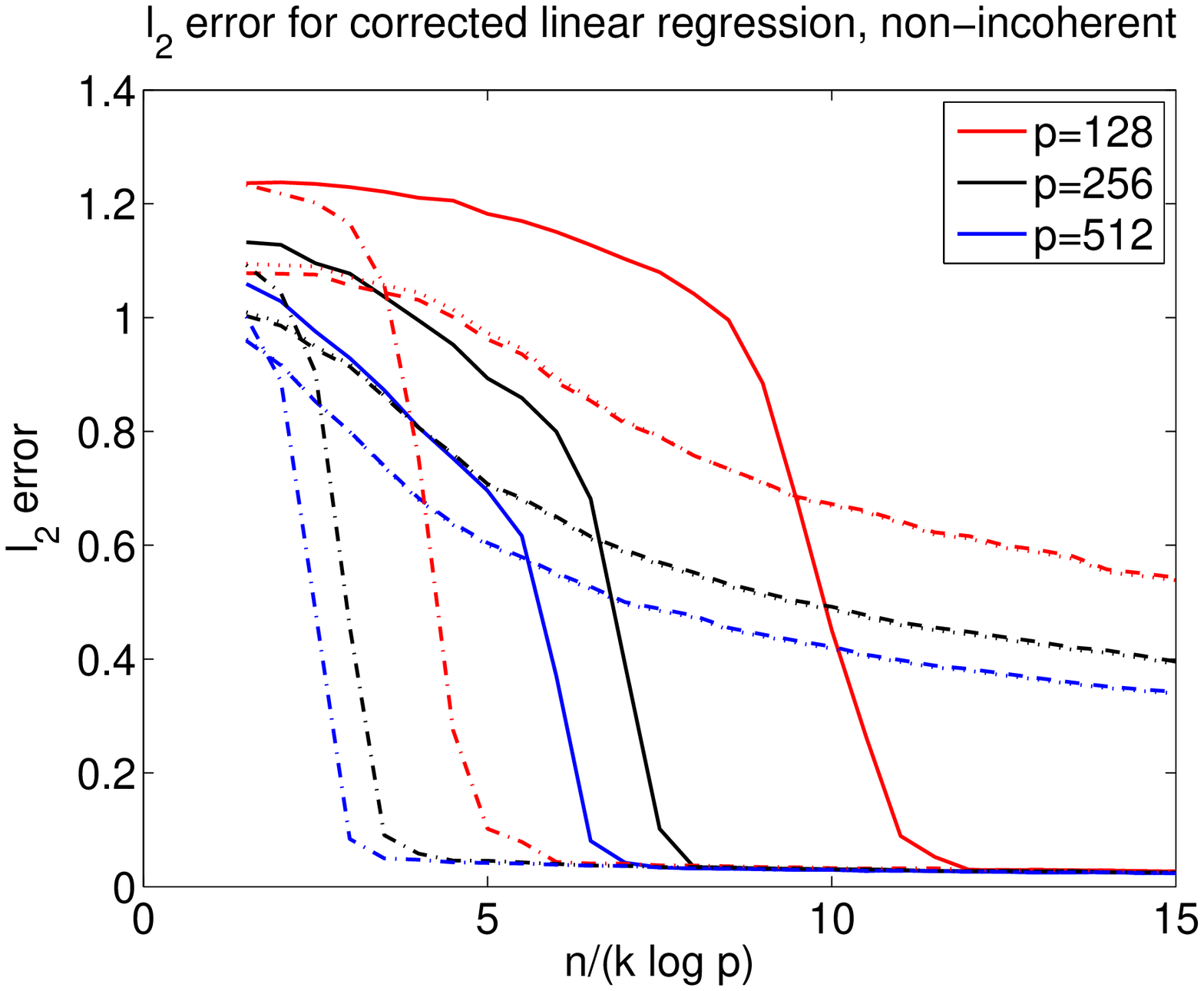} \\
	\end{tabular}
	\end{center}
\caption{Plots showing simulation results for least squares linear
  regression with covariates corrupted by additive noise, for three
  problem sizes: $p = 128$ (red), $p = 256$ (black), and $p = 512$
  (blue). (a) Plot showing variable selection consistency with the
  SCAD (solid) and MCP (dash-dotted) regularizers. The probability of
  success in recovering the correct signed support transitions sharply
  from 0 to 1 as a function of the sample size, agreeing with the
  theoretical predictions of Theorem~\ref{ThmSuppRecovery}. (b) Plot
  showing $\ell_\infty$-error $\|\betahat - \betastar\|_\infty$ with
  the SCAD (solid) and MCP (dash-dotted) regularizers. As predicted by
  Theorem~\ref{ThmEllInf}, both regularizers demonstrate consistency
  in $\ell_\infty$-error, even though the design matrix is \emph{not}
  incoherent. (c) Plot showing $\ell_2$-error $\|\betahat -
  \betastar\|_2$ with the $\ell_1$-penalty (dotted), LSP (dashed),
  SCAD (solid), and MCP (dash-dotted) regularizers. All four
  regularizers demonstrate consistency in $\ell_2$-error. Note that
  the two sets of lines for the $\ell_1$-penalty and LSP nearly align
  for this choice of regularization parameters.}
\label{FigOLSCorr}
\end{figure}
\end{center}

In our second set of simulations, we explore the uniqueness of
stationary points of the composite objectives. We focus on
settings where the loss function comes from either linear regression
with ordinary least squares, or logistic regression. Our theory
guarantees that stationary points are unique when $\mu < 2\alpha_1$,
but when $2\alpha_1 \le \mu$, multiple stationary points may emerge. In
fact, when $2\alpha_1 \le \mu$, convergence of the composite gradient
descent algorithm and consistent support recovery are no longer
guaranteed. In practice, we observe that multiple initializations of
the composite gradient descent algorithm still appear to converge to a
single stationary point with the correct support, when $\mu$ is
slightly larger than $2\alpha_1$; however, when that condition is violated
more severely, the composite gradient descent algorithm indeed
terminates at several distinct stationary
points. Figure~\ref{FigOLSUnique} shows the result of multiple runs of
the composite gradient descent algorithm with different regularizers
in the cleanly-observed linear regression setting. We generated
observations $x_i \sim N(0, \Sigma_x)$, with $\Sigma_x = M_2(\theta)$
coming from the family of spiked identity models~\eqref{EqnSpiked}, for $\theta = 0.7$ and $0.8$, and independent additive noise, $\epsilon_i \sim
N(0, (0.1)^2)$. We set the problem dimensions to be $p = 128$, $k
\approx \sqrt{p}$, and $n \approx 20 k \log p$, and generated
$\betastar$ to have $k$ nonzero values $\pm \frac{1}{\sqrt{k}}$ with
equal probability for each sign. When using the SCAD or MCP
regularizers (panels (b) and (d)), distinct stationary points emerge
and the recovered support is incorrect, since $2 \alpha_1
< \mu$. In contrast, the $\ell_1$-penalty and LSP still continue to
produce unique stationary points with the correct support (panels (a)
and (b)). Observe from the plots in Figure~\ref{FigOLSUnique} that the
error $\|\beta^t - \betastar\|_2$ decreases at a rate that is linear
on a log scale, as predicted by Theorem 3 of Loh and
Wainwright~\cite{LohWai13}, until it reaches the threshold of
statistical accuracy. Further note the significant increase in overall
precision from using SCAD or MCP, as seen by comparing the vertical
axes in panels (a) \& (b) and panels (c) \& (d).

\begin{center}
\begin{figure}
\begin{tabular}{cc}
	\includegraphics[width=7cm]{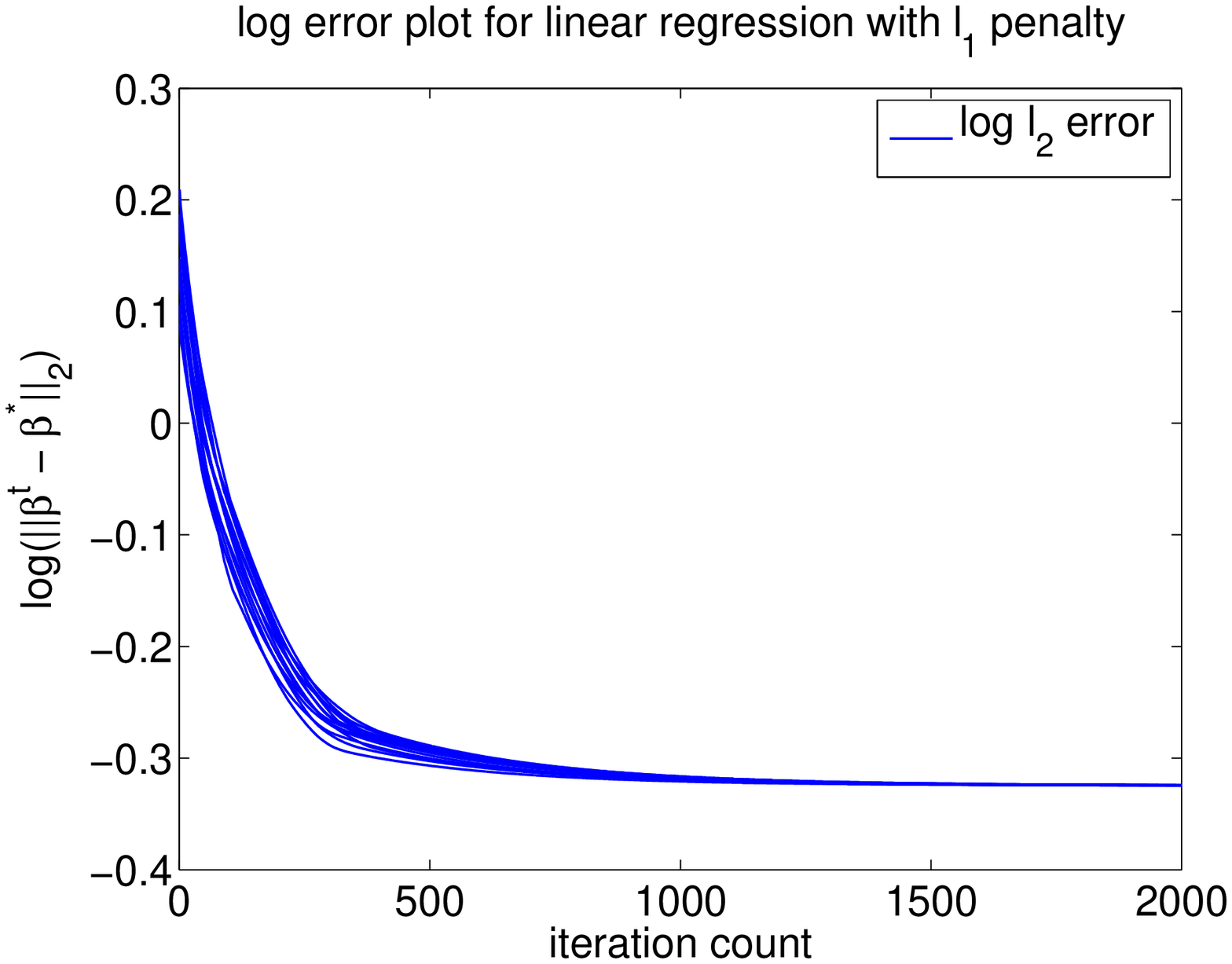} &
        \includegraphics[width=7cm]{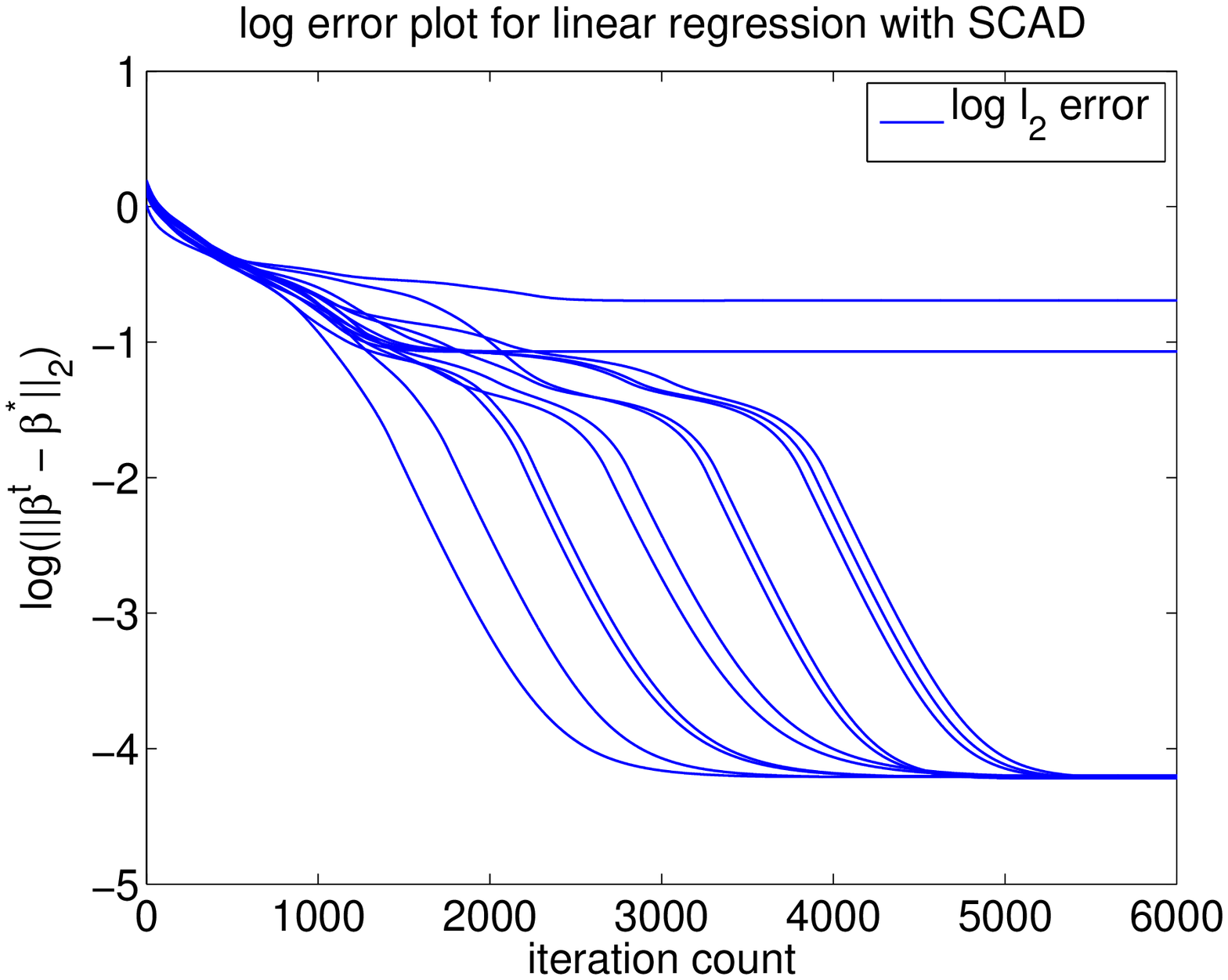} \\ (a) & (b) \\
\includegraphics[width=7cm]{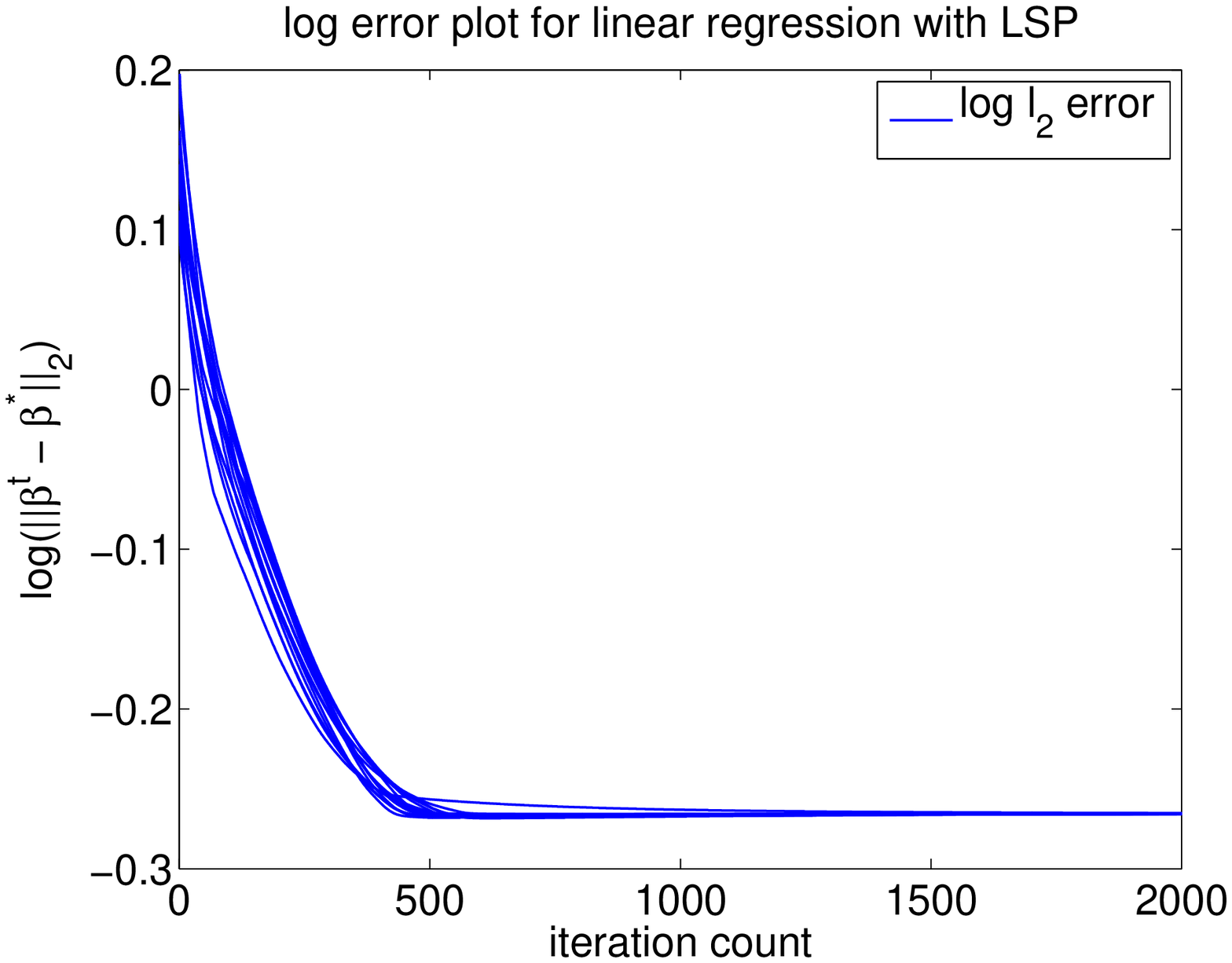} &
\includegraphics[width=7cm]{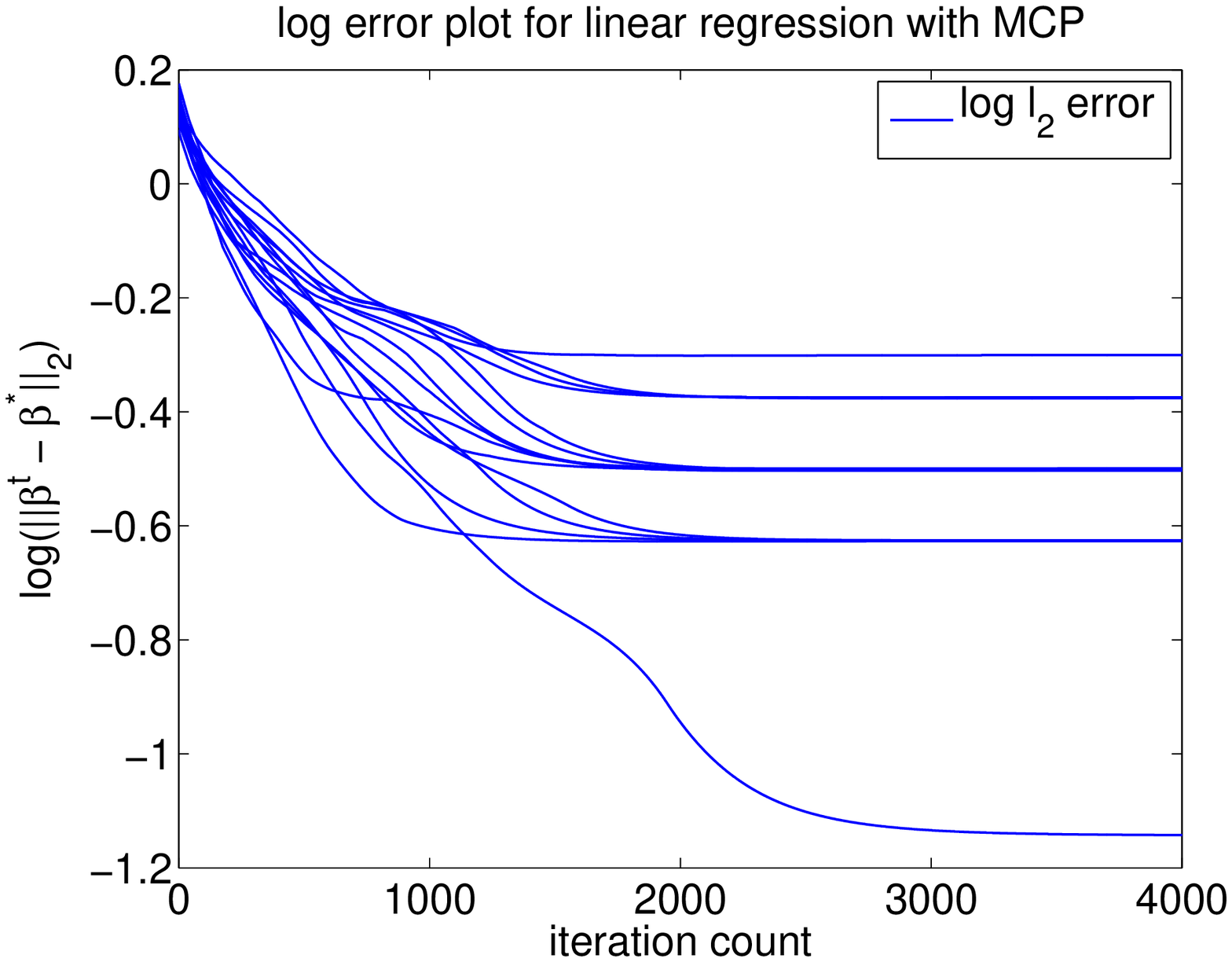} \\ (c) & (d)
\end{tabular}
\caption{Plots showing log $\ell_2$-error $\log\left(\|\beta^t -
  \betastar\|_2\right)$ as a function of iteration number $t$ for OLS
  linear regression with a variety of regularizers and 15 random
  initializations of composite gradient descent. The covariates are
  normally distributed with covariance matrix coming from a spiked
  identity model with parameter $\theta$. (a) $\ell_1$-penalty with
  $\theta = 0.7$. (b) SCAD with $\theta = 0.7$. (c) LSP with $\theta =
  0.8$. (d) MCP with $\theta = 0.8$. The SCAD and MCP regularizers
  clearly give rise to multiple distinct stationary points, agreeing
  with our predictions.}
\label{FigOLSUnique}
\end{figure}
\end{center}

Finally, we present a third set of simulations, analogous to the
second, with the OLS loss function replaced by the maximum likelihood loss
function for logistic regression:
\begin{equation*}
	\Loss_\numobs(\beta) = \frac{1}{n} \sum_{i=1}^n
        \left\{\log(1+\exp(x_i^T \beta)) - y_i x_i^T \beta\right\}.
\end{equation*}
We generated $x_i \sim N(0, \sigma_x^2 I)$, with $\sigma_x \in \{1,
3\}$, and set the problem dimensions to be $p = 128$, $k \approx
\sqrt{p}$, and $n \approx 10 k^3 \log p$. We generated $\betastar$ to
have $k$ nonzero values $\pm \frac{1}{\sqrt{k}}$ with equal
probability for each sign, and generated response variables $y_i \in
\{0,1\}$ according to
\begin{equation*}
\mprob(y_i = 1 \mid x_i, \betastar) = \frac{\exp(x_i^T \betastar)}{1 +
  \exp(x_i^T \betastar)}.
\end{equation*}
Figure~\ref{FigLogisticUnique} shows the results of our
simulations. Panels (a)--(d) plot the log $\ell_2$-error as a function
of iteration number, when $\sigma_x = 1$. Note that in this case, an
empirical evaluation shows that $\lambda_{\min}(\nabla^2
\Loss(\betastar)) \approx 0.14$, so we expect $\alpha_1 \approx 0.14$
and $\mu \nless 2\alpha_1$. As in the plots of
Figure~\ref{FigOLSUnique}, multiple stationary points emerge in panels
(c) and (d) when $\rho_\lambda$ is the SCAD or MCP regularizer; in contrast, we see from panels (a) and (b) that all 15 runs of composite
gradient descent converge to the same stationary point when
$\rho_\lambda$ is the $\ell_1$-penalty or LSP. In panels (e) and (f),
we repeat the simulations with $\sigma_x = 3$. In this case,
$\lambda_{\min}(\nabla^2 \Loss(\betastar)) \approx 0.25$, and we see
from our plots that although the condition $\mu < 2\alpha_1$ is still
violated, the larger value of $\alpha_1$ is enough to make the
stationary points under SCAD or MCP regularization unique. We may
again observe the geometric rate of convergence of the $\ell_2$-error
$\|\beta^t - \betastar\|_2$ in each plot, up to a certain small
threshold. The improved performance from using the SCAD and MCP
regularizers may be observed empirically by comparing the vertical
axes in the panels of Figure~\ref{FigLogisticUnique}.

\begin{center}
\begin{figure}
\begin{tabular}{cc}
\includegraphics[width=7cm]{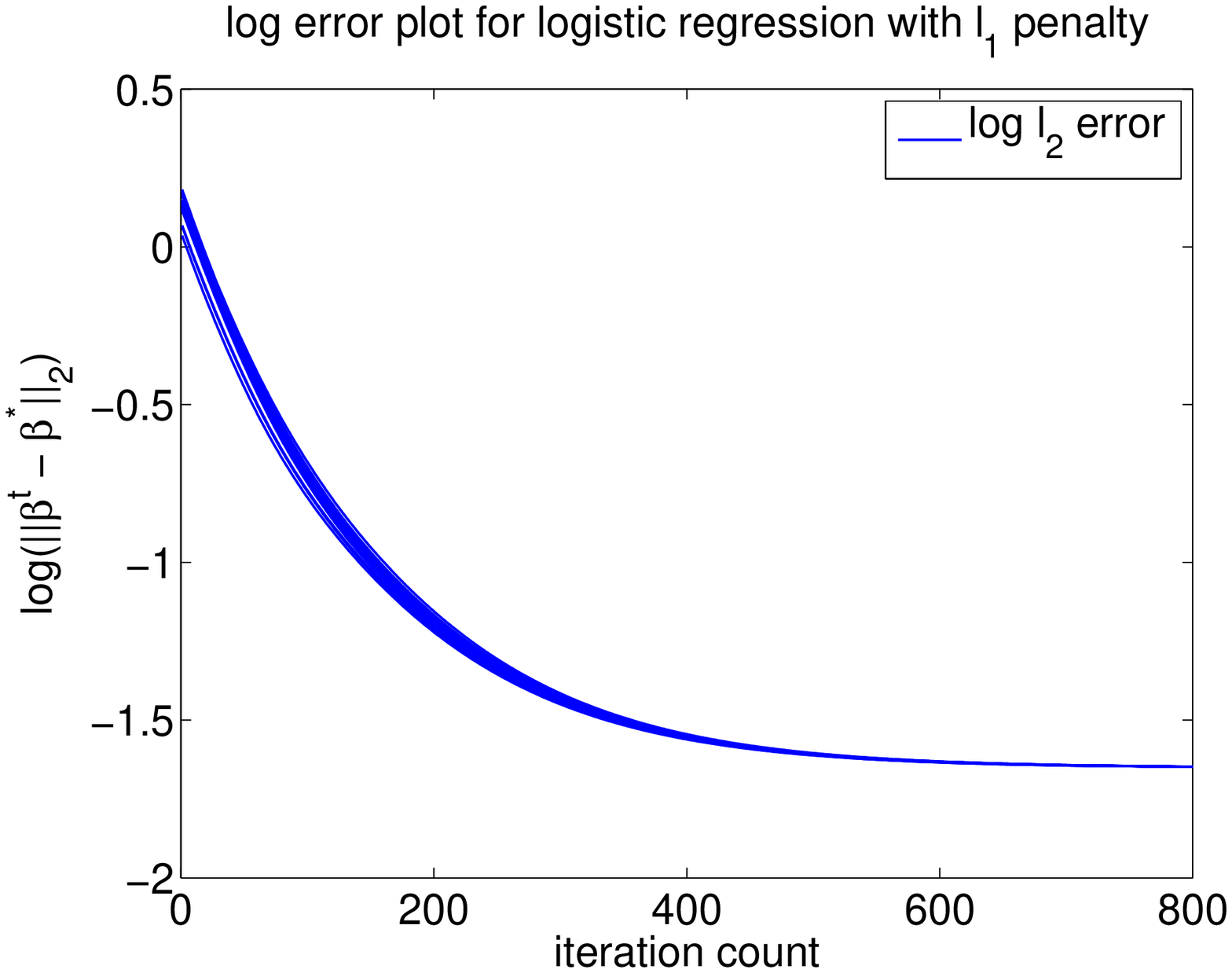} &
\includegraphics[width=7cm]{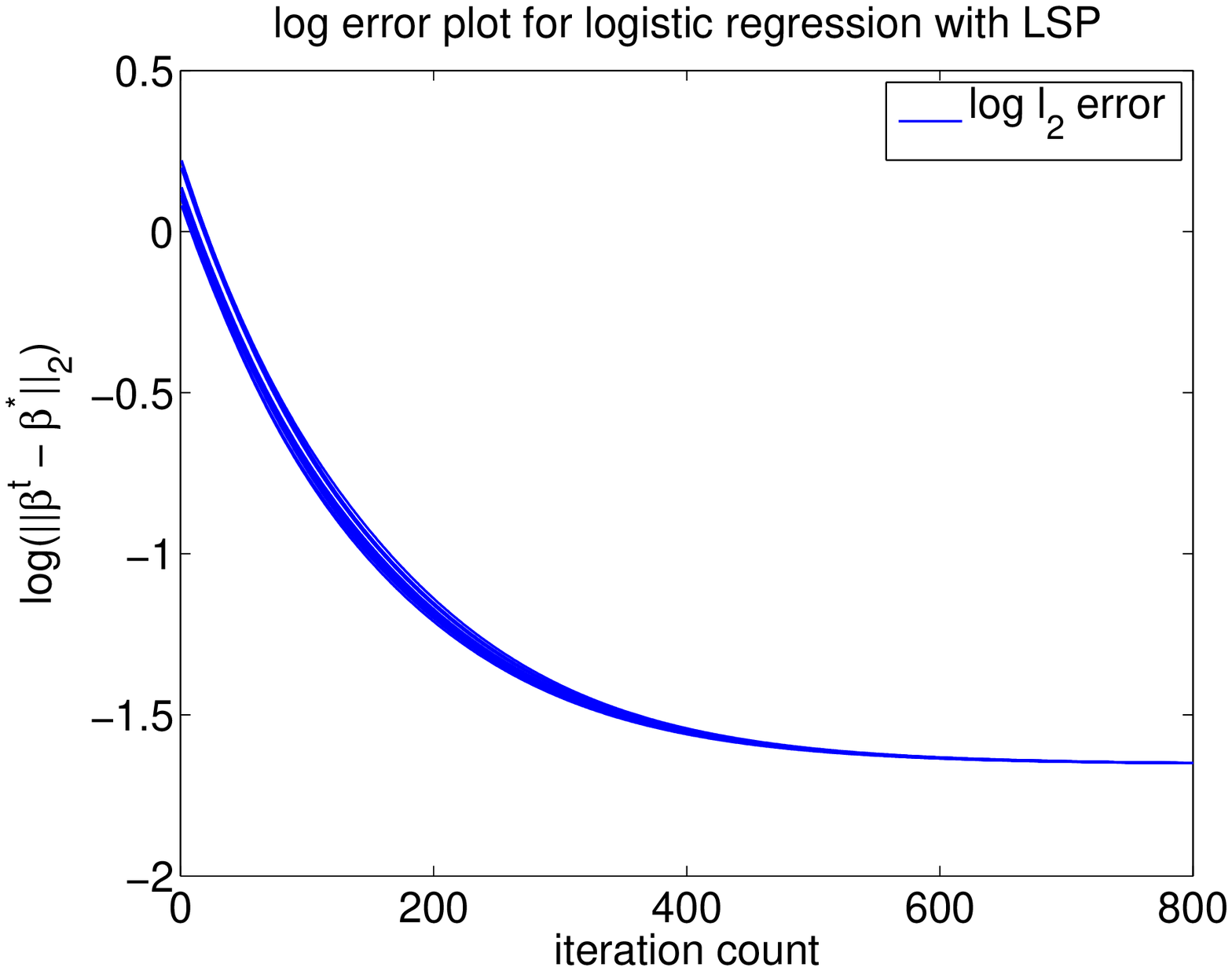} \\
	(a) & (b) \\
	\includegraphics[width=7cm]{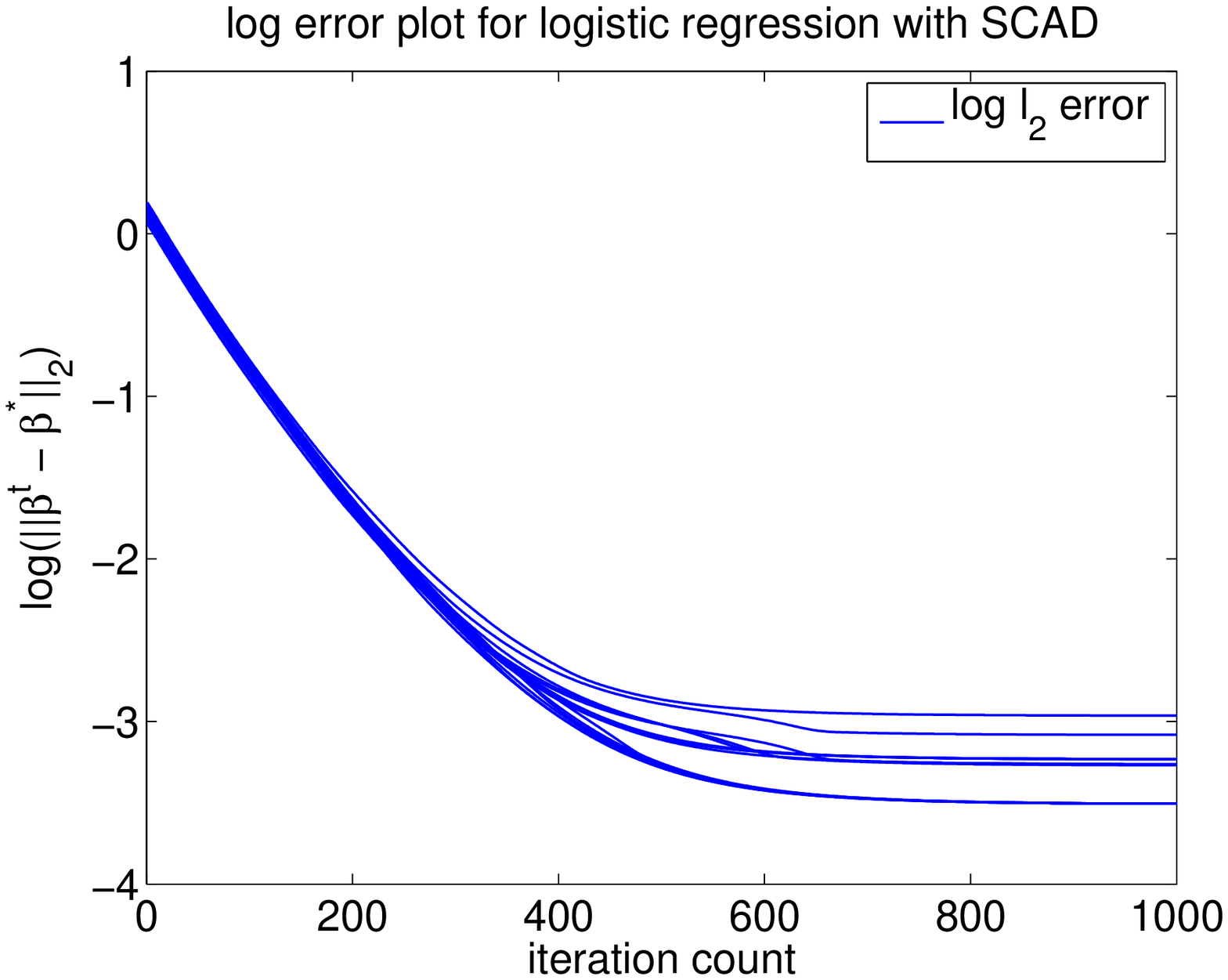} &
        \includegraphics[width=7cm]{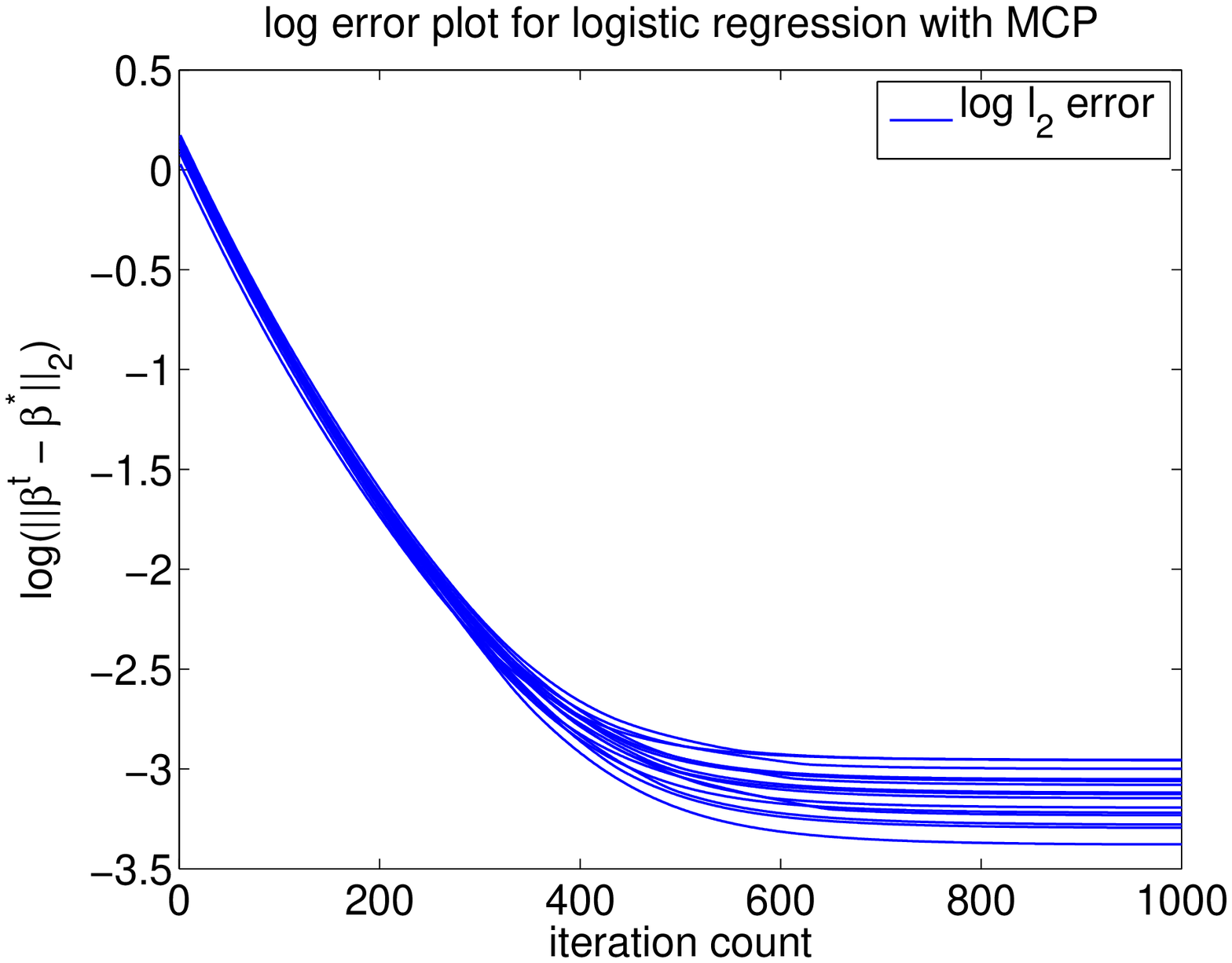} \\
	(c) & (d) \\
	\includegraphics[width=7cm]{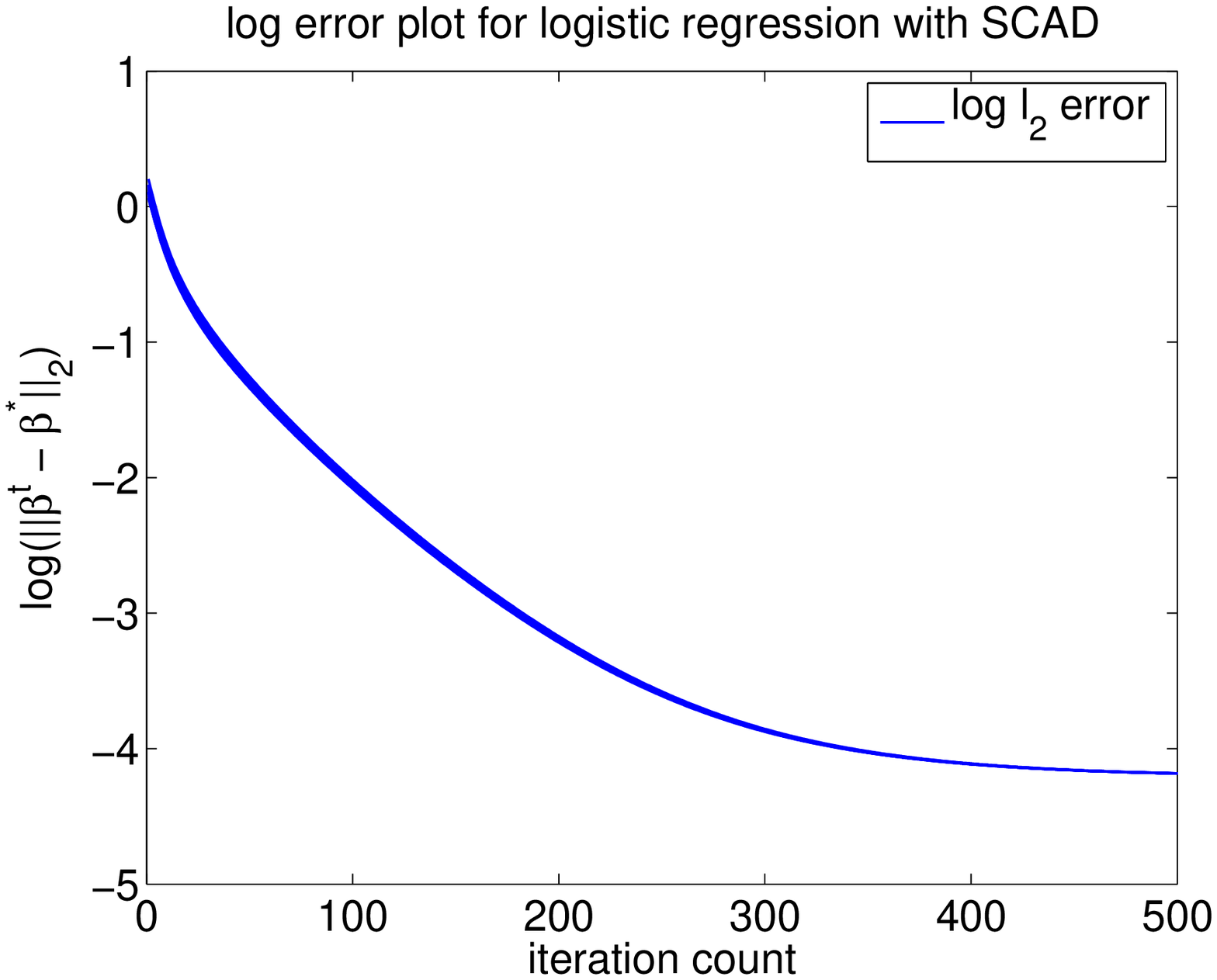} & \includegraphics[width=7cm]{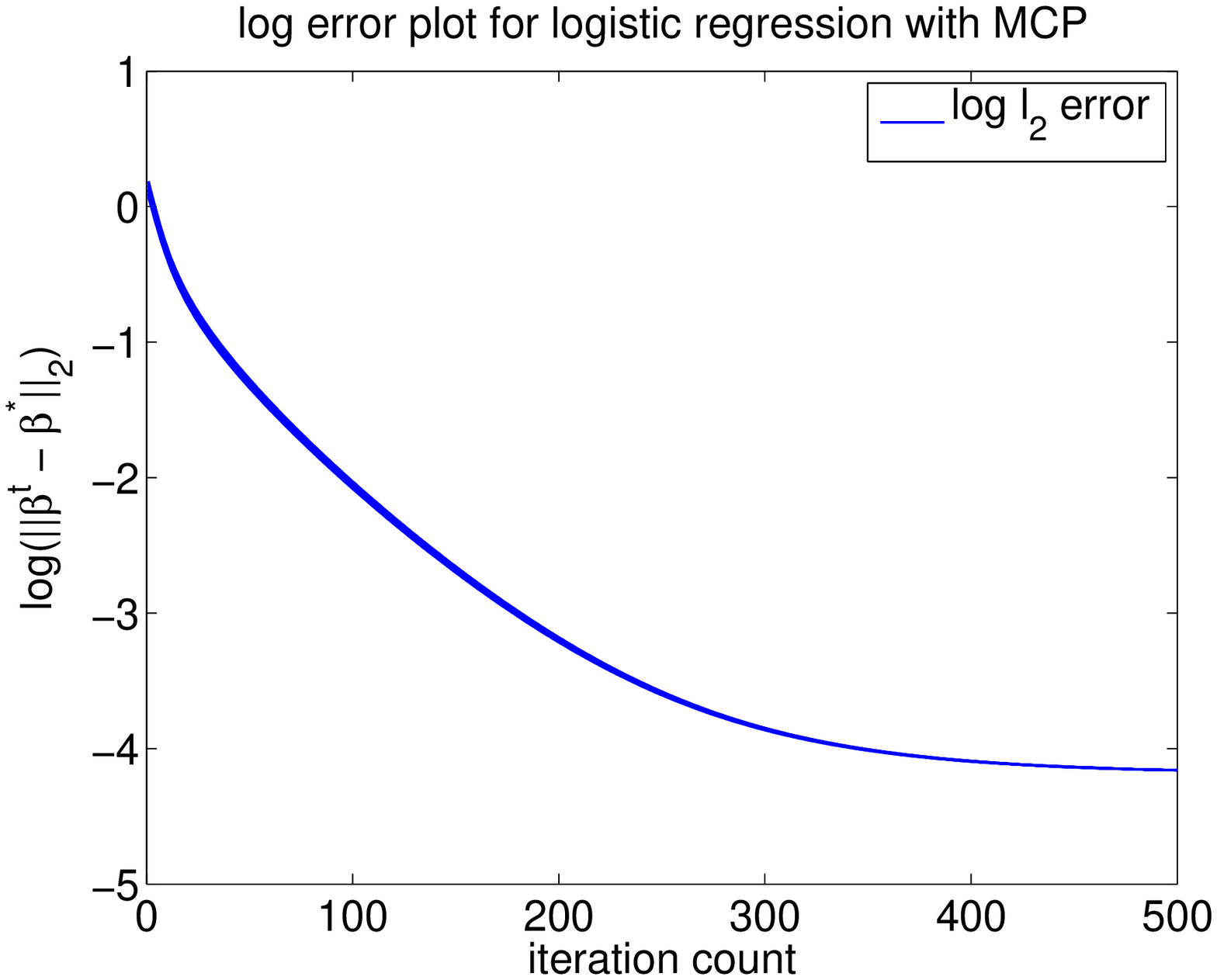} \\
	(e) & (f)
\end{tabular}
\caption{Plots showing log $\ell_2$-error $\log\left(\|\beta^t -
  \betastar\|_2\right)$ as a function of iteration number $t$ for
  logistic regression, with a variety of regularizers and 15 random
  initializations of composite gradient descent. The covariates are
  normally distributed according to $x_i \sim N(0, \sigma_x^2 I)$,
  with $\sigma_x = 1$ in plots (a)--(d), and $\sigma_x = 3$ in plots
  (e)--(f). In panels (c) and (d), the composite gradient descent
  algorithm settles into multiple distinct stationary points, which
  exist because $2\alpha_1 < \mu$ for the SCAD and MCP. However, when
  the covariates have a larger covariance, the SCAD and MCP
  regularizers produce unique stationary points, as observed in panels
  (e) and (f).}
\label{FigLogisticUnique}
\end{figure}
\end{center}


\section{Discussion}

We have developed an extended framework for analyzing a variety of
nonconvex problems via the primal-dual witness proof technique. Our
results apply to composite optimization programs where both the loss
and regularizer function are allowed to be nonconvex, and our analysis
significantly generalizes the machinery previously established to
study convex objective functions. As a consequence, we have provided a
powerful reason for using nonconvex regularizers such as the SCAD and
MCP rather than the convex $\ell_1$-penalty: In addition to being
consistent in $\ell_2$-error, the nonconvex regularizers actually
produce an overall estimator that is consistent for support recovery
when the design matrix is non-incoherent and the usual
$\ell_1$-regularized program fails in recovering the correct
support. We have also established a similar strong result for the graphical
Lasso objective function with nonconvex regularizers, which eliminates
the need for complicated incoherence conditions on the inverse
covariance matrix.

Future directions of research include devising theoretical guarantees
when the condition $\mu < 2\alpha_1$ is only mildly violated, since
the condition does not appear to be strictly necessary based on our
simulations, and establishing a rigorous justification for why the
SCAD and MCP regularizers perform appreciably better than the
$\ell_1$-penalty even in terms of $\ell_2$-error, in situations where
the assumptions are not strong enough for an oracle result to
apply. It is also an open question as to how generally the restricted
strong convexity condition may hold for various other nonconvex loss
functions of interest, or whether a local RSC condition is sufficient
to guarantee good behavior with proper initializations of a gradient descent algorithm. Finally, it
would be useful to be able to compute the RSC constants $(\alpha_1,
\alpha_2)$ empirically from data, so as to assign a nonconvex
regularizer with the proper amount of curvature.


\section*{Acknowledgments}
The work of PL was partly supported from a Hertz Foundation
Fellowship, an NSF Graduate Research Fellowship and a PD Fellowship
while studying at Berkeley. MJW and PL were also partially supported
by grants NSF grant DMS-1107000, NSF grant CIF-31712-23800, and Air
Force Office of Scientific Research Grant AFOSR-FA9550-14-1-0016.


\appendix



\section{Proof of Theorem~\ref{ThmSuppRecovery}}
\label{SecSuppRecovery}

In this Appendix, we provide the proof of
Theorem~\ref{ThmSuppRecovery}.  We begin with the main body of the
argument, with the proofs of some more technical lemmas deferred to
later subsections.

\subsection{Main part of proof}

We follow the outline of the primal-dual witness construction
described in Section~\ref{SecPDW}.  For step (i) of the construction,
we use Lemma~\ref{LemEll2Bounds} in Appendix~\ref{AppSupporting},
where we simply replace $p$ by $k$ and $\Loss_\numobs$ by $\LossSn$,
which is the function $\Loss_\numobs$ restricted to $\real^S$. It
follows that as long as $n \ge \frac{16 R^2 \max(\tau_1^2,
  \tau_2^2)}{\alpha_2^2} \log k$, we are guaranteed that $\|\betahat_S
- \betastar_S\|_1 \le \frac{28\lambda k}{2\alpha_1 - \mu}$, whence
\begin{equation*}
\|\betahat_S\|_1 \le \|\betastar\|_1 + \|\betahat_S - \betastar_S\|_1
\leq \frac{R}{2} + \frac{28\lambda k}{2 \alpha_1 - \mu} < R.
\end{equation*}
Here, the final inequality follows by the lower bound in
inequality~\eqref{EqnRCond}. We conclude that $\betahat_S$ must be in
the interior of the feasible region.

Moving to step (ii) of the PDW construction, we define the shifted
objective function $\Lossbar_n(\beta) \defn \Loss_\numobs(\beta) -
q_\lambda(\beta)$.  Since $\betahat_S$ is an interior point, it must
be a zero-subgradient point for the restricted
program~\eqref{EqnRestrictedM}, so $\nabla \LossSnbar (\betahat_S) +
\lambda \zhat_S = 0$, where $\zhat_S \in \partial \|\betahat_S\|_1$ is
the dual vector.  By the chain rule, this implies that $\left(\nabla
\Lossbar_n(\betahat)\right)_S + \lambda \zhat_S = 0$, where $\betahat
\defn (\betahat_S, 0_{S^c})$.  Accordingly, we may define the
subvector $\zhat_{S^c} \in \real^{S^c}$ such that
\begin{equation}
\label{EqnZero}
\nabla \Lossbar_n(\betahat) + \lambda \zhat = 0,
\end{equation}
where $\zhat \defn (\zhat_S, \zhat_{S^c})$ is the extended
subgradient.  Under the assumption~\eqref{EqnStrictDual}, this
completes step (ii) of the construction.

For step (iii), we first establish that $\betahat$ is a local minimum
for the program~\eqref{EqnMEst} by verifying the sufficient conditions
of Lemma~\ref{LemSecondOrder} in Appendix~\ref{AppSupporting}, with $f
= \Lossbar_n$, $g = q_\lambda$, and $(x^*, v^*, w^*, \mu^*) =
(\betahat, \zhat, \zhat, 0)$. Note that Lemma~\ref{LemSmiley}(b) from
Appendix~\ref{AppAmenable} ensures the concavity and differentiability
of $g(x) - \frac{\mu}{2} \|x\|_2^2$. Since $\mu^* = 0$,
condition~\eqref{EqnCompSlack} is trivially satisfied. Furthermore,
condition~\eqref{EqnZeroGrad} holds by
equation~\eqref{EqnZero}. Hence, it remains to verify the
condition~\eqref{EqnStrictConvex}.

We first show that $G^* \subseteq \real^S$. Supposing the contrary,
consider a vector $\specvec \in G^*$ such that $\supp(\specvec)
\subsetneq S$. Fixing some index $j \in S^c$ such that $\specvec_j
\neq 0$, by the definition of $G^*$, we have
\begin{equation}
\label{EqnGDefn}
\sup_{v \in \partial \|\betahat\|_1} \specvec^T (\nabla
\Lossbar_n(\betahat) + \lambda v) = 0.
\end{equation}
However, if $\ztil$ denotes the vector $\zhat$ with entry $j$
replaced by $\sign(s_j) \in \{-1, 1\}$, we clearly still have $\ztil
\in \partial \|\betahat\|_1$, but
\begin{equation*}
\specvec^T(\nabla \Lossbar_n(\betahat) + \lambda \ztil) > \specvec^T
(\nabla \Lossbar_n(\betahat) + \lambda \zhat) = 0,
\end{equation*}
where the strict inequality holds because $\|\zhat_{S^c}\|_\infty < 1$,
by our assumption. We have thus obtained a contradiction to
equation~\eqref{EqnGDefn}; consequently, our initial assumption was
false, and we may conclude that $G^* \subseteq \real^S$.  \\

 The following lemma, proved Appendix~\ref{AppLemSecOrder}, guarantees
 that a shifted form of the loss function is strictly convex over an
 $|S|$-dimensional subspace:
\begin{lem*}
\label{LemSecOrder}
Consider any $(\alpha,\tau)$-RSC loss function $\Loss_\numobs$ and
$\mu$-amenable regularizer $\rho_\lambda$, with $\mu < \alpha_1$. If
$\numobs \ge \frac{2\tau_1}{\alpha_1 - \mu} k \log p$, the function
\mbox{$\Loss_\numobs(\beta) - \frac{\mu}{2} \|\beta\|_2^2$ } is
strictly convex on $\beta \in \real^S$, and the restricted
program~\eqref{EqnRestrictedM} is also strictly convex.
\end{lem*}
\noindent In particular, since $G^* \subseteq S$ and $\supp(\betahat)
\subseteq S$, Lemma~\ref{LemSecOrder} immediately implies
condition~\eqref{EqnStrictConvex} of Lemma~\ref{LemSecondOrder}. We
conclude that $\betahat$ is indeed a local minimum of the
program~\eqref{EqnMEst}. \\

The following lemma, proved in Appendix~\ref{AppStationBd}, show that
all stationary points of the program~\eqref{EqnMEst} are supported on
$S$:
\begin{lem*}
\label{LemStationBd}
Suppose $\betatil$ is a stationary point of the
program~\eqref{EqnMEst} and the conditions of Theorem~\ref{ThmSuppRecovery} hold. Then $\supp(\betatil)
\subseteq S$.
\end{lem*}

Turning to the uniqueness assertion, note that since all stationary
points are supported in $S$, any stationary point $\betatil$ of the
program~\eqref{EqnMEst} must satisfy $\betatil = (\betatil_S,
0_{S^c})$, where $\betatil_S$ is a stationary point of the restricted
program~\eqref{EqnRestrictedM}. By Lemma~\ref{LemSecOrder}, the
restricted program is strictly convex. Hence, the vector $\betatil_S$,
and consequently also $\betatil$, is unique.



\subsection{Proof of Lemma~\ref{LemSecOrder}}
\label{AppLemSecOrder}

We begin by establishing the bound $\left(\nabla^2
\Loss_\numobs(\beta)\right)_{SS} \succeq \left(\alpha_1 - \tau_1
\frac{\kdim \log \pdim}{\numobs}\right) I$, for all $\beta \in
\real^\pdim$. Equivalently,
\begin{equation}
\label{EqnLychee}
v^T \left(\nabla^2 \Loss_\numobs(\beta)\right) v \ge \left(\alpha_1 -
\tau_1 \frac{k \log p}{n}\right), \qquad \mbox{$\forall v \in \big \{ v \in
  \real^p \mid \supp(v) \subseteq S, \|v\|_2 = 1 \big \}$.}
\end{equation}
When this lower bound holds, we are guaranteed that
\begin{equation*}
\left(\nabla^2 \left(\Loss_\numobs(\beta) - \frac{\mu}{2}
\|\beta\|_2^2\right)\right)_{SS} = \left(\nabla^2
\Loss_\numobs(\beta)\right)_{SS} - \mu I \succeq \left(\alpha_1 - \mu
- \tau_1 \frac{k \log p}{n}\right)I,
\end{equation*}
so $\Loss_\numobs(\beta) - \frac{\mu}{2} \|\beta\|_2^2$ is strictly convex
on $\real^S$ under the prescribed sample size.
	
In order to prove the bound~\eqref{EqnLychee}, consider a fixed $v \in
\real^p$ such that $\supp(v) \subseteq S$ and $\|v\|_2 = 1$. For this
vector, we have $\nabla^2 \Loss_\numobs(\beta) v = \lim \limits_{t
  \rightarrow 0} \left\{\frac{\nabla \Loss_\numobs(\beta + tv) -
  \nabla \Loss_\numobs(\beta)}{t}\right\}$, so
\begin{align}
\label{EqnBaklava}
v^T \left(\nabla^2 \Loss_\numobs(\beta)\right) v = \lim_{t \rightarrow 0}
\left\{\frac{\inprod{\nabla \Loss_\numobs(\beta + tv) - \nabla
    \Loss_\numobs(\beta)}{tv}}{t^2}\right\}.
\end{align}
Furthermore, by the RSC assumption~\eqref{EqnRSC} and the fact that $\|v\|_2 =
1$, we have
\begin{equation}
\label{EqnTorte}
\inprod{\nabla \Loss_\numobs(\beta + tv) - \nabla \Loss_\numobs(\beta)}{tv} \ge
t^2\left( \alpha_1 \|v\|_2^2 - \tau_1 \frac{\log p}{n}
\|v\|_1^2\right), \qquad \text{for } t \le 1.
\end{equation}
Since $\supp(v) \subseteq S$, we also have $\|v\|_1 \le \sqrt{k}
\|v\|_2$. Combining this bound with equations~\eqref{EqnBaklava}
and~\eqref{EqnTorte} then gives the desired
inequality~\eqref{EqnLychee}.

Finally, we note the decomposition $\Lossbar_n(\beta) =
\left(\Loss_\numobs(\beta) - \frac{\mu}{2} \|\beta\|_2^2\right) +
\left(\frac{\mu}{2} \|\beta\|_2^2 - q_\lambda(\beta)\right)$, showing
that $\Lossbar_n$ is the sum of a strictly convex and convex function
over $\real^S$. Hence, $\Lossbar_n \mid_S$ is strictly convex, as
claimed. The strict convexity of $\Lossbar_n(\beta) + \lambda
\|\beta\|_1$ over $\real^S$ follows immediately.


\subsection{Proof of Lemma~\ref{LemStationBd}}
\label{AppStationBd}

Let $\nutil \defn \betatil - \betastar$. We first show that
$\|\nutil\|_2 \le 1$. Suppose on the contrary that $\|\nutil\|_2 >
1$. By inequality~\eqref{EqnL2Bd}, we have $\inprod{\nabla
  \Loss_\numobs(\betatil) - \nabla \Loss_\numobs(\betahat)}{\nutil}
\ge \alpha_2 \|\nutil\|_2 - \tau_2 \sqrt{\frac{\log p}{n}}
\|\nutil\|_1$.  Moreover, since $\betahat$ is feasible, the
first-order optimality condition gives 
\begin{align}
\label{EqnKiwi}
0 \le \inprod{\nabla \Loss_\numobs(\betatil) + \nabla
  \rho_\lambda(\betatil)}{\betahat - \betatil}.
\end{align}
Summing the two preceding inequalities yields
\begin{equation}
\label{EqnLemon}
\alpha_2 \|\nutil\|_2 - \tau_2 \sqrt{\frac{\log p}{n}} \|\nutil\|_1
\le \inprod{- \nabla \Loss_\numobs(\betahat) - \nabla
  \rho_\lambda(\betatil)}{\nutil}.
\end{equation}
Since $\betahat$ is an interior local minimum, we have $\nabla
\Loss_\numobs(\betahat) + \nabla \rho_\lambda(\betahat) = 0$. Hence,
inequality~\eqref{EqnLemon} implies that
\begin{align*}
\alpha_2 \|\nutil\|_2 - \tau_2 \sqrt{\frac{\log p}{n}} \|\nutil\|_1 \,
\le \, \inprod{\nabla \rho_\lambda(\betahat) - \nabla
  \rho_\lambda(\betatil)}{\nutil} & \le \left(\|\nabla
\rho_\lambda(\betahat)\|_\infty + \|\nabla
\rho_\lambda(\betatil)\|_\infty\right) \|\nutil\|_1 \\
& \le 2 \lambda \|\nutil\|_1,
\end{align*}
where the bound $\|\nabla \rho_\lambda(\beta)\|_\infty \le \lambda$
holds by Lemma~\ref{LemSmiley} in
Appendix~\ref{AppAmenable}. Rearranging, we then have
\begin{equation*}
\|\nutil\|_2 \le \frac{\|\nutil\|_1}{\alpha_2} \left( 2 \lambda +
\tau_2 \sqrt{\frac{\log p}{n}} \right) \leq \frac{2R}{\alpha_2} \left
( 2 \lambda + \tau_2 \sqrt{\frac{\log p}{n}} \right).
\end{equation*}
Since $\lambda \le \frac{\alpha_2}{8R}$ and $n \ge \frac{16R^2
  \tau_2^2}{\alpha_2^2} \log p$ by assumption, this implies that
\mbox{$\|\nutil\|_2 \le 1$,} as claimed.

Now, applying the RSC condition~\eqref{EqnL2Bd}, we have
\begin{equation*}
\inprod{\nabla \Loss_\numobs(\betatil) - \nabla
  \Loss_\numobs(\betahat)}{\nutil} \ge \alpha_1 \|\nutil\|_2^2 -
\tau_1 \frac{\log p}{n} \|\nutil\|_1^2,
\end{equation*}
which implies that
\begin{equation}
\label{EqnGuava}
\inprod{\nabla \Lossbar_n(\betatil) - \nabla
  \Lossbar_n(\betahat)}{\nutil} \ge (\alpha_1 - \mu) \|\nutil\|_2^2 -
\tau_1 \frac{\log p}{n} \|\nutil\|_1^2.
\end{equation}
By inequality~\eqref{EqnKiwi}, we also have
\begin{equation}
\label{EqnDonut}
0 \le \inprod{\nabla \Lossbar_n(\betatil)}{\betahat - \betatil} +
\lambda \cdot \inprod{\ztil}{\betahat - \betatil},
\end{equation}
where $\ztil \in \partial \|\betatil\|_1$. From the zero-subgradient
condition~\eqref{EqnZeroSub}, we have $\inprod{\nabla
  \Lossbar_n(\betahat) + \lambda \zhat}{\betatil - \betahat} = 0$.
Combining with inequality~\eqref{EqnDonut} then yields
\begin{equation}
\label{EqnRambutan}
0 \le \inprod{\nabla \Lossbar_n(\betahat) - \nabla
  \Lossbar_n(\betatil)}{\betatil - \betahat} + \lambda \cdot
\inprod{\zhat}{\betatil} - \lambda \|\betahat\|_1 + \lambda \cdot
\inprod{\ztil}{\betahat} - \lambda \|\betatil\|_1.
\end{equation}
Rearranging, we have
\begin{align}
\label{EqnRaspberry}
\lambda \|\betatil\|_1 - \lambda \cdot \inprod{\zhat}{\betatil} & \le
\inprod{\nabla \Lossbar_n(\betahat) - \nabla
  \Lossbar_n(\betatil)}{\betatil - \betahat} + \lambda \cdot
\inprod{\ztil}{\betahat} - \lambda \|\betahat\|_1 \notag \\
& \le \inprod{\nabla \Lossbar_n(\betahat) - \nabla
  \Lossbar_n(\betatil)}{\betatil - \betahat} \notag \\
& \le \tau_1 \frac{\log p}{n} \|\nutil\|_1^2 - (\alpha_1-\mu)
\|\nutil\|_2^2,
\end{align}
where the second inequality comes from the fact that
$\inprod{\ztil}{\betahat} \le \|\ztil\|_\infty \cdot \|\betahat\|_1
\le \|\betahat\|_1$, and the third inequality comes from the
bound~\eqref{EqnGuava}. Finally, we need a lemma showing that $\nutil$
lies in a cone set:
\begin{lem*}
\label{LemNuCone}
If $\|\zhat_{S^c}\|_\infty \le 1 - \delta$ for some $\delta \in (0,1]$
and $\lambda \ge \frac{4 R \tau_1 \log p}{\delta n}$, then
\begin{equation*}
\|\nutil\|_1 \le \left(\frac{4}{\delta} + 2\right) \sqrt{k}
\|\nutil\|_2.
\end{equation*}
\end{lem*}

\begin{proof}
From inequality~\eqref{EqnRambutan} together with the RSC
bound~\eqref{EqnGuava}, we have
\begin{subequations}
\begin{equation}
\label{EqnPear}
(\alpha_1 - \mu) \|\nutil\|_2^2 - \tau_1 \frac{\log p}{n}
\|\nutil\|_1^2 \le \inprod{\nabla \Lossbar_n(\betatil) - \nabla
  \Lossbar_n(\betahat)}{\nutil} \le \lambda \cdot
\inprod{\ztil}{\betahat} - \lambda \|\betatil\|_1 + \lambda \cdot
\inprod{\zhat}{\nutil}.
\end{equation}
Note that since $\supp(\betahat) \subseteq S$, we have
\begin{equation}
\label{EqnMilk}
\lambda \cdot \inprod{\ztil}{\betahat} - \lambda \|\betatil\|_1 \le
\lambda \|\betahat\|_1 - \lambda \|\betatil\|_1 = \lambda
\left(\|\betahat_S\|_1 - \|\betatil_S\|_1 -
\|\betatil_{S^c}\|_1\right) \le \lambda \left(\|\nutil_S\|_1 -
\|\nutil_{S^c}\|_1\right).
\end{equation}
Furthermore, we have
\begin{align}
\label{EqnJuice}
\lambda \cdot \inprod{\zhat}{\nutil} =
\lambda\left(\inprod{\zhat_S}{\nutil_S} +
\inprod{\zhat_{S^c}}{\nutil_{S^c}}\right) & \le \lambda
\left(\|\zhat_S\|_\infty \cdot \|\nutil_S\|_1 + \|\zhat_{S^c}\|_\infty
\cdot \|\nutil_{S^c}\|_1\right) \notag \\
& \le \lambda \left(\|\nutil_S\|_1 + (1-\delta) \cdot
\|\nutil_{S^c}\|_1\right).
\end{align}
\end{subequations}
Combining inequalities~\eqref{EqnPear}, \eqref{EqnMilk},
and~\eqref{EqnJuice} then yields
\begin{equation}
\label{EqnBlueberry}
-\tau_1 \frac{\log p}{n} \|\nutil\|_1^2 \le (\alpha_1 - \mu)
\|\nutil\|_2^2 - \tau_1 \frac{\log p}{n} \|\nutil\|_1^2 \le \lambda
\left(2 \|\nutil_S\|_1 - \delta \|\nutil_{S^c}\|_1\right).
\end{equation}
Under the assumption on $\lambda$, we have $\tau_1 \frac{\log p}{n}
\|\nutil\|_1 \le 2R \tau_1 \frac{\log p}{n} \le \frac{\delta}{2} \cdot
\lambda$, so that inequality~\eqref{EqnBlueberry} implies
\begin{equation*}
-\frac{\delta}{2} \cdot \lambda \|\nutil\|_1 \le \lambda \left(2
\|\nutil_S\|_1 - \delta \|\nutil_{S^c}\|_1\right), \qquad \mbox{or
  equivalently}, \qquad \frac{\delta}{2} \|\nutil_{S^c}\|_1 \le \left(2
+ \frac{\delta}{2} \right) \|\nutil_S\|_1.
\end{equation*}
Putting together the pieces, we then have
\begin{equation*}
\|\nutil\|_1 = \|\nutil_S\|_1 + \|\nutil_{S^c}\|_1 \le \|\nutil_S\|_1
+ \left(\frac{4}{\delta} + 1\right) \|\nutil_S\|_1 \le
\left(\frac{4}{\delta} + 2\right) \sqrt{k} \|\nutil\|_2,
\end{equation*}
which establishes the claim.
\end{proof}

Combining Lemma~\ref{LemNuCone} with inequality~\eqref{EqnRaspberry} then gives
\begin{equation*}
\lambda \|\betatil\|_1 - \lambda \cdot \inprod{\zhat}{\betatil} \le
\tau_1 \frac{k \log p}{n} \left(\frac{4}{\delta} + 2\right)^2
\|\nutil\|_2^2 - (\alpha_1-\mu) \|\nutil\|_2^2.
\end{equation*}
Hence, if $n \ge \frac{2\tau_1}{\alpha_1-\mu} \left(\frac{4}{\delta} +
2\right)^2 k \log p$, we have $\lambda \|\betatil\|_1 - \lambda \cdot
\inprod{\zhat}{\betatil} \le -\frac{\alpha_1-\mu}{2} \|\nutil\|_2^2
\le 0$.  On the other hand, H\"{o}lder's inequality gives $\lambda
\cdot \inprod{\zhat}{\betatil} \le \lambda \|\zhat\|_\infty \cdot
\|\betatil\|_1 \le \lambda \|\betatil\|_1$.  It follows that we must
have $\inprod{\zhat}{\betatil} = \|\betatil\|_1$.  Since
$\|\zhat_{S^c}\|_\infty < 1$, we conclude that $\betatil_j = 0$, for
all $j \notin S$. Hence, $\supp(\betatil) \subseteq S$, as
claimed.


\section{Proof of Theorem~\ref{ThmEllInf}}
\label{SecThmEllInf}

Note that by the fundamental theorem of calculus, $\Qhat$ satisfies
\begin{equation*}
\Qhat (\betahat - \betastar) = \nabla \Loss_\numobs(\betahat) - \nabla \Loss_\numobs(\betastar).
\end{equation*}
By the zero-subgradient condition~\eqref{EqnZeroSub} and the
invertibility of $\left(\nabla^2 \Loss_\numobs(\betatil)_{SS}\right)$,
we then have
\begin{equation*}
\betahat_S - \betastar_S = \left(\Qhat_{SS}\right)^{-1} \left(- \nabla
\Loss_\numobs(\betastar)_S + \nabla q_\lambda(\betahat_S) - \lambda
\zhat_S\right),
\end{equation*}
implying that
\begin{equation}
\label{EqnGrapefruit}
\|\betahat - \betastar\|_\infty \le \left\|\left(\Qhat_{SS}\right)^{-1} \left(\nabla
\Loss_\numobs(\betastar)_S - \nabla q_{\lambda}(\betahat_S) + \lambda
\zhat_S\right)\right\|_\infty.
\end{equation}
Lemma~\ref{LemSmiley} in Appendix~\ref{AppAmenable} guarantees that
$\left\|\left(\nabla q_\lambda(\betahat_S) - \lambda
\zhat_S\right)\right\|_\infty \le \lambda$, so
\begin{align*}
\|\betahat - \betastar\|_\infty & \le \left\|\left(\Qhat_{SS}\right)^{-1} \nabla
\Loss_\numobs(\betastar)_S\right\|_\infty + \left\|\left(\Qhat_{SS}\right)^{-1} \left(\nabla q_\lambda(\betahat_S)
- \lambda \zhat_S\right)\right\|_\infty \\
& \le \left\|\left(\Qhat_{SS}\right)^{-1} \nabla \Loss_\numobs(\betastar)_S
\right\|_\infty + \lambda \opnorm{\left(\Qhat_{SS}\right)^{-1}}_\infty,
\end{align*}
which is inequality~\eqref{EqnHoneydew}. \\
 
\noindent To establish inequality~\eqref{EqnWatermelon}, we use the
following lemma:
\begin{lem*}
\label{LemSInfty}
Suppose $\rho_\lambda$ is $(\mu, \gamma)$-amenable, and the
bound~\eqref{EqnLavender} holds. Then $|\betahat_j| \ge \gamma
\lambda$ for all $j \in S$, and in particular, $q'_\lambda(\betahat_j)
= \lambda \cdot \sign(\betahat_j)$.
\end{lem*}

\begin{proof}
	
By inequality~\eqref{EqnHoneydew} and the triangle inequality, we have
\begin{align*}
|\betahat_j| \ge |\betastar_j| - |\betahat_j - \betastar_j| & \ge
\betamin - \|\betahat - \betastar\|_\infty \\
& \ge \betamin - \left(\left\|\left(\Qhat_{SS} \right)^{-1} \nabla
\Loss_\numobs(\betastar)_S\right\|_\infty + \lambda
\opnorm{\left(\Qhat_{SS}\right)^{-1}}_\infty\right),
\end{align*}
for all $j \in S$. Hence, if the bound~\eqref{EqnLavender} holds, we
must have $|\betahat_j| \ge \gamma \lambda$, and the desired results
follow.
\end{proof}

Lemma~\ref{LemSInfty} implies that $\nabla q_\lambda(\betahat_S) =
\lambda \zhat_S$.  Hence, the zero-subgradient
condition~\eqref{EqnZeroSub} reduces to $\left(\nabla \Loss_\numobs
(\betahat_S) \right) \Big |_S = 0$.  Since $\Loss_\numobs$ is strictly
convex on $\real^S$ by Lemma~\ref{LemSecOrder}, this zero-gradient
condition implies that $\betahat_S$ is the unique global minimum of
$\LossSn$, so we have the equalities \mbox{$\betahat_S = \boracle_S$}
and \mbox{$\betahat = \boracle$,} as claimed. Finally,
inequality~\eqref{EqnGrapefruit} simplifies to
inequality~\eqref{EqnWatermelon}, using the equality \mbox{$\nabla
  q_\lambda(\betahat_S) = \lambda \zhat_S$.}


\section{Establishing strict dual feasibility}
\label{AppDual}

In this Appendix, we derive conditions that allow us to establish the
strict dual feasibility conditions required in order to apply
Theorem~\ref{ThmSuppRecovery}. We use these derivations to prove
Corollaries~\ref{CorOLSLoss}--\ref{CorGLMLoss}.

By the zero-subgradient condition~\eqref{EqnZeroSub}, we have
\begin{equation*}
\left(\nabla \Loss_\numobs(\betahat) - \nabla
\Loss_\numobs(\betastar)\right) + \left(\nabla
\Loss_\numobs(\betastar) - \nabla q_\lambda(\betahat)\right) + \lambda
\zhat = 0.
\end{equation*}
Defining $\Qhat \defn \int_0^1 \nabla^2 \Loss_n\left(\betastar +
t(\betahat - \betastar)\right) dt$, we then have
\begin{equation*}
\Qhat (\betahat - \betastar) + \left(\nabla
\Loss_\numobs(\betastar) - \nabla q_\lambda(\betahat)\right) + \lambda
\zhat = 0.
\end{equation*}
In block form, this means
\begin{equation}
\label{EqnZeroSubBlock}
\begin{bmatrix}
\Qhat_{SS} & \Qhat_{SS^c} \\
\Qhat_{S^cS} & \Qhat_{S^cS^c}
\end{bmatrix}  
\begin{bmatrix}		\betahat_S - \betastar_S \\ 0
\end{bmatrix}
+
\begin{bmatrix} \nabla \Loss_\numobs(\betastar)_S - \nabla
                q_\lambda(\betahat_S) \\ \nabla
                \Loss_\numobs(\betastar)_{S^c} - \nabla
                q_\lambda(\betahat_{S^c})
\end{bmatrix}
+ \lambda \begin{bmatrix} \zhat_S \\ \zhat_{S^c}
        \end{bmatrix}
= 0.
\end{equation}
By simple algebraic manipulations, we then have
\begin{multline}
\label{EqnRabbit0}
\zhat_{S^c} = \frac{1}{\lambda} \Big\{\left(\nabla
q_\lambda(\betahat_{S^c}) - \nabla \Loss_\numobs(\betastar)_{S^c}\right) +\Qhat_{S^cS} \left(\Qhat_{SS}\right)^{-1} \left(\nabla \Loss_\numobs(\betastar)_S -
\nabla q_\lambda(\betahat_S) + \lambda \zhat_S\right)\Big\}.
\end{multline}
Note that by the selection property (vi), we have $\nabla
q_\lambda(\betahat_{S^c}) = \nabla q_\lambda(0_{S^c}) = 0_{S^c}$, so
the first term in equation~\eqref{EqnRabbit0} vanishes, giving
\begin{equation}
\label{EqnRabbit}
\zhat_{S^c} = \frac{1}{\lambda} \Big\{- \nabla
\Loss_\numobs(\betastar)_{S^c} + \Qhat_{S^cS}
\left(\Qhat_{SS}\right)^{-1} \left(\nabla
\Loss_\numobs(\betastar)_S - \nabla q_\lambda(\betahat_S) + \lambda
\zhat_S\right)\Big\}.
\end{equation}
If the unbiasedness property (vii) also holds, we have the following result:
\begin{prop*}
\label{PropUnbiased}
Under the conditions of Theorem~\ref{ThmSuppRecovery}, suppose
$\rho_\lambda$ is $(\mu, \gamma)$-amenable. Also suppose
inequality~\eqref{EqnLavender} holds. Then strict dual feasibility
holds provided 
\begin{subequations}
\begin{equation}
\label{EqnPlum1}
\|\nabla \Loss_\numobs(\betastar)\|_\infty \le \frac{1-\delta}{2}
\cdot \lambda, \qquad \text{and}
\end{equation}
\begin{equation}
\label{EqnGrape1}
\left\|\Qhat_{S^cS} \left(\Qhat_{SS}\right)^{-1} \nabla
\Loss_\numobs(\betastar)_S\right\|_\infty \le \frac{1-\delta}{2} \cdot
\lambda.
\end{equation}
\end{subequations}
\end{prop*}

\begin{proof}
	
Combining the result of Lemma~\ref{LemSInfty} in
Appendix~\ref{SecThmEllInf} with equation~\eqref{EqnRabbit}, we see
that
\begin{align*}
\zhat_{S^c} & = \frac{1}{\lambda} \left\{\nabla
\Loss_\numobs(\betastar)_{S^c} + \Qhat_{S^cS}
\left(\Qhat_{SS}\right)^{-1} \nabla \Loss_\numobs(\betastar)_S\right\}
\\
& \le \frac{1}{\lambda} \|\nabla
\Loss_\numobs(\betastar)_{S^c}\|_\infty + \frac{1}{\lambda}
\left\|\Qhat_{S^cS} \left(\Qhat_{SS}\right)^{-1} \nabla
\Loss_\numobs(\betastar)_S\right\|_\infty \\
& \le \; \frac{1 - \delta}{2} + \frac{1 - \delta}{2} \le 1 - \delta,
\end{align*}
as claimed.
\end{proof}


However, for regularizers that do \emph{not} satisfy the unbiasedness
property (vii), such as the LSP and the usual $\ell_1$-norm, we impose
slightly stronger conditions to ensure strict dual feasibility. The
notion of an \emph{incoherence condition} is taken from the Lasso
literature~\cite{Zhao06, Wai09, MeiYu09}:

\begin{assumption}
\label{AsIncoherence}
There exists $\eta \in (0,1)$ such that $\opnorm{\Qhat_{S^cS} \left(\Qhat_{SS}\right)^{-1}}_\infty \le \eta$, where
$\opnorm{\cdot}_\infty$ denotes the $\ell_\infty$-operator norm.
\end{assumption}

\noindent Based on this condition, we have the following result:
\begin{prop*}
\label{PropJerk}
Under the conditions of Theorem~\ref{ThmSuppRecovery}, suppose in
addition that Assumption~\ref{AsIncoherence} holds. Then strict dual
feasibility holds provided
\begin{subequations}
\begin{equation}
\label{EqnPlum2}
\|\nabla \Loss_\numobs(\betastar)\|_\infty \le \frac{1 - \delta -
  \eta}{2} \, \lambda, \qquad \text{and}
\end{equation}
\begin{equation}
\label{EqnGrape2}
\left\|\Qhat_{S^cS} \left(\Qhat_{SS}\right)^{-1} \nabla
\Loss_\numobs(\betastar)_S\right\|_\infty \le \frac{1 - \delta - \eta}{2}
\, \lambda.
\end{equation}
\end{subequations}
\end{prop*}

\begin{proof}
Lemma~\ref{LemSmiley}(a) implies that $\|\lambda \zhat_S - \nabla
q_\lambda(\betahat_S)\|_\infty = \|\nabla
\rho_\lambda(\betahat_S)\|_\infty \le \lambda$.  Hence,
equation~\eqref{EqnRabbit}, together with the triangle inequality,
implies
\begin{align*}
\|\zhat_{S^c}\|_\infty & \le \frac{1}{\lambda} \|\nabla
\Loss_\numobs(\betastar)\|_\infty +
\frac{1}{\lambda}\left\|\Qhat_{S^cS} \left(\Qhat_{SS}\right)^{-1}
\nabla \Loss_\numobs(\betastar)_S\right\|_\infty \\
& \qquad + \frac{1}{\lambda}\left\|\Qhat_{S^cS}
\left(\Qhat_{SS}\right)^{-1} \left(\lambda
\zhat_S - \nabla q_\lambda(\betahat_S)\right)\right\|_\infty \\
& \le \frac{1 - \delta - \eta}{2} + \frac{1 - \delta - \eta}{2} + \eta
\; = 1 - \delta,
\end{align*}
where we have used inequalities~\eqref{EqnPlum2} and~\eqref{EqnGrape2}
and the incoherence condition (Assumption~\ref{AsIncoherence}).
\end{proof}

Comparing Propositions~\ref{PropUnbiased} and~\ref{PropJerk} reveals a
theoretical benefit of using unbiased penalties such as SCAD or MCP
over LSP or the $\ell_1$-norm: Whereas the bound~\eqref{EqnGrape1} is
essentially the same as the bound~\eqref{EqnGrape2},
Proposition~\ref{PropJerk} imposes an additional incoherence condition
on the covariates.


\section{Proofs of corollaries in Section~\ref{SecMain}}
\label{AppMain}

In this section, we provide detailed proofs of the corollaries to
Theorems~\ref{ThmSuppRecovery} and~\ref{ThmEllInf} appearing in
Section~\ref{SecMain}.


\subsection{Proof of Corollary~\ref{CorOLSLoss}}
\label{AppCorOLSLoss}

As established in Corollary 1 of our previous work~\cite{LohWai13},
the RSC condition~\eqref{EqnRSC} holds w.h.p.\ with $\alpha_1 =
\alpha_2 = \frac{1}{2} \lambda_{\min}(\Sigma_x)$ and $\tau_1, \tau_2
\asymp 1$, under sub-Gaussian assumptions on the variables.

\subsubsection{Proof of part (a)}

We now use the machinery developed in Appendix~\ref{AppDual}. We use a
slightly more direct analysis than the result stated in
Proposition~\ref{PropJerk}, which allows us to produce tighter
bounds. Denoting $(\GamHat, \gamhat) = \left(\frac{X^TX}{n},
\frac{X^Ty}{n}\right)$, we have
\begin{equation}
\label{EqnStar}
\nabla \Loss_\numobs(\beta) = \GamHat \beta - \gamhat, \qquad
\text{and} \qquad \nabla^2 \Loss_\numobs(\beta) = \GamHat,
\end{equation}
so equation~\eqref{EqnRabbit} takes the form
\begin{align}
\label{EqnKangaroo}
\zhat_{S^c} & = \frac{1}{\lambda} \left\{-\GamHat_{S^cS} \betastar_S +
\gamhat_{S^c} + \GamHat_{S^cS} \GamHat_{SS}^{-1} \left(\GamHat_{SS}
\betastar_S - \gamhat_S\right)\right\} + \frac{1}{\lambda}
\left\{\GamHat_{S^cS} \left(\GamHat_{SS}\right)^{-1} \left(\lambda
\zhat_S - \nabla q_\lambda(\betahat_S)\right)\right\} \notag \\
& \le \frac{1}{\lambda} \left\{-\GamHat_{S^cS} \betastar_S +
\gamhat_{S^c} + \GamHat_{S^cS} \betastar_S - \GamHat_{S^cS}
\GamHat_{SS}^{-1} \gamhat_S\right\} + \opnorm{\GamHat_{S^cS}
  \left(\GamHat_{SS}\right)^{-1}}_\infty \notag \\
& \le \frac{1}{\lambda} \left\{\gamhat_{S^c} - \GamHat_{S^cS}
\GamHat_{SS}^{-1} \gamhat_S\right\} + \eta,
\end{align}
using the fact that $\|\lambda \zhat_s - \nabla
q_\lambda(\betahat_S)\|_\infty = \|\nabla
\rho_\lambda(\betahat_S)\|_\infty \le \lambda$, by
Lemma~\ref{LemSmiley}. We now write
\begin{align}
\label{EqnCyan}
\left\|\gamhat_{S^c} - \GamHat_{S^cS} \GamHat_{SS}^{-1}
\gamhat_S\right\|_\infty & = \left\|\frac{X_{S^c}^T (X_S\betastar_S +
  \epsilon)}{n} - \left(\frac{X_{S^c}^T X_S}{n}\right)
\left(\frac{X_S^T X_S}{n}\right)^{-1} \left(\frac{X_S^T(X_S\betastar_S
  + \epsilon)}{n}\right)\right\|_\infty \notag \\
& = \left\|\frac{X_{S^c}^T \epsilon}{n} - \left(\frac{X_{S^c}^T
  X_S}{n}\right) \left(\frac{X_S^T X_S}{n}\right)^{-1}
\left(\frac{X_S^T \epsilon}{n}\right)\right\|_\infty \notag \\
	& = \left\|\frac{X_{S^c}^T \Pi \epsilon}{n}\right\|_\infty,
\end{align}
where $\Pi \defn I - X_S(X_S^T X_S)^{-1} X_S^T$ is an orthogonal
projection matrix. Note that for $t > 0$, we have
\begin{equation*}
\E\left[\exp\left(t \cdot e_j^T \frac{X_{S^c}^T \Pi
    \epsilon}{n}\right) \Big | X\right] \le c t^2 \sigma_{\epsilon}^2
\left\|\frac{\Pi X_{S^c} e_j}{n}\right\|_2^2 \le c t^2
\sigma_\epsilon^2 \left\|\frac{X_{S^c} e_j}{n}\right\|_2^2,
\end{equation*}
since $\epsilon$ is a sub-Gaussian vector with parameter
$\sigma_\epsilon^2$, and $\Pi$ is a projection matrix. Taking another
expectation, we have
\begin{equation*}
\E\left[\exp\left(t \cdot e_j^T \frac{X_{S^c}^T \Pi
    \epsilon}{n}\right) \right] \le \frac{c t^2 \sigma_\epsilon^2}{n}
\cdot e_j^T \Gamma_{S^c S^c} e_j \le c' \frac{t^2 \sigma_\epsilon^2
  \sigma_x^2}{n},
\end{equation*}
so $e_j^T \frac{X_{S^c}^T \Pi \epsilon}{n}$ is sub-Gaussian with
parameter $\frac{c\sigma_\epsilon^2 \sigma_x^2}{n}$. Using
sub-Gaussian tail bounds and a union bound, we conclude from
equation~\eqref{EqnCyan} that
\begin{equation*}
\left\|\gamhat_{S^c} - \GamHat_{S^cS}\GamHat_{SS}^{-1}
\gamhat_S\right\|_\infty \le C\sqrt{\frac{\log p}{n}},
\end{equation*}
with probability at least $1 - c\exp(-c' \log p)$, provided $n
\succsim k \log p$. Hence, strict dual feasibility holds, w.h.p.,
provided $\lambda > \frac{C}{1-\eta}\sqrt{\frac{\log p}{n}}$, and the
variable selection consistency property follows by
Theorem~\ref{ThmSuppRecovery}.

Turning to $\ell_\infty$-bounds, we have
\begin{align*}
\left\|\left(\Qhat_{SS}\right)^{-1} \nabla
\Loss_n(\betastar)\right\|_\infty & = \left\|\GamHat_{SS}^{-1}
\left(\GamHat_{SS} \betastar_S - \gamhat_S\right)\right\|_\infty \\
& = \left\|\left(\frac{X_S^T X_S}{n}\right)^{-1} \left(\frac{X_S^T
  X_S}{n} \betastar_S - \frac{X_S^T y}{n}\right)\right\|_\infty \\
& = \left\|\left(\frac{X_S^T X_S}{n}\right)^{-1} \left(\frac{X_S^T
  \epsilon}{n}\right)\right\|_\infty.
\end{align*}
For $j \in S$, we may write $e_j^T \left(\frac{X_S^T
  X_S}{n}\right)^{-1} \left(\frac{X_S^T \epsilon}{n}\right) \defn
v_j^T \epsilon$. Furthermore,
\begin{equation*}
\max_{j \in S} \|v_j\|_2^2 = \frac{1}{n^2} \cdot \max_{j \in S}
\left\|X_S \left(\frac{X_S^T X_S}{n}\right)^{-1} e_j\right\|_2^2 =
\frac{1}{n} \cdot \max_{j \in S} \left|e_j^T \left(\frac{X_S^T
  X_S}{n}\right)^{-1} e_j\right| \le \frac{2}{n} \cdot
\lambda_{\max}(\Sigma_x),
\end{equation*}
with probability at least $1 - c_1 \exp(-c_2 k)$. Conditioned on $X$,
the variables $v_j^T \epsilon$ are sub-Gaussian with parameter
$\sigma_\epsilon^2 \|v_j\|_2^2$. Denoting
\begin{equation*}
\mathcal{A} \defn \left\{ |v_j^T \epsilon| \le \sigma_\epsilon
\|v_j\|_2 \sqrt{\log p}, \quad \forall j \in S\right\}
\end{equation*}
and applying a union bound, we then have
\begin{equation*}
\mprob\left(\mathcal{A}^c \; \big | \; X\right) \le c_1' \exp(-c_2'
\log p),
\end{equation*}
so
\begin{equation*}
\mprob(\mathcal{A}^c) = \int \mprob(\mathcal{A}^c \mid X) \mprob(X) \;
d\mprob(X) \le c_1 \exp(-c_2' \log p),
\end{equation*}
as well. We conclude that
\begin{equation*}
\mprob\left(\max_{j \in S} |v_j^T \epsilon| \ge \lambda_{\max}^{1/2}
(\Sigma_x) \cdot \sigma_\epsilon \sqrt{\frac{2 \log p}{n}}\right) \le
\mprob(\mathcal{A}^c) + \mprob\left(\max_{j \in S} \|v_j\|_2^2 \ge
\frac{2}{n} \cdot \lambda_{\max} (\Sigma_x)\right),
\end{equation*}
and
\begin{equation*}
\left\|\GamHat_{SS}^{-1} \left(\GamHat_{SS} \betastar_S -
\gamhat_S\right)\right\|_\infty \le \lambda_{\max}^{1/2}(\Sigma_x)
\cdot \sigma_\epsilon \sqrt{\frac{2 \log p}{n}},
\end{equation*}
with probability at least $1 - c''_1 \exp(-c''_2 \min\{k, \log
p\})$. Combining this with the assumption~\eqref{EqnCInf}, we conclude
by part (a) of Theorem~\ref{ThmEllInf} that the desired
$\ell_\infty$-bound holds.


\subsubsection{Proof of part (b)}

The proof of this corollary is nearly identical to the proof of part
(a), except that equation~\eqref{EqnKangaroo} does not have the extra
term $\eta$; by the assumption of $(\mu, \gamma)$-amenability,
Lemma~\ref{LemSInfty} implies that $\lambda \zhat_S - \nabla
q_\lambda(\betahat_S) = 0$.  Hence, equation~\eqref{EqnRabbit} takes
the form

\begin{equation*}
\left\|\zhat_{S^c}\right\|_\infty \le \frac{1}{\lambda}
\left\|\gamhat_{S^c} - \GamHat_{S^cS} \GamHat_{SS}^{-1}
\gamhat_S\right\|_\infty.
\end{equation*}
Observe that the matrix concentration arguments from the proof of part
(a) also imply that \mbox{$\frac{1}{\lambda} \left\|\gamhat_{S^c} -
  \GamHat_{S^cS} \GamHat_{SS}^{-1} \gamhat_S\right\|_\infty < 1 -
  2\eta$}, w.h.p.  The remainder of the proof follows by part (b) of
Theorem~\ref{ThmEllInf}.


\subsection{Proof of Corollary~\ref{CorCorrLinearL1}}
\label{AppCorCorrLinearL1}

The expressions for $\nabla \Loss_\numobs$ and $\nabla^2
\Loss_\numobs$ are as in equation~\eqref{EqnStar}, with $(\GamHat,
\gamhat)$ defined in equation~\eqref{EqnGamCorrupted}. Furthermore, as
remarked in the proof of Corollary~\ref{CorOLSLoss}, Corollary 1 of
our previous work~\cite{LohWai13} implies that the RSC
condition~\eqref{EqnRSC} holds w.h.p., when $\alpha_1 = \alpha_2 =
\frac{1}{2} \lambda_{\min}(\Sigma_x)$ and $\tau_1, \tau_2 \asymp 1$.

We may verify using standard techniques that when the $x_i$'s and
$\epsilon_i$'s are sub-Gaussian and $n \succsim k \log p$, we have the
bound $\left\|\nabla \Loss_\numobs(\betastar)\right\|_\infty \precsim
\sqrt{\frac{\log p}{n}}$, with probability at least $1 - c_1 \exp(-c_2
\log p)$~\cite{LohWai11a}. Hence, if $\lambda \succsim
\sqrt{\frac{\log p}{n}}$, the lower bound in
inequality~\eqref{EqnLambdaCond} is satisfied w.h.p.\ We may then check
that for the choice of regularization parameters $\lambda \asymp
\sqrt{\frac{\log p}{n}}$ and $R \asymp \frac{1}{\lambda}$, and under
the scaling $n \succsim k \log p$, both bounds in
condition~\eqref{EqnLambdaCond} hold.

We now use Proposition~\ref{PropJerk} to verify condition (a) in
Theorem~\ref{ThmSuppRecovery}, establishing strict dual
feasibility~\eqref{EqnStrictDual}. From the bound~\eqref{EqnGrape2},
and with the scaling $\lambda \asymp \sqrt{\frac{\log p}{n}}$, it
suffices to show that
\begin{equation}
\label{EqnGamCond}
\left\|\GamHat_{S^cS} \GamHat_{SS}^{-1} \left(\GamHat_{SS} \betastar_S
- \gamhat_S\right)\right\|_\infty \precsim \sqrt{\frac{\log p}{n}},
\end{equation}
w.h.p. To this end, we have the following proposition, which we state
in some generality to indicate its applicability to other types of
corrupted linear models~\cite{LohWai11a}:
\begin{prop*}
\label{PropLinCase}
Suppose $\GamHat$ satisfies the deviation bound
\begin{equation}
\label{EqnDevMaster}
\mprob\left(\left|v^T \left(\GamHat - \Gamma\right) w\right| \ge t
\|v\|_2 \|w\|_2\right) \le c \exp(-c'nt^2), \qquad \forall v, w \in
\real^p.
\end{equation}
Also suppose $\gamhat$ satisfies the deviation bound
\begin{equation}
\label{EqnDevThree}
\mprob\left(\left\|A\left(\gamhat - \Gamma_{SS} \betastar\right)
\right\|_\infty \ge t \opnorm{A}_2 \right) \le c \exp(-c'nt^2 + c''
\log p), \qquad \forall A \in \real^{p \times p},
\end{equation}
and suppose $\opnorm{\Gamma}_2$ and $\opnorm{\Gamma^{-1}}_2$ are
uniformly bounded. Under the scaling $n \succsim k \cdot \max\{\log p,
k\}$ and $\lambda \asymp \sqrt{\frac{\log p}{n}}$,
inequality~\eqref{EqnGamCond} holds, with probability at least $1 - c
\exp(c' \min\{k, \log p\})$.
\end{prop*}

\begin{proof}
First note that inequality~\eqref{EqnDevMaster} implies the following
deviation bounds:
\begin{subequations}
\begin{align}
\label{EqnDevOne}
\mprob\left(\opnorm{\GamHat_{SS} - \Gamma_{SS}}_2 \ge t\right) & \le c
\exp(-c' nt^2 + c'' k), \\
\label{EqnDevOneHalf}
\mprob\left(\max_{j \in S^c} \left\|\left(\GamHat_{SS^c} -
\Gamma_{SS^c}\right)e_j\right\|_2 \ge t\right) & \le c \exp(-c' nt^2 +
c'' k + c''' \log p), \\
\label{EqnDevTwo}
\mprob\left(\left\|A \left(\GamHat - \Gamma\right)v\right\|_\infty \ge
t \opnorm{A}_2 \|v\|_2\right) & \le c \exp(-c' nt^2 + c'' \log p),
\qquad \forall v \in \real^p, \quad A \in \real^{p \times p}.
\end{align}
\end{subequations}
Inequality~\eqref{EqnDevOne} follows from from a discretization
argument over the $\ell_2$-ball in $\real^S$, and
inequality~\eqref{EqnDevOneHalf} is
similar. Inequality~\eqref{EqnDevTwo} follows by taking $w = Ae_j$ and
a union bound over $1 \le j \le p$. Furthermore, by
Lemma~\ref{LemSpecInverse}, inequality~\eqref{EqnDevOne} implies that
\begin{equation}
\label{EqnDevZero}
\mprob\left(\opnorm{\GamHat_{SS}^{-1} - \Gamma_{SS}^{-1}}_2 \ge
t\right) \le c \exp(-c' nt^2 + c'' k),
\end{equation}
as well.

Note that by inequalities~\eqref{EqnDevTwo} and~\eqref{EqnDevThree}, we have
\begin{equation}
\label{EqnBunny}
\left\|A\left(\GamHat_{SS} \betastar_S -
\gamhat_S\right)\right\|_\infty \le \left\|A\left(\GamHat_{SS} -
\Gamma_{SS}\right) \betastar_S\right\|_\infty +
\left\|A\left(\gamhat_S - \Gamma_{SS}
\betastar_S\right)\right\|_\infty \precsim \opnorm{A}_2 \cdot
\sqrt{\frac{\log p}{n}},
\end{equation}
with probability at least $1 - c \log(-c' \log p)$. Furthermore, the
triangle inequality gives
\begin{multline}
\label{EqnExpanded}
\left\|\GamHat_{S^cS} \GamHat_{SS}^{-1} \left(\GamHat_{SS} \betastar_S
- \gamhat_S\right)\right\|_\infty \le \left\|\Gamma_{S^cS}
\Gamma_{SS}^{-1} \left(\GamHat_{SS} \betastar_S -
\gamhat_S\right)\right\|_\infty \\
+ \left\|\left(\GamHat_{S^cS} \GamHat_{SS}^{-1} - \Gamma_{S^cS}
\Gamma_{SS}^{-1}\right) \left(\GamHat_{SS} \betastar_S -
\gamhat_S\right) \right\|_\infty,
\end{multline}
and the first term is bounded above by $C \sqrt{\frac{\log p}{n}}$, by
inequality~\eqref{EqnBunny}. For the second term, we write
\begin{align}
\label{EqnHyena}
\left\|\left(\GamHat_{S^cS} \GamHat_{SS}^{-1} - \Gamma_{S^cS}
\Gamma_{SS}^{-1}\right) \left(\GamHat_{SS} \betastar_S -
\gamhat_S\right)\right\|_\infty & \le \max_{j \in S^c}
\left\|e_j^T\left(\GamHat_{S^cS} \GamHat_{SS}^{-1} - \Gamma_{S^cS}
\Gamma_{SS}^{-1}\right)\right\|_2 \cdot \left\|\GamHat_{SS}
\betastar_S - \gamhat_S\right\|_2 \notag \\
& \le \max_{j \in S^c} \left\|e_j^T \left(\GamHat_{S^cS}
\GamHat_{SS}^{-1} - \Gamma_{S^cS} \Gamma_{SS}^{-1}\right)\right\|_2
\cdot C\sqrt{\frac{k \log p}{n}},
\end{align}
where we have again used inequality~\eqref{EqnBunny} in the second
inequality. Furthermore,
\begin{multline}
\label{EqnCroc}
\max_{j \in S^c} \left\|e_j^T \left(\GamHat_{S^cS} \GamHat_{SS}^{-1} -
\Gamma_{S^cS} \Gamma_{SS}^{-1}\right)\right\|_2 \le \max_{j \in S^c}
\left\|e_j^T\left(\GamHat_{S^cS} -
\Gamma_{S^cS}\right)\left(\GamHat_{SS}^{-1} -
\Gamma_{SS}^{-1}\right)\right\|_2 \\
+ \max_{j \in S^c} \left\|e_j^T \Gamma_{S^cS} \left(\GamHat_{SS}^{-1}
- \Gamma_{SS}^{-1}\right)\right\|_2 + \max_{j \in S^c} \left\|e_j^T
\left(\GamHat_{S^cS} - \Gamma_{S^cS}\right)\Gamma_{SS}^{-1}\right\|_2.
\end{multline}
The terms in inequality~\eqref{EqnCroc} are bounded as
\begin{align*}
\max_{j \in S^c} \left\|e_j^T\left(\GamHat_{S^cS} -
\Gamma_{S^cS}\right)\left(\GamHat_{SS}^{-1} -
\Gamma_{SS}^{-1}\right)\right\|_2 & \le \max_{j \in S^c} \left\|e_j^T
\left(\GamHat_{S^cS} - \Gamma_{S^cS}\right)\right\|_2 \cdot
\opnorm{\GamHat_{SS}^{-1} - \Gamma_{SS}^{-1}}_2, \\
\max_{j \in S^c} \left\|e_j^T \Gamma_{S^cS} \left(\GamHat_{SS}^{-1} -
\Gamma_{SS}^{-1}\right)\right\|_2 & \le \opnorm{\Gamma_{S^cS}} \cdot
\opnorm{\GamHat_{SS}^{-1} - \Gamma_{SS}^{-1}}_2, \\
\max_{j \in S^c} \left\|e_j^T \left(\GamHat_{S^cS} -
\Gamma_{S^cS}\right)\Gamma_{SS}^{-1}\right\|_2 & \le \max_{j \in S^c}
\left\|e_j^T \left(\GamHat_{S^cS} - \Gamma_{S^cS}\right)\right\|_2
\cdot \opnorm{\Gamma_{SS}^{-1}}_2.
\end{align*}
Combining the bounds with inequalities~\eqref{EqnDevOneHalf}
and~\eqref{EqnDevZero}, we conclude from inequality~\eqref{EqnCroc}
that
\begin{equation}
\label{EqnPiggy}
\max_{j \in S^c} \left\|e_j^T \left(\GamHat_{S^cS} \GamHat_{SS}^{-1} -
\Gamma_{S^cS} \Gamma_{SS}^{-1}\right)\right\|_2 \le
\max\left\{\sqrt{\frac{k}{n}}, \sqrt{\frac{\log p}{n}}\right\},
\end{equation}
with probability at least $1 - c \exp(-c' \min\{k, \log p\})$, under
the scaling $n \succsim \max\{k, \log p\}$. Plugging
inequality~\eqref{EqnPiggy} into inequality~\eqref{EqnHyena} and
combining with inequality~\eqref{EqnBunny}, we conclude that
\begin{equation*}
\left\|\GamHat_{S^cS} \GamHat_{SS}^{-1} - \left(\GamHat_{SS}
\betastar_S - \gamhat_S\right)\right\|_\infty \precsim
\sqrt{\frac{\log p}{n}} + \max\left\{\sqrt{\frac{k^2}{n}},
\sqrt{\frac{k \log p}{n}}\right\} \cdot \sqrt{\frac{\log p}{n}},
\end{equation*}
with probability at least $1 - c \exp(-c' \min\{k, \log p\})$. The
desired result then follows under the scaling $n \succsim k \cdot
\max\{k, \log p\}$.
\end{proof}

It is easy to check that the bounds~\eqref{EqnDevMaster}
and~\eqref{EqnDevThree} hold when $(\GamHat, \gamhat)$ are defined by
equation~\eqref{EqnGamCorrupted} and $X, W$, and $\epsilon$ are
sub-Gaussian. Hence, strict dual feasibility holds by
Propositions~\ref{PropJerk} and~\ref{PropLinCase}, and the PDW
technique succeeds.

Turning to $\ell_\infty$-bounds, note that
\begin{equation*}
\left\|\left(\Qhat_{SS}\right)^{-1} \nabla
\Loss_\numobs(\betastar)_S\right\|_\infty = \left\|\GamHat_{SS}^{-1}
\left(\GamHat_{SS} \betastar_S - \gamhat_S\right)\right\|_\infty.
\end{equation*}
We have the following result:
\begin{prop*}
\label{PropLinEllInf}
Under the same conditions as Proposition~\ref{PropLinCase}, we have
\begin{equation}
\label{EqnBlackberry}
\left\|\left(\Qhat_{SS}\right)^{-1} \nabla
\Loss_\numobs(\betastar)_S\right\|_\infty \precsim \sqrt{\frac{\log p}{n}},
\end{equation}
with probability at least $1 - c \exp(-c' \min\{k, \log p\})$.
\end{prop*}

\begin{proof}
	
We decompose
\begin{align*}
\left\|\GamHat_{SS}^{-1} \left(\GamHat_{SS} \betastar_S -
\gamhat_S\right)\right\|_\infty & \le \left\|\left(\GamHat_{SS}^{-1} -
\Gamma_{SS}^{-1}\right)\left(\GamHat_{SS} \betastar_S -
\gamhat_S\right)\right\|_\infty + \left\|\Gamma_{SS}^{-1}
\left(\GamHat_{SS} \betastar_S - \gamhat_S\right)\right\|_\infty \\
& \le \sqrt{k} \cdot \opnorm{\GamHat_{SS}^{-1} - \Gamma_{SS}^{-1}}_2
\cdot \left\|\GamHat_{SS} \betastar_S - \gamhat_S\right\|_\infty +
\left\|\Gamma_{SS}^{-1} \left(\GamHat_{SS} \betastar_S -
\gamhat_S\right)\right\|_\infty \\
& \precsim \sqrt{k} \cdot \sqrt{\frac{k}{n}} \cdot \sqrt{\frac{\log
    p}{n}} + \sqrt{\frac{\log p}{n}},
\end{align*}
where the last inequality holds with probability at least $1 - c
\exp(-c' \min{k, \log p})$, and we have used the
bounds~\eqref{EqnDevZero} and~\eqref{EqnBunny} to establish the last
inequality. Inequality~\eqref{EqnBlackberry} then follows under the
scaling $n \ge k^2$.
	
\end{proof}

The bound for $\|\betatil - \betastar\|_\infty$ follows by plugging
inequality~\eqref{EqnBlackberry} into inequality~\eqref{EqnHoneydew}
of Theorem~\ref{ThmEllInf}.


\subsection{Proof of Corollary~\ref{CorGLMLoss}}
\label{AppCorGLMLoss}

Note that Corollary 2 of our previous work~\cite{LohWai13} establishes
the RSC condition~\eqref{EqnRSC} when $\psi''$ is bounded and the
$x_i$'s are again sub-Gaussian. We again proceed via the framework and
terminology of Appendix~\ref{AppDual}, particularly
Proposition~\ref{PropUnbiased}.  Taking derivatives of the loss
function~\eqref{EqnLossGLM}, we have
\begin{equation*}
\nabla \Loss_\numobs(\beta) = \frac{1}{n} \sum_{i=1}^n \left(\psi'(x_i^T
\beta) x_i - y_i x_i\right), \qquad \text{and} \qquad \nabla^2
\Loss_\numobs(\beta) = \frac{1}{n} \sum_{i=1}^n \psi''(x_i^T \beta) x_i
x_i^T.
\end{equation*}
To verify inequality~\eqref{EqnPlum1}, we may use straightforward
sub-Gaussian concentration techniques (cf.\ the proof of Corollary 2
in Loh and Wainwright~\cite{LohWai13}). For
inequality~\eqref{EqnGrape1}, note that
\begin{align*}
\Qhat - \nabla^2 \Loss_\numobs(\betastar) & = \int_0^1 \left\{\nabla^2
\Loss_n\left(\betastar + t(\betahat - \betastar)\right) - \nabla^2
\Loss_n(\betastar)\right\} dt \\
& = 
\int_0^1 \left\{\frac{1}{n} \sum_{i=1}^n \left(\psi''(x_i^T \left(\betastar + t(\betahat - \betastar)\right) - \psi''(x_i^T
\betastar)\right) x_i x_i^T \right\} dt \\
& = \int_0^1 \left\{t \cdot \frac{1}{n} \sum_{i=1}^n \psi'''(x_i^T \betahat_t) x_i^T (\betahat -
\betastar) \cdot x_i x_i^T \right\} dt,
\end{align*}
by the mean value theorem, where $\betahat_t$ lies on the line segment
between $\betastar$ and $\betastar + t(\betahat - \betastar)$.

For each pair $v, w \in \real^p$, we write
\begin{align}
\label{EqnAntelope}
\left|v^T \left(\Qhat - \nabla^2 \Loss_\numobs(\betastar)\right)
w\right| & = \left|\int_0^1 \left\{t \cdot \frac{1}{n} \sum_{i=1}^n
\psi'''(x_i^T \betahat_t) x_i^T (\betahat - \betastar) \cdot x_i^T v
\cdot x_i^T w \right\} dt \right| \notag \\
& \le \int_0^1 t \cdot \left|\frac{1}{n} \sum_{i=1}^n
\psi'''(x_i^T \betahat_t) x_i^T (\betahat - \betastar) \cdot x_i^T v
\cdot x_i^T w \right | dt \notag \\
& \le \int_0^1 t \left\{\sqrt{\frac{1}{n} \sum_{i=1}^n
  \left|\psi'''(x_i^T \betahat_t) x_i^T (\betahat -
  \betastar)\right|^2 (x_i^T v)^2} \cdot \sqrt{\frac{1}{n}
  \sum_{i=1}^n (x_i^T w)^2} \right\} dt \notag \\
& \le \kappa_3 \sqrt{\frac{1}{n} \sum_{i=1}^n \|x_{i, S}\|_2^2
  \|\betahat - \betastar\|_2^2 (x_i^T v)^2} \cdot \sqrt{\frac{1}{n}
  \sum_{i=1}^n (x_i^T w)^2},
\end{align}
where $x_{i,S}$ denotes the vector $x_i$ restricted to
$S$. Inequality~\eqref{EqnAntelope} follows by two applications of the
Cauchy-Schwarz inequality and the fact that $\supp(\betatil) \subseteq
S$. Furthermore, by Lemma~\ref{LemEll2Bounds} in
Appendix~\ref{AppSupporting}, we have
\begin{equation*}
\|\betahat - \betastar\|_2 \precsim
\sqrt{\frac{k \log p}{n}},
\end{equation*}
with probability at least $1 - c \exp(-c' \log p)$. By
Assumption~\ref{AsGLM}(i), we also have $\|x_{i,S}\|_2^2 \le
Mk$. Hence, inequality~\eqref{EqnAntelope} becomes
\begin{equation*}
\left|v^T \left(\Qhat - \nabla^2 \Loss_\numobs(\betastar)\right)
w\right| \precsim \kappa_3 \sqrt{\frac{M k^2 \log p}{n}} \cdot
\sqrt{\frac{1}{n} \sum_{i=1}^n (x_i^T v)^2} \cdot \sqrt{\frac{1}{n}
  \sum_{i=1}^n (x_i^T w)^2}.
\end{equation*}
In particular, taking a supremum over unit vectors $v, w \in \real^S$, we have
\begin{align}
\label{EqnYin}
\opnorm{\Qhat_{SS} - \nabla^2 \Loss_\numobs(\betastar)_{SS}}_2 &
\precsim \kappa_3 \sqrt{\frac{Mk^2 \log p}{n}} \cdot
\opnorm{\left(\frac{1}{n} \sum_{i=1}^n x_i x_i^T\right)_{SS}}_2 \notag
\\
& \precsim \kappa_3 \sqrt{\frac{Mk^2 \log p}{n}} \cdot
\left(\lambda_{\max}(\Sigma_x) + C\sqrt{\frac{k}{n}}\right) \notag \\
& \precsim \sqrt{\frac{k^2 \log p}{n}},
\end{align}
with probability at least $1 - c \exp(-c' k)$, where the second
inequality holds by a standard spectral norm bound on the sample
covariance matrix.

Furthermore, for $v, w \in \real^p$, the expression $v^T \nabla^2
\Loss_\numobs(\betastar) w = \frac{1}{n} \sum_{i=1}^n \psi''(x_i^T
\betastar) v^T x_i \cdot w^T x_i$ is an i.i.d.\ average of a product
of sub-Gaussian random variables $\psi''(x_i^T \betastar) v^T x_i$ and
$w^T x_i$, since the $x_i$'s are sub-Gaussian and $\psi''$ is
uniformly bounded by assumption.  Defining \mbox{$\Loss(\betastar)
  \defn \E\left[\Loss_\numobs(\betastar)\right]$,} a standard
discretization argument of $k$-dimensional unit sphere yields
\begin{equation}
\label{EqnYang}
\opnorm{\nabla^2 \Loss_\numobs(\betastar)_{SS} - \nabla^2
  \Loss(\betastar)_{SS}}_2 \precsim \sqrt{\frac{k}{n}},
\end{equation}
with probability at least $1 - c \exp(-c' k)$. Hence, by
inequalities~\eqref{EqnYin} and~\eqref{EqnYang}, we have
$\opnorm{\Qhat_{SS} - \nabla^2 \Loss(\betastar)_{SS}}_2 \precsim
\sqrt{\frac{k^2 \log p}{n}}$.  Applying Lemma~\ref{LemSpecInverse} in
Appendix~\ref{AppSupporting} yields
\begin{equation}
\label{EqnSquash}
\opnorm{\left(\Qhat_{SS}\right)^{-1} -
  \left(\nabla^2 \Loss(\betastar)_{SS}\right)^{-1}}_2 \precsim
\sqrt{\frac{k^2 \log p}{n}},
\end{equation}
as well. A similar argument shows that
\begin{equation*}
\max_{j \in S^c} \left\|e_j^T \left(\Qhat_{S^cS} - \nabla^2
\Loss_\numobs(\betastar)_{S^cS}\right)\right\|_2 \precsim
\sqrt{\frac{k^2 \log p}{n}},
\end{equation*}
and
\begin{equation*}
\max_{j \in S^c} \left\|e_j^T \left(\nabla^2 \Loss_\numobs(\betastar)_{S^cS}
- \nabla^2 \Loss(\betastar)\right)_{S^cS}\right\|_2 \precsim
\max\left\{\sqrt{\frac{k}{n}}, \sqrt{\frac{\log p}{n}}\right\},
\end{equation*}
with probability at least $1 - c \exp(-c' \min\{k, \log p\})$.  Putting
together the pieces, we find that
\begin{equation}
\label{EqnZucchini}
\max_{j \in S^c} \left\|e_j^T \left(\Qhat_{S^cS} - \nabla^2
\Loss(\betastar)_{S^cS}\right)\right\|_2 \precsim \sqrt{\frac{k^2 \log
    p}{n}},
\end{equation}
with high probability. Returning to the expression~\eqref{EqnGrape1},
we have the bound
\begin{align*}
\left \|\Qhat_{S^cS} \left(\Qhat_{SS}\right)^{-1} \nabla
\Loss_\numobs(\betastar)_S\right\|_\infty \leq \Term_1 + \Term_2
\end{align*}
where $\Term_1 \defn \left\|\nabla^2 \Loss(\betastar)_{S^cS}
\left(\nabla^2 \Loss(\betastar)_{SS}\right)^{-1} \nabla
\Loss(\betastar)_S\right\|_\infty$, and
\begin{align}
\label{EqnApricot}
\Term_2 & \defn \left\|\left\{\Qhat_{S^cS}
\left(\Qhat_{SS}\right)^{-1} - \nabla \Loss(\betastar)_{S^cS}
\left(\nabla^2 \Loss(\betastar)_{SS}\right)^{-1}\right\} \nabla
\Loss_\numobs(\betastar)_S\right\|_\infty.
\end{align}
Since the eigenspectrum of $\nabla^2 \Loss(\betastar)$ is bounded by
assumption, standard techniques (cf.\ the proofs of Corollary 2 in Loh
and Wainwright~\cite{LohWai13} and Lemma 6 in Negahban et
al.~\cite{NegRavWaiYu12}) guarantee that $\Term_1 \leq c
\sqrt{\frac{\log \pdim}{\numobs}}$, with probability at least $1-c_1
\exp(-c_2 \log \pdim)$.  Turning to the second term, we have
\begin{equation}
\label{EqnOneStar}
\Term_2 \le \max_{j \in S^c} \left\|e_j^T \left\{\Qhat_{S^cS}
\left(\Qhat_{SS}\right)^{-1} - \nabla \Loss(\betastar)_{S^cS}
\left(\nabla^2 \Loss(\betastar)_{SS}\right)^{-1}\right\} \right\|_2
\cdot \left\|\nabla \Loss_\numobs(\betastar)_S\right\|_2.
\end{equation}
Now,
\begin{equation}
\label{EqnTwoStar}
\|\nabla \Loss_\numobs(\betastar)_S\|_2 \le \sqrt{k} \cdot \|\nabla
\Loss_\numobs(\betastar)\|_\infty \precsim \sqrt{k} \cdot
\sqrt{\frac{\log p}{n}},
\end{equation}
and by the triangle inequality,
\begin{multline}
\label{EqnThreeStar}
\left\|e_j^T \left\{\Qhat_{S^cS} \left(\Qhat_{SS}\right)^{-1} - \nabla
\Loss(\betastar)_{S^cS} \left(\nabla^2
\Loss(\betastar)_{SS}\right)^{-1}\right\} \right\|_2 \\
\le \left\|e_j^T \nabla^2 \Loss(\betastar)_{S^cS} \Delta_1\right\|_2 +
\left\|e_j^T \Delta_2 \left(\nabla^2
\Loss(\betastar)_{SS}\right)^{-1}\right\|_2 + \left\|e_j^T \Delta_2
\Delta_1\right\|_2,
\end{multline}
where
\begin{align*}
\Delta_1 \defn \left(\Qhat_{SS}\right)^{-1} -
\left(\nabla^2 \Loss(\betastar)_{SS}\right)^{-1}, \qquad \text{and}
\qquad \Delta_2 \defn \Qhat_{S^cS} - \nabla^2
\Loss(\betastar)_{S^cS}.
\end{align*}
By the bounds~\eqref{EqnSquash} and~\eqref{EqnZucchini}, we have
\begin{equation*}
\opnorm{\Delta_1}_2, \; \max_{j \in S^c} \left\|e_j^T
\Delta_2\right\|_2 \precsim \sqrt{\frac{k^2 \log p}{n}},
\end{equation*}
w.h.p., so combined with inequalities~\eqref{EqnOneStar},
\eqref{EqnTwoStar}, and~\eqref{EqnThreeStar}, we have
\begin{equation*}
R \precsim \sqrt{k} \cdot \sqrt{\frac{\log p}{n}} \cdot
\sqrt{\frac{k^2 \log p}{n}}.
\end{equation*}
Hence, by Proposition~\ref{PropUnbiased}, strict dual feasibility
holds under the scaling $n \succsim k^3 \log p$.

Turning to $\ell_\infty$-bounds, we have
\begin{align*}
\left\|\left(\Qhat_{SS}\right)^{-1} \nabla
\Loss_\numobs(\betastar)_S\right\|_\infty & =
\left\|\left\{\left(\Qhat_{SS}\right)^{-1} - \left(\nabla^2
\Loss(\betastar)_{SS}\right)^{-1}\right\} \nabla
\Loss_\numobs(\betastar)_S\right\|_\infty \\
& \qquad + \left\|\left(\nabla^2 \Loss(\betastar)_{SS}\right)^{-1}
\nabla \Loss_\numobs(\betastar)_S\right\|_\infty \\
& \le \sqrt{k} \cdot \opnorm{\left(\Qhat_{SS}\right)^{-1} - \left(\nabla^2
  \Loss(\betastar)_{SS}\right)^{-1}}_2 \cdot \left\|\nabla
\Loss_\numobs(\betastar)_S\right\|_\infty \\
& \qquad + \left\|\left(\nabla^2 \Loss(\betastar)_{SS}\right)^{-1}
\nabla \Loss_\numobs(\betastar)_S\right\|_\infty \\
& \precsim \sqrt{k} \cdot \sqrt{\frac{k^2 \log p}{n}} \cdot
\sqrt{\frac{\log p}{n}} + \sqrt{\frac{\log p}{n}},
\end{align*}
using inequality~\eqref{EqnSquash} and the same sub-Gaussian
concentration result used to bound the term $T_1$ above. Hence,
\begin{equation*}
\left\|\left(\Qhat_{SS}\right)^{-1} \nabla
\Loss_n(\betastar)_S\right\|_\infty \precsim \sqrt{\frac{\log p}{n}},
\end{equation*}
under the scaling $n \precsim k^3 \log p$. Finally, note that inequality~\eqref{EqnSquash} implies
\begin{equation*}
\opnorm{\left(\Qhat_{SS}\right)^{-1} -
  \left(Q^*_{SS}\right)^{-1}}_\infty \le \sqrt{k}
\opnorm{\left(\Qhat_{SS}\right)^{-1} - \left(\nabla^2
  \Loss(\betastar)_{SS}\right)^{-1}}_2 \precsim \sqrt{\frac{k^3 \log
    p}{n}} \le c_\infty,
\end{equation*}
so
\begin{equation*}
\opnorm{\left(\Qhat_{SS}\right)^{-1}}_\infty \le
\opnorm{\left(\Qhat_{SS}\right)^{-1} -
  \left(Q^*_{SS}\right)^{-1}}_\infty +
\opnorm{\left(Q^*_{SS}\right)^{-1}}_\infty \le 2c_\infty.
\end{equation*}
Theorem~\ref{ThmEllInf} then implies the desired result.


\subsection{Proof of Corollary~\ref{CorGlassoLoss}}
\label{AppCorGlassoLoss}

In this section, we provide the proof of
Corollary~\ref{CorGlassoLoss}. Corollary 3 of our previous
work~\cite{LohWai13} establishes the RSC condition~\eqref{EqnRSC}, with
$\alpha_1$ and $\alpha_2$ scaling as $\opnorm{\Thetastar}_2$ and
$\tau_1 = 0$. Our proof strategy deviates mildly from the framework
described in Section~\ref{SecPDW}, in that we provide a slightly
different way of constructing a matrix $\Thetahat_S$ such that
$\supp(\Thetahat_S) \subseteq S$ and $\Thetahat_S$ is a
zero-subgradient point of the restricted
program~\eqref{EqnRestrictedM}. However, subject to this minor
adjustment, the remainder of the primal-dual witness technique
proceeds as before.

We begin with a simple lemma establishing the convexity of the
program~\eqref{EqnGlassoLoss}:
\begin{lem*}
\label{LemGlassoConvex}
Suppose $\rho_\lambda$ is $\mu$-amenable. Then for the choice of
parameter $\kappa = \sqrt{\frac{2}{\mu}}$, the objective function in
the program~\eqref{EqnGlassoLoss} is strictly convex over the
constraint set.
\end{lem*}

\begin{proof}
A simple calculation shows that for the graphical Lasso loss
\begin{equation*}
\Loss_\numobs(\Theta) = \tr(\Sigmahat \Theta) - \log \det(\Theta),
\end{equation*}
we have $\nabla^2 \Loss_\numobs(\Theta) = (\Theta \otimes
\Theta)^{-1}$. Note that $\nabla^2 \Loss_\numobs(\Theta)$ is a deterministic
quantity that does not depend on $\{x_i\}_{i=1}^n$. Hence,
\begin{equation*}
\lambda_{\min}\left(\nabla^2 \Loss_\numobs(\Theta)\right) =
\lambda_{\max}^{-1}\left(\Theta \otimes \Theta\right) =
\lambda_{\max}^{-2} (\Theta).
\end{equation*}
We see that for $\opnorm{\Theta}_2 \le \sqrt{\frac{1}{\mu}}$, we have
$\lambda_{\min}\left(\nabla^2 \Loss_\numobs(\Theta)\right) \ge
\frac{\mu}{2}$, so $\Loss_\numobs(\Theta) - \frac{\mu}{2}
\opnorm{\Theta}_F^2$ is convex. Further note that by assumption, the
quantity $\sum_{j \neq k} \rho_\lambda(\Theta_{jk}) + \frac{\mu}{2}
\opnorm{\Theta}_F^2$ is also convex. Hence, the overall objective
function is strictly convex over the feasible set, as claimed.
\end{proof}

In particular, if there exists a zero-subgradient point of the
composite objective function $\Loss_\numobs(\Theta) +
\rho_\lambda(\Theta)$ within the feasible set, it must be the unique
global minimum.  Our strategy is to construct such a zero-subgradient
point.  Let $S \defn \{(j,k) \, \mid \, j \le k, \Thetastar_{jk} \neq
0\}$ denote the support of $\Thetastar$, where we remove redundant
elements. We define the map $F: \real^{|S|} \rightarrow \real^{|S|}$
according to
\begin{equation*}
F(\VEC(\Delta_S)) \defn -\left(\Gamstar_{SS}\right)^{-1}
\left(\VEC\left(\Sigmahat_S - \left(\left(\Thetastar +
\Delta\right)^{-1}\right)_S\right)\right) + \VEC\left(\Delta_S\right),
\end{equation*}
where $\Delta \in \real^{p \times p}$ is the symmetric matrix agreeing
with $\Delta_S$ on $S$ and having 0's elsewhere, and $\Gamstar \defn
\Theta^{*-1} \otimes \Theta^{*-1} = \Sigmastar \otimes \Sigmastar$. We
analyze the behavior of $F$ over the $\ell_\infty$-ball
$\ball_\infty(r)$ of radius $r$ to be specified later. In particular,
note that for $\Delta_S \in \ball_\infty(r)$, we have
\begin{equation}
\label{EqnDeltaBd}
\opnorm{\Delta}_2 \le \opnorm{\Delta}_\infty \le dr,
\end{equation}
since $\Delta$ has at most $d$ nonzero entries per row. Hence, when
$dr < \lambda_{\min}(\Thetastar)$, we are guaranteed that the matrix
$(\Thetastar + \Delta)$ is invertible, making $F$ a continuous map. We
show that $F(\ball_\infty(r)) \subseteq \ball_\infty(r)$, so by
Brouwer's fixed point theorem~\cite{OrtRhe00}, the function must have
a fixed point, which we denote by $\Deltastar_S \in \ball_\infty(r)$.
Defining the constants $\kappa_\Gamma \defn
\opnorm{\left(\Gamstar_{SS}\right)^{-1}}_\infty$ and $\kappa_\Sigma
\defn \opnorm{\Sigmastar}_\infty$, this insight is summarized in the
following lemma:
\begin{lem*}
\label{LemBrouwer}
Let $r \defn 2c_0 \kappa_\Gamma \sqrt{\frac{\log p}{n}}$, where $c_0$
is a constant depending only on the sub-Gaussian parameter of the
$x_i$'s. Suppose
\begin{equation}
\label{EqnTwoSmiley}
d r \le \min\left\{\frac{1}{2} \lambda_{\min}(\Thetastar), \;
\frac{1}{2 \kappa_\Sigma}, \; \frac{1}{4 \kappa_\Gamma
  \kappa_\Sigma^3}\right\},
\end{equation}
and suppose the sample size satisfies $n \succsim \kappa_\Gamma^2 \log
p$. Then with probability at least $1 - c \exp(-c' \log p)$, there
exists $\Thetahat \in \real^{p \times p}$ such that
\begin{equation}
\label{EqnDeltaBd2}
\|\Thetahat - \Thetastar\|_{\max} \le r, \qquad \opnorm{\Thetahat -
  \Thetastar}_2 \le dr, \quad \mbox{and} \quad \Sigmahat_S -
\left(\Thetahat^{-1}\right)_S = 0.
\end{equation}
Furthermore, if $\rho_\lambda$ is $(\mu, \gamma)$-amenable and
$\min_{(j,k) \in S} |\Thetastar_{jk}| \ge \gamma\lambda + r$, we have
$\rho'_\lambda(\Thetahat_{jk}) = 0$, for all $(j,k) \in S$.
\end{lem*}

\begin{proof}
We first establish that $F(\ball_\infty(r)) \subseteq
\ball_\infty(r)$. Consider $\Delta_S \in \ball_\infty(r)$. We have
\begin{equation*}
F(\VEC(\Delta_S)) = - \left(\Gamstar_{SS}\right)^{-1}
\left\{\VEC\left(\Sigmahat_S - \Sigmastar_S\right) +
\VEC\left(\left(\Sigmastar - (\Thetastar +
\Delta)^{-1}\right)_S\right) - \Gamstar_{SS} \VEC(\Delta_S)\right\},
\end{equation*}
implying that
\begin{align}
\label{EqnSpruce}
\|F(\VEC(\Delta_S))\|_\infty & \le \kappa_\Gamma
\left\|\VEC\left(\Sigmahat_S - \Sigmastar_S\right) +
\VEC\left(\left(\Sigmastar - (\Thetastar +
\Delta)^{-1}\right)_S\right) - \Gamstar_{SS}
\VEC(\Delta_S)\right\|_\infty \notag \\
& \le \kappa_\Gamma \left\|\VEC\left(\Sigmahat_S -
\Sigmastar_S\right)\right\|_\infty + \kappa_\Gamma
\left\|\VEC\left(\left(\Sigmastar - (\Thetastar +
\Delta)^{-1}\right)_S\right) - \Gamstar_{SS}
\VEC(\Delta_S)\right\|_\infty.
\end{align}
For the first term, we have
\begin{equation}
\label{EqnVine}
\kappa_\Gamma \left\|\VEC\left(\Sigmahat_S -
\Sigmastar_S\right)\right\|_\infty \le c_0 \kappa_\Gamma
\sqrt{\frac{\log p}{n}} \defn \frac{r}{2},
\end{equation}
w.h.p., by the sub-Gaussian assumption on the $x_i$'s. Furthermore, by
the matrix expansion
\begin{equation}
\label{EqnMatInvExpand}
(A + \Delta)^{-1} - A^{-1} = \sum_{\ell=1}^\infty \left(-A^{-1}
\Delta\right)^\ell A^{-1},
\end{equation}
we have
\begin{align*}
\VEC\left(\left(\Sigmastar - (\Thetastar +
\Delta)^{-1}\right)_S\right) - \Gamstar_{SS} \VEC(\Delta_S) & =
\VEC\left(\left(\Sigmastar - (\Thetastar + \Delta)^{-1}\right)_S -
\left(\Sigmastar \Delta \Sigmastar\right)_S\right) \\
& = - \VEC\left(\left(\sum_{\ell=2}^\infty \left(- \Sigmastar
\Delta\right)^{\ell} \Sigmastar\right)_S\right).
\end{align*}
By the triangle inequality, we then have
\begin{equation}
\label{EqnCedar}
\left\|\VEC\left(\left(\Sigmastar - (\Thetastar +
\Delta)^{-1}\right)_S\right) - \Gamstar_{SS}
\VEC(\Delta_S)\right\|_\infty \le \max_{(j,k) \in S}
\sum_{\ell=2}^\infty \left|e_j^T (\Sigmastar \Delta)^{\ell} \Sigmastar
e_k\right|.
\end{equation}
By H\"{o}lder's inequality, we have
\begin{align*}
\left|e_j^T (\Sigmastar \Delta)^{\ell} \Sigmastar e_k\right| \; \le \;
\left\|e_j^T (\Sigmastar \Delta)^{\ell-1} \Sigmastar\right\|_1 \cdot
\left\|\Delta \Sigmastar e_k\right\|_\infty \; & \le
\opnorm{\Sigmastar (\Delta \Sigmastar)^{\ell-1}}_1 \cdot
\|\Delta\|_{\max} \cdot \|\Sigmastar e_k\|_1 \\
& \le \opnorm{\Sigmastar}_1^\ell \opnorm{\Delta}_1^{\ell-1} \cdot
\|\Delta\|_{\max} \cdot \opnorm{\Sigmastar}_1 \\
& = \opnorm{\Sigmastar}_\infty^{\ell+1}
\opnorm{\Delta}_\infty^{\ell-1} \cdot \|\Delta\|_{\max}.
\end{align*}
Using inequality~\eqref{EqnDeltaBd} and plugging back into
inequality~\eqref{EqnCedar}, we then have
\begin{align}
\label{EqnRose}
\left\|\VEC\left(\left(\Sigmastar - (\Thetastar +
\Delta)^{-1}\right)_S\right) - \Gamstar_{SS}
\VEC(\Delta_S)\right\|_\infty \le \sum_{\ell=2}^\infty
\kappa_{\Sigma}^{\ell+1} d^{\ell-1} r^\ell = \frac{\kappa_\Sigma^3
  dr^2}{1-\kappa_\Sigma dr} \; \le 2 \kappa_\Sigma^3 dr^2,
\end{align}
where the last inequality follows from our
assumption~\eqref{EqnTwoSmiley}. Combining
inequalities~\eqref{EqnSpruce}, \eqref{EqnVine}, and~\eqref{EqnRose},
and using the assumption that $2 \kappa_\Gamma \kappa_\Sigma^3 dr^2
\le \frac{r}{2}$, we find that $\|F(\VEC(\Delta_S))\|_\infty \le r$,
as desired. Applying Brouwer's fixed point theorem then yields the
required fixed point $\Deltastar_S$. Defining $\Thetahat \defn
\Thetastar + \Deltastar$, the third equality in
line~\eqref{EqnDeltaBd2} follows.  The operator norm bound in
inequality~\eqref{EqnDeltaBd2} follows from the same argument as in
inequality~\eqref{EqnDeltaBd}. Finally, by the triangle inequality and
the assumption that $\rho_\lambda$ is $(\mu, \gamma)$-amenable, we
have $\rho'_\lambda(\Thetahat_{jk}) = \rho'_\lambda(\Thetastar_{jk}) =
0$.
\end{proof}

Equipped with Lemma~\ref{LemBrouwer}, we now prove that $\Thetahat$ is
the (unique) global optimum of the program~\eqref{EqnGlassoLoss}. Note
that it suffices to show that $\Thetahat$ is a zero-subgradient point
of the objective function in the program~\eqref{EqnGlassoLoss}; since
the objective is strictly convex by Lemma~\ref{LemGlassoConvex},
$\Thetahat$ must be the global minimum. The same manipulations used in
Appendix~\ref{AppDual} and Proposition~\ref{PropUnbiased} reveal that
$\Thetahat$ is a zero-subgradient point satisfying dual feasibility,
if inequalities~\eqref{EqnPlum1} and~\eqref{EqnGrape1} hold. In fact,
we may take $\delta = 0$, since the convexity of the problem
immediately implies uniqueness of a stationary point if it exists. We
treat $\nabla^2 \Loss_\numobs(\Theta)$ as a $p^2 \times p^2$ matrix
operating on the $p$-dimensional vector $\VEC(\Theta)$. For the
graphical Lasso, we have $\nabla \Loss_\numobs(\Thetastar) = \Sigmahat
- \Sigmastar$, so we need to show that $\left\|\Sigmahat -
\Sigmastar\right\|_{\max} \le \frac{\lambda}{2}$.  It is
straightforward to see that by the sub-Gaussianity of the $x_i$'s,
setting $\lambda = c \sqrt{\frac{\log p}{n}}$ for a suitably large
constant is sufficient to satisfy inequality~\eqref{EqnPlum1}.

Turning to inequality~\eqref{EqnGrape1}, note that
\begin{align*}
& \opnorm{\Qhat - \nabla^2 \Loss_n(\Thetastar)}_\infty \\
& \qquad \qquad = \opnorm{\int_0^1 \left\{\left(\Thetastar +
    t(\Thetahat - \Thetastar)\right)^{-1} \otimes \left(\Thetastar +
    t(\Thetahat - \Thetastar)\right)^{-1} - \Theta^{*-1} \otimes
    \Theta^{*-1}\right\} dt}_\infty \\
& \qquad \qquad \le \int_0^1 \opnorm{\left(\Thetastar + t(\Thetahat -
    \Thetastar)\right)^{-1} \otimes \left(\Thetastar + t(\Thetahat -
    \Thetastar)\right)^{-1} - \Theta^{*-1} \otimes
    \Theta^{*-1}}_\infty dt.
\end{align*}
Furthermore, for $t \in [0,1]$, we have
\begin{equation*}
\opnorm{\left(\Thetastar + t(\Thetahat - \Thetastar)\right) -
  \Thetastar}_\infty = t \cdot \opnorm{\Thetahat - \Thetastar}_\infty
\le d \left\|\Thetahat - \Thetastar\right\|_{\max} \precsim d
\sqrt{\frac{\log p}{n}},
\end{equation*}
using Lemma~\ref{LemBrouwer}. Lemma~\ref{LemSpecInverse} in
Appendix~\ref{AppSupporting} then implies that
\begin{equation*}
\opnorm{\left(\Thetastar + t(\Thetahat - \Thetastar)\right)^{-1} -
  \Theta^{*-1}}_\infty \precsim d \sqrt{\frac{\log p}{n}}.
\end{equation*}
Hence, by Lemma~\ref{LemKronProd} in Appendix~\ref{AppSupporting}, we have
\begin{align*}
\opnorm{\Qhat - \nabla^2 \Loss_\numobs(\Thetastar)}_\infty \precsim \;
d \sqrt{\frac{\log p}{n}},
\end{align*}
implying the bounds
\begin{equation*}
\delta_1 \defn \opnorm{\Qhat_{S^cS} - \left(\nabla^2
  \Loss_\numobs(\Thetastar)\right)_{S^cS}}_\infty \precsim d
\sqrt{\frac{\log p}{n}},
\end{equation*}
and
\begin{equation*}
\opnorm{\Qhat_{SS} - \left(\nabla^2
  \Loss_\numobs(\Thetastar)\right)_{SS}}_\infty \precsim d
\sqrt{\frac{\log p}{n}}.
\end{equation*}
Applying Lemma~\ref{LemSpecInverse} in Appendix~\ref{AppSupporting}
yields
\begin{equation}
\label{EqnPoutine}
\delta_2 \defn \opnorm{\left(\Qhat_{SS}\right)^{-1} - \left(\nabla^2
  \Loss_\numobs(\Thetastar) \right)_{SS}^{-1}}_\infty \precsim d
\sqrt{\frac{\log p}{n}}.
\end{equation}
Next we define the pair $(R, \Xi)$ according to
\begin{equation*}
R \defn \Bigg\|\underbrace{\left\{\Qhat_{S^cS}
  \left(\Qhat_{SS}\right)^{-1} - \left(\nabla^2
  \Loss_\numobs(\Thetastar)\right)_{S^cS} \left(\nabla^2
  \Loss_\numobs(\Thetastar)\right)_{SS}^{-1} \right\}}_\Xi (\nabla
\Loss_\numobs(\Thetastar))_S\Bigg\|_\infty.
\end{equation*}
Note that we have $R \le \opnorm{\Xi}_\infty \cdot \left\|\left(\nabla
\Loss_\numobs(\Thetastar)\right)_S\right\|_{\max} \precsim
\opnorm{\Xi}_\infty \cdot \sqrt{\frac{\log p}{n}}$, and moreover,
\begin{align*}
\opnorm{\Xi}_\infty & \le \opnorm{\left(\Qhat_{S^cS} - \left(\nabla^2
  \Loss_\numobs(\Thetastar)\right)_{S^cS}\right)
  \left(\left(\Qhat_{SS}\right)^{-1} - \left(\nabla^2
  \Loss_\numobs(\Thetastar)\right)_{SS}^{-1}\right)}_\infty \\
& \quad + \opnorm{\left(\Qhat_{S^cS} - \left(\nabla^2
  \Loss_\numobs(\Thetastar)\right)_{S^cS}\right) \left(\nabla^2
  \Loss_\numobs(\Thetastar)\right)_{SS}^{-1}}_\infty \\
& \quad + \opnorm{\left(\nabla^2
  \Loss_\numobs(\Thetastar)\right)_{S^cS}
  \left(\left(\Qhat_{SS}\right)^{-1} - \left(\nabla^2
  \Loss_\numobs(\Thetastar)\right)_{SS}^{-1}\right)}_\infty \\
& \le \delta_1 \delta_2 + \delta_1 \cdot
\opnorm{\left(\Qhat_{SS}\right)^{-1}}_\infty + \delta_2 \cdot
\opnorm{\left(\nabla^2 \Loss_\numobs(\Thetastar)\right)_{S^cS}}_\infty
\\
& \precsim d \sqrt{\frac{\log p}{n}}.
\end{align*}
Combined with our earlier upper bound on $R$, we conclude that $R
\precsim \sqrt{\frac{\log p}{n}}$, under the scaling $n \succsim d^2
\log p$.

Next, observe that
\begin{align*}
\Bigg\|\left(\nabla^2 \Loss_\numobs(\Thetastar)\right)_{S^cS} &
\left(\nabla^2 \Loss_\numobs(\Thetastar)\right)_{SS}^{-1} \left(\nabla
\Loss_\numobs(\Thetastar)\right)_S\Bigg\|_\infty \\
& \le \opnorm{\left(\nabla^2 \Loss_\numobs(\Thetastar)\right)_{S^cS}
  \left(\nabla^2 \Loss_\numobs(\Thetastar)\right)_{SS}^{-1}}_\infty
\cdot \left\|\left(\nabla
\Loss_\numobs(\Thetastar)\right)_S\right\|_{\max} \\
& \le \opnorm{\left(\nabla^2
  \Loss_\numobs(\Thetastar)\right)_{S^cS}}_\infty \cdot
\opnorm{\left(\nabla^2
  \Loss_\numobs(\Thetastar)\right)_{SS}^{-1}}_\infty \cdot
\left\|\left(\nabla \Loss_\numobs(\Thetastar)\right)_S\right\|_{\max}
\\
& \precsim \sqrt{\frac{\log p}{n}}.
\end{align*}
Hence, the primal-dual witness technique succeeds. Note that by
inequality~\eqref{EqnPoutine}, we also have
\begin{equation*}
\opnorm{\left(\Qhat_{SS}\right)^{-1}}_\infty \le \opnorm{\left(\Qhat_{SS}\right)^{-1} - \left(\nabla^2 \Loss_n(\Thetastar)\right)_{SS}^{-1}}_\infty + \opnorm{\left(\nabla^2 \Loss_n(\Thetastar)\right)_{SS}^{-1}}_\infty \le 2c_\infty.
\end{equation*}
By Theorem~\ref{ThmEllInf}, we conclude that $\Thetahat$ is the unique
global minimum of the program~\eqref{EqnGlassoLoss} with the desired
properties.

Finally, let us prove the claimed bounds on the Frobenius and spectral
norms.  Note that $\opnorm{\Thetahat - \Thetastar}_2 \le
\opnorm{\Thetahat - \Thetastar}_F \le \sqrt{s} \; \|\Thetahat -
\Thetastar\|_{\max}$.  Furthermore, since $\Thetahat - \Thetastar$ is
a symmetric matrix, we also have
\begin{equation*}
\opnorm{\Thetahat - \Thetastar}_2 \le \opnorm{\Thetahat -
  \Thetastar}_\infty \le d \; \|\Thetahat - \Thetastar\|_{\max}.
\end{equation*}
The bounds then follow from our earlier bound on the elementwise
$\ell_\infty$-norm.


\section{Proofs for Section~\ref{SecSims}}

In this section, we provide details of proofs for the results in
Section~\ref{SecSims}.

\subsection{Proof of Proposition~\ref{PropOpt}}
\label{AppPropOpt}

This proof is a fairly straightforward modification of the proof of
Theorem 3 in Loh and Wainwright~\cite{LohWai13}, so we provide only a
sketch of how the argument deviates from the proof supplied there.

The only substantial difference between the two settings is that
$\Lossbar_n(\beta) = \Loss_\numobs(\beta) - q_\lambda(\beta)$ rather
than $\Lossbar_n(\beta) = \Loss_\numobs(\beta) - \frac{\mu}{2}
\|\beta\|_2^2$, and the side constraint is slightly
tweaked. Nonetheless, we have $\|\beta\|_1 \le R$, for all $\beta$ in
the feasible region, which is the only property of the side constraint
needed for the proofs of Loh and
Wainwright~\cite{LohWai13}. Concerning the particular form of
$\Lossbar_n$, we simply need to establish for the RSC relations that
\begin{equation}
\label{EqnTurmeric}
\widebar{\scriptT}(\beta_1, \beta_2) \ge \scriptT(\beta_1, \beta_2) -
\frac{\mu}{2} \|\beta_1 - \beta_2\|_2^2
\end{equation}
still holds, where $\widebar{\scriptT}(\beta_1, \beta_2) \defn
\Lossbar_n(\beta_1) - \Lossbar_n(\beta_2) - \inprod{\nabla
  \Lossbar_n(\beta_2)}{\beta_1 - \beta_2}$.  Note that
\begin{equation}
\label{EqnCurry}
\widebar{\scriptT}(\beta_1, \beta_2) = \scriptT(\beta_1, \beta_2) -
q_\lambda(\beta_1) + q_\lambda(\beta_2) + \inprod{\nabla
  q_\lambda(\beta_2)}{\beta_1 - \beta_2}.
\end{equation}
By Lemma~\ref{LemSmiley}(b) in Appendix~\ref{AppAmenable}, we have
\begin{equation*}
q_\lambda(\beta_1) - \frac{\mu}{2} \|\beta_1\|_2^2 -
q_\lambda(\beta_2) + \frac{\mu}{2} \|\beta_2\|_2^2 - \inprod{\nabla
  q_\lambda(\beta_2) - \mu \beta_2}{\beta_1 - \beta_2} \le 0,
\end{equation*}
implying that
\begin{equation}
\label{EqnCumin}
q_\lambda(\beta_1) - q_\lambda(\beta_2) - \inprod{\nabla
  q_\lambda(\beta_2)}{\beta_1 - \beta_2} \le \frac{\mu}{2} \|\beta_1 -
\beta_2\|_2^2.
\end{equation}
Combining inequalities~\eqref{EqnCurry} and~\eqref{EqnCumin} yields
the required inequality~\eqref{EqnTurmeric}. Finally, note that by our
assumption and equation~\eqref{EqnCurry}, we have
$\widebar{\scriptT}(\beta_1, \beta_2) \le \scriptT(\beta_1, \beta_2)$,
so the RSM condition holds for $\Lossbar_n$ with the same parameter
$\alpha_3$.  The remaining arguments proceed as before.


\subsection{Proof of Corollary~\ref{CorOpt}}
\label{AppCorOpt}

We set $\delta \asymp \sqrt{\frac{k \log p}{n}} \cdot \|\betahat -
\betastar\|_2$ in Proposition~\ref{PropOpt}. Then
\begin{equation*}
\|\beta^t - \betahat\|_\infty \le \|\beta^t - \betahat\|_2 \precsim
\frac{1}{(\alpha - \mu)^{1/2}}\sqrt{\frac{k \log p}{n}} \cdot
\|\betahat - \betastar\|_2 \leq \frac{1}{(\alpha - \mu)^{1/2}} \cdot
\frac{k \log p}{n},
\end{equation*}
where the last inequality follows by the assumption of statistical
consistency for $\betahat$. Under the scaling $n \succsim k^2 \log p$,
the desired result follows.


\subsection{Proof of Lemma~\ref{LemGam1}}
\label{AppLemGam1}

The computation of the incoherence parameter is straightforward. For
the spectral properties, note that $v^T \Gamma v = 1 + 2\theta v_{k+1}
\sum_{j=1}^k v_j$ for any vector $v \in \real^p$.  Consequently, we
have
\begin{align*}
\lambda_{\min}(\Gamma) & = 1 + 2\theta \cdot \min_{\|v\|_2 = 1}
\left\{v_{k+1} \sum_{j=1}^k v_j \right\}, \quad \text{and} \quad
\lambda_{\max}(\Gamma) & = 1 + 2\theta \cdot \max_{\|v\|_2 = 1}
\left\{v_{k+1} \sum_{j=1}^k \right\}.
\end{align*}
We may write 
\begin{equation*}
\max_{\|v\|_2 = 1} \{v_{k+1} \sum_{j=1}^k v_j \} = \max_{0 \le \alpha
  \le 1} \Big\{\alpha \cdot \max_{\stackrel{w \in \real^k,}{\|w\|_2^2
    \le 1 - \alpha^2}} \|w\|_1 \Big\} \le \max_{0 \le \alpha \le 1}
\alpha \cdot \sqrt{k} \sqrt{1-\alpha^2},
\end{equation*}
where we have used the fact that $\|w\|_1 \le \sqrt{k} \|w\|_2$. It is
easy to see that the final expression is maximized when $\alpha =
\frac{1}{\sqrt{2}}$, so $\lambda_{\max}(\Gamma) \le \theta
\sqrt{k}$. Equality is achieved for the vector \mbox{$(v_1, \dots,
  v_k, v_{k+1}) = \left(\frac{1}{\sqrt{2k}}, \cdots,
  \frac{1}{\sqrt{2k}}, \frac{1}{\sqrt{2}}\right)$.}  The
lower-eigenvalue bound follows by a similar argument, with equality
achieved when $(v_1, \dots, v_k, v_{k+1}) = \left(\frac{1}{\sqrt{2k}},
\cdots, \frac{1}{\sqrt{2k}}, \frac{-1}{\sqrt{2}}\right)$.


\section{Some useful auxiliary results}
\label{AppSupporting}

Finally, we provide some useful auxiliary results, which we employ in
the proofs of our main theorems. \\

\subsection{Properties of amenable regularizers}
\label{AppAmenable}

The following lemma is useful in various parts of our analysis. Part
(a) is based on Lemma 4 of Loh and Wainwright~\cite{LohWai13}.
\begin{lem*}
\label{LemSmiley}
Consider a $\mu$-amenable regularizer $\rho_\lambda$.  Then we have
\begin{enumerate}
\item[(a)] $|\rho'_\lambda(t)| \le \lambda$, for all $t \neq 0$, and
\item[(b)] The function $q_\lambda(t) - \frac{\mu}{2} t^2$ is concave
  and everywhere differentiable.
\end{enumerate}
\end{lem*}

\begin{proof}
(a) Consider $0 < t \le s$. By condition (iii), we have
  $\frac{\rho_\lambda(s) - \rho_\lambda(t)}{s-t} \leq
  \frac{\rho_\lambda(t)}{t}$.  By conditions (iii) and (iv), we also
  have $\frac{\rho_\lambda(t)}{t} \le \lim_{u \rightarrow 0^+}
  \frac{\rho_\lambda(u)}{u} = \rho_\lambda'(0) \le \lambda$.  Putting
  together the pieces, we find that \mbox{$\rho'_\lambda(t) = \lim_{s
      \rightarrow t} \frac{\rho_\lambda(s) - \rho_\lambda(t)}{s-t} \le
    \lambda$.}  A similar argument holds when $t < 0$.\\

\noindent (b) If $t > 0$, we can write $q_\lambda(t) - \frac{\mu}{2}
t^2 = \lambda t - \rho_\lambda(t) - \frac{\mu}{2} t^2$, which is
concave since $\rho_\lambda(t) + \frac{\mu}{2}$ is convex, by
condition (v). Similarly, $q_\lambda(t)$ is concave for $t < 0$. At $t
= 0$, we have $q'_\lambda(0) = 0$, by condition (vi). Then
$q_\lambda(t) - \frac{\mu}{2} t^2$ is a differentiable function with
monotonically decreasing derivative, implying concavity of the
function.
\end{proof}


\subsection{Bounds on $\ell_2$-errors of stationary points}
\label{AppLemEll2}

The following result is taken from Loh and Wainwright~\cite{LohWai13}.
It applies to any stationary point of the program $\min_{\|\beta\|_1
  \le R, \; \beta \in \Omega} \left\{\Loss_\numobs(\beta) +
\rho_\lambda(\beta)\right\}$, meaning a vector $\betatil$ such that
$\inprod{\nabla \Loss_\numobs(\betatil) + \nabla
  \rho_\lambda(\betatil)}{\beta - \betatil} \ge 0$ for all feasible
$\beta \in \real^p$.

\begin{lem*} [Theorem 1 of Loh and Wainwright~\cite{LohWai13}]
\label{LemEll2Bounds}
Suppose $\Loss_\numobs$ satisfies the RSC condition~\eqref{EqnRSC} and
$\rho_\lambda$ is $\mu$-amenable, with $\mu < 2 \alpha_1$. Suppose the
sample size satisfies the scaling $n \ge \frac{16R^2 \max(\tau_1^2,
  \tau_2^2)}{\alpha_2^2} \log p$, and $2 \cdot \max\left\{\|\nabla
\Loss_\numobs(\betastar)\|_\infty, \; \alpha_2 \sqrt{\frac{\log
    p}{n}}\right\} \le \lambda \le \frac{\alpha_2}{6R}$.  Then any
stationary vector $\betatil$ satisfies the bounds
\begin{equation*}
\|\betatil - \betastar\|_2 \le \frac{7 \lambda \sqrt{k}}{4\alpha_1 -
  2\mu}, \qquad \text{and} \qquad \|\betatil - \betastar\|_1 \le
\frac{28\lambda k}{2\alpha_1 - \mu}.
\end{equation*}
\end{lem*}

\vspace*{.1in}


\subsection{Sufficient conditions for local minima}

The following lemma is a minor extension of results from Fletcher and
Watson~\cite{FleWat80}.  It applies to functions
$f \in C^2$ and $g \in C^1$, such that $g(x) - \frac{\kappa}{2}
\|x\|_2^2$ is concave, for some $\kappa \ge 0$.
\begin{lem*}
\label{LemSecondOrder}
Suppose $x^*$ is feasible for the program
\begin{equation}
\label{EqnNorm}
\min_x \big\{\underbrace{f(x) - g(x)}_{h(x)} + \lambda\|x\|_1\big\},
\qquad \mbox{s.t.} \quad \|x\|_1 \le R,
\end{equation}
and there exist $v^*, w^* \in \partial \|x^*\|_1$, $\mu^* \ge 0$
such that
\begin{subequations}
\begin{align}
\label{EqnCompSlack}
\mu^*(R - \|x^*\|_1) & = 0, \qquad \\
\label{EqnZeroGrad}
\nabla h(x^*) + \lambda v^* + \mu^* w^* & = 0, \qquad \mbox{and} \\
\label{EqnStrictConvex}
s^T \left(\nabla^2 f(x^*)\right) s & > \kappa, \qquad \mbox{for all $s
  \in G^*$},
\end{align}
\end{subequations}
where 
\begin{align*}
G^* \defn & \{s: \|s\|_1 = 1; \quad \sup_{w \in \partial \|x^*\|_1}
s^T w \le 0 \quad \mbox{if} \quad \|x^*\|_1 = R; \\
& \sup_{v \in \partial \|x^*\|_1} s^T (\nabla h(x^*) + \lambda v) = 0;
\quad \mu^* \sup_{w \in \partial \|x^*\|_1} s^T w = 0\}.
\end{align*}
Then $x^*$ is an isolated local minimum of the
program~\eqref{EqnNorm}.
\end{lem*}

\begin{proof}
	
The proof of this lemma essentially follows the proof of Theorem 3 of
Fletcher and Watson~\cite{FleWat80}, except it allows for a composite
function $h = f-g$ in the objective that is not in $C^2$. Nonetheless,
we include a full proof for clarty and completeness.

Suppose for the sake of contradiction that $x^*$ is not an isolated
local minimum. Then there exists a sequence of feasible points
$\{x^{(k)}\} \rightarrow x^*$ with $\phi(x^{(k)}) \le \phi(x^*)$,
where
\begin{equation*}
\phi(x) \defn h(x) + \lambda \|x\|_1.
\end{equation*}
Let $s^{(k)} \defn \frac{x^{(k)} - x^*}{\|x^{(k)} - x\|_2}$; then
$\{s^{(k)}\}$ is a set of feasible directions. Since $\{s^{(k)}\}
\subseteq \ball_2(1)$, the set must possess a point of accumulation $s
\in \ball_2(1)$, and we may extract a convergent
subsequence. Relabeling the points as necessary, we assume that
$\{s^{(k)}\} \rightarrow s$. We show that $s \in G^*$.

Since the feasible region is closed, $s$ is also a feasible direction
at $x^*$. In particular, if $\|x^*\|_1 = R$, we must have
\begin{equation}
\label{EqnCoriander}
\sup_{w \in \partial \|x^*\|_1} s^T w \le 0.
\end{equation}
Together with equation~\eqref{EqnCompSlack}, this implies that
\begin{equation}
\label{EqnRosemary}
\mu^* \sup_{w \in \partial \|x^*\|_1} s^T w \le 0.
\end{equation}
Note that by equation~\eqref{EqnZeroGrad}, we have
\begin{equation}
\label{EqnCilantro}
s^T(\nabla h(x^*) + \lambda v^*) = - \mu^* s^T w^* \ge 0,
\end{equation}
where the inequality follows from the fact that if $\mu^* \neq 0$, we
have $\|x^*\|_1 = R$, by equation~\eqref{EqnCompSlack}, Hence, $s^T w^* \le
0$, by inequality~\eqref{EqnCoriander}.

By the definition of the subgradient, we have
\begin{equation*}
\|x^{(k)} - x^*\|_2 \cdot s^{(k)T} v \le \bigg \|x^* + \|x^{(k)} -
x^*\|_2 s^{(k)} \bigg \|_1 - \|x^*\|_1 = \|x^{(k)}\|_1 - \|x^*\|_1,
\end{equation*}
for all $v \in \partial \|x^*\|_1$ and all $k$. In particular, $s^T v
= \lim \limits_{k \rightarrow \infty} s^{(k)T} v \le \lim_{k
  \rightarrow \infty} \frac{\|x^{(k)}\|_1 - \|x^*\|_1}{\|x^{(k)} -
  x^*\|_2}$, so
\begin{equation}
\label{EqnParsley}
\sup_{v \in \partial \|x^*\|_1} s^T v \le \lim_{k \rightarrow \infty}
\frac{\|x^{(k)}\|_1 - \|x^*\|_1}{\|x^{(k)} - x^*\|_2}.
\end{equation}
Furthermore,
\begin{equation}
\label{EqnSage}
s^T \nabla h(x^*) = \lim_{k \rightarrow \infty} s^{(k)T} \nabla h(x^*)
= \lim_{k \rightarrow \infty} \frac{\inprod{x^{(k)} - x^*}{\nabla
    h(x^*)}}{\|x^{(k)} - x^*\|_2} = \lim_{k \rightarrow \infty}
\frac{h(x^{(k)}) - h(x^*)}{\|x^{(k)} - x^*\|_2},
\end{equation}
since $x^{(k)} \rightarrow x^*$ and $f \in C^1$. Combining
inequality~\eqref{EqnParsley} with equation~\eqref{EqnSage}, we
conclude that
\begin{equation*}
\sup_{v \in \partial \|x^*\|_1} s^T(\nabla h(x^*) + \lambda v) \le
\lim_{k \rightarrow \infty} \frac{\phi(x^{(k)}) - \phi(x^*)}{\|x^{(k)}
  - x^*\|_2} \le 0,
\end{equation*}
where the second inequality follows from the assumption $\phi(x^{(k)})
\le \phi(x^*)$. Hence, using inequality~\eqref{EqnCilantro}, we
conclude that
\begin{equation}
\label{EqnDill}
\sup_{v \in \partial \|x^*\|_1} s^T (\nabla h(x^*) + \lambda v) = s^T
(\nabla h(x^*) + \lambda v^*) = 0,
\end{equation}
and by equation~\eqref{EqnZeroGrad}, we have $\mu^* s^T w^* = 0$, as
well. Together with inequality~\eqref{EqnRosemary}, this implies that
\begin{equation}
\label{EqnThyme}
\mu^* \sup_{w \in \partial \|x^*\|_1} s^T w = 0.
\end{equation}
Combining inequalities~\eqref{EqnCoriander}, \eqref{EqnDill},
and~\eqref{EqnThyme}, we therefore conclude that $s \in G^*$.

Now note that
\begin{equation*}
\phi(x^{(k)}) = h(x^{(k)}) + \lambda\|x^{(k)}\|_1 \ge h(x^{(k)}) +
\lambda x^{(k)T} v^* + \mu^*(x^{(k)T} w^* - R),
\end{equation*}
using the fact that $\mu^* \ge 0$ and $x^{(k)T} w^* \le \|x^{(k)}\|_1
\le R$. Noting that $\phi(x^*) = h(x^*) + \lambda x^{*T} v^* +
\mu^*(x^{*T} w^* - R)$, we have
\begin{align}
\label{EqnCraisin}
0 & \ge \phi(x^{(k)}) - \phi(x^*) \notag \\
& \ge h(x^{(k)}) - h(x^*) + \inprod{\lambda v^* + \mu^* w^*}{x^{(k)} -
  x^*} \notag \\
& = h(x^{(k)}) - h(x^*) - \inprod{\nabla h(x^*)}{x^{(k)} - x^*} \notag
\\
& = \left(f(x^{(k)}) - f(x^*) - \inprod{\nabla f(x^*)}{x^{(k)} -
  x^*}\right) - \left(g(x^{(k)}) - g(x^*) - \inprod{\nabla
  g(x^*)}{x^{(k)} - x^*}\right).
\end{align}
By the concavity of $g(x) - \frac{\kappa}{2} \|x\|_2^2$, we have
$g(x^{(k)}) - g(x^*) - \inprod{\nabla g(x^*)}{x^{(k)} - x^*} \le
\frac{\kappa}{2} \|x^{(k)} - x^*\|_2^2$.  Combining with
inequality~\eqref{EqnCraisin} and using Taylor's theorem, we then have
\begin{equation*}
0 \ge \frac{1}{2} (x^{(k)} - x^*)^T \nabla^2 f(x^*) (x^{(k)} - x^*) -
\frac{\kappa}{2} \|x^{(k)} - x^*\|_2^2 + o(\|x^{(k)} - x^*\|_2^2).
\end{equation*}
Dividing through by $\|x^{(k)} - x^*\|_2^2$ and taking a limit as $k
\rightarrow \infty$, we obtain the bound $\frac{1}{2} s^T
\left(\nabla^2 f(x^*)\right) s - \frac{\kappa}{2} \leq 0$,
contradicting the assumption~\eqref{EqnStrictConvex}. Hence, we
conclude that $x^*$ must indeed be an isolated local minimum.

\end{proof}


\subsection{Some matrix-theoretic lemmas}

Here, we collect some useful lemmas on matrices and Kronecker products.

\begin{lem*}
\label{LemSpecInverse}
Let $A, B \in \real^{p \times p}$ be invertible. For any matrix norm
$\opnorm{\cdot}$, we have
\begin{equation}
\label{EqnPeach} \
\opnorm{A^{-1} - B^{-1}} \le \frac{\opnorm{A^{-1}}^2 \opnorm{A-B}}{1 -
  \opnorm{A^{-1}} \opnorm{A-B}}.
	\end{equation}
In particular, if $\opnorm{A^{-1}} \opnorm{A-B} \leq 1/2$, then
$\opnorm{A^{-1} - B^{-1}} = \order\left(\opnorm{A^{-1}}^2
\opnorm{A-B}\right)$.
\end{lem*}

\begin{proof}
We use the matrix expansion~\eqref{EqnMatInvExpand}, where $\Delta =
B-A$. By the triangle inequality and multiplicativity of the matrix
norm, we then have
\begin{align*}
\opnorm{A^{-1} - B^{-1}} & \le \sum_{\ell=1}^\infty \opnorm{A^{-1}}^2
\opnorm{A-B}^\ell \; = \; \frac{\opnorm{A^{-1}}^2 \opnorm{A-B}}{1 -
  \opnorm{A^{-1}} \opnorm{A-B}},
\end{align*}
as claimed.
\end{proof}

\begin{lem*}
\label{LemKronInfty}
For any matrices $A$ and $B$, we have $\opnorm{A \otimes B}_\infty =
\opnorm{A}_\infty \cdot \opnorm{B}_\infty$.
\end{lem*}

\begin{proof}
Using the definition of the Kronecker product, we have
\begin{align*}
\opnorm{A \otimes B}_\infty = \max_{s,t} \sum_{u,v} |A_{su} B_{tv}| \;
= \max_{s,t} \sum_{u,v} |A_{su}| \; |B_{tv}| & = \max_{s,t}
\left(\sum_u |A_{su}|\right) \; \left(\sum_v |B_{tv}|\right) \\
& = \opnorm{A}_\infty \cdot \opnorm{B}_\infty,
\end{align*}
as claimed.
\end{proof}

\begin{lem*}
\label{LemKronProd}
Let $A$ and $B$ be matrices of the same dimension. Then
\begin{equation*}
  \opnorm{A \otimes A - B \otimes B}_\infty \le \opnorm{A -
    B}_\infty^2 + 2 \min\{\opnorm{A}_\infty, \opnorm{B}_\infty\}
  \cdot \opnorm{A-B}_\infty.
\end{equation*}
\end{lem*}

\begin{proof}
Note that $A \otimes A - B \otimes B = (A-B) \otimes (A-B) + B \otimes
(A-B) + (A-B) \otimes B$.  By the triangle inequality and
Lemma~\ref{LemKronInfty}, we then have
\begin{equation*}
\opnorm{A \otimes A - B \otimes B}_\infty \le \opnorm{A-B}_\infty^2 +
2 \opnorm{B}_\infty \cdot \opnorm{A-B}_\infty.
\end{equation*}
By symmetry, the same bound holds with the roles of $A$ and $B$
reversed, from which the claim follows.
\end{proof}


\bibliography{new_refs}

\end{document}